\documentclass[a4paper,english,12pt]{article}% or something else

\usepackage{geometry}
\usepackage{pdfpages}
\usepackage[T1]{fontenc}
\usepackage[utf8]{inputenc}
\usepackage{babel}
\usepackage{amsmath}
\usepackage{amssymb}
\usepackage{amsthm}
\usepackage{textcomp}
\usepackage{enumitem}
\usepackage{dsfont}
\usepackage{mathrsfs}
\usepackage{cite}
\usepackage{color}
\usepackage[colorlinks=true,urlcolor=blue,linkcolor=blue]{hyperref}
\usepackage{subcaption}

\usepackage{tikz,pgfplots}
\usetikzlibrary{decorations.markings}

\newcommand{\R}{\mathbb{R}}

\newcommand{\Z}{\mathbb{Z}}
\newcommand{\N}{\mathbb{N}}
\newcommand{\C}{\mathbb{C}}
\newcommand{\D}{\mathbb{D}}

\newcommand{\T}{\mathbb{T}}

\newcommand{\classeC}{\mathcal{C}}

\renewcommand{\Im}{\operatorname{Im}}
\renewcommand{\Re}{\operatorname{Re}}
\newcommand{\un}{\mathds{1}}

\newcommand{\sinc}{\operatorname{sinc}}

\newcommand\dd{\,{\mathrm d}}

\newcommand{\longrightarroww}[2] {\mathop{\longrightarrow}\limits_{#1}^{#2}}
\newcommand{\weaklim}[2] {\mathop{\rightharpoonup}\limits_{#1}^{#2}}

\newcommand{\vertiii}[1]{{\left\vert\kern-0.25ex\left\vert\kern-0.25ex\left\vert #1 
    \right\vert\kern-0.25ex\right\vert\kern-0.25ex\right\vert}}

\newtheorem{mydef}{Definition}[section]
\newtheorem{thm}[mydef]{Theorem}
\newtheorem{lem}[mydef]{Lemma}
\newtheorem{prop}[mydef]{Proposition}
\newtheorem{cor}[mydef]{Corollary}
\newtheorem{rk}[mydef]{Remark}

\title{Lax eigenvalues in the zero-dispersion limit for the Benjamin-Ono equation on the torus.}
\author{Louise Gassot}%\footnote{\noindent{Departement Mathematik und Informatik, Universität Basel, Spiegelgasse 1, 4051 Basel, Schweiz}\\
\date{}
\newcommand{\Addresses}{{
  \bigskip
  \footnotesize
\noindent
  \textsc{Departement
Mathematik und Informatik, Universität Basel, Spiegelgasse 1, 4051 Basel, Schweiz}\par\nopagebreak
  \noindent
  \textit{E-mail address :} \texttt{louise.gassot@normalesup.org}
}}

\pgfplotsset{compat=1.17} 

\begin{document}
\maketitle
\vspace{-20pt}

\abstract{
We consider the zero-dispersion limit for the Benjamin-Ono equation on the torus for bell shaped initial data. Using the approximation by truncated Fourier series, we transform the eigenvalue equation for the Lax operator into a problem in the complex plane. Then, we use the steepest descent method to get asymptotic expansions of the Lax eigenvalues. As a consequence, we determine the weak limit of solutions as the dispersion parameter goes to zero, as long as the initial data is an even bell shaped potential.}

\tableofcontents

%\newpage
%%%%%%%%%%%%%%%%%%%%%%%%%%%%%%%%%%%%%%%%%%%%%%%%%%%%%%%%%%%%%%%%%%%%%%%%%%%%%%%%%%%%%%%%%%%%%%%%%%%%%%%%%%%%%%%%%%%%%%%%%%%%%%%%%%%%%%%%%%%%%%%%%
%\section{Introduction}
%%%%%%%%%%%%%%%%%%%%%%%%%%%%%%%%%%%%%%%%%%%%%%%%%%%%%%%%%%%%%%%%%%%%%%%%%%%%%%%%%%%%%%%%%%%%%%%%%%%%%%%%%%%%%%%%%%%%%%%%%%%%%%%%%%%%%%%%%%%%%%%%%

%\newpage
%%%%%%%%%%%%%%%%%%%%%%%%%%%%%%%%%%%%%%%%%%%%%%%%%%%%%%%%%%%%%%%%%%%%%%%%%%%%%%%%%%%%%%%%%%%%%%%%%%%%%%%%%%%%%%%%%%%%%%%%%%%%%%%%%%%%%%%%%%%%%%%%%
\section{Introduction}
%%%%%%%%%%%%%%%%%%%%%%%%%%%%%%%%%%%%%%%%%%%%%%%%%%%%%%%%%%%%%%%%%%%%%%%%%%%%%%%%%%%%%%%%%%%%%%%%%%%%%%%%%%%%%%%%%%%%%%%%%%%%%%%%%%%%%%%%%%%%%%%%%

%\newpage
%%%%%%%%%%%%%%%%%%%%%%%%%%%%%%%%%%%%%%%%%%%%%%%%%%%%%%%%%%%%%%%%%%%%%%%%%%%%%%%%%%%%%%%%%%%%%%%%%%%%%%%%%%%%%%%%%%%%%%%%%%%%%%%%%%%%%%%%%%%%%%%%%
%\subsection{Main result}
%%%%%%%%%%%%%%%%%%%%%%%%%%%%%%%%%%%%%%%%%%%%%%%%%%%%%%%%%%%%%%%%%%%%%%%%%%%%%%%%%%%%%%%%%%%%%%%%%%%%%%%%%%%%%%%%%%%%%%%%%%%%%%%%%%%%%%%%%%%%%%%%%

In this paper, we focus on the zero-dispersion limit  $\varepsilon\to 0$ for the Benjamin-Ono equation on the torus.  The parameter $\varepsilon$ balances the dispersive and the nonlinear effects in the Benjamin-Ono equation
\begin{equation}\tag{BO-$\varepsilon$}\label{eq:bo}
\partial_t u
	=\partial_x(\varepsilon |\partial_{x}|u-u^2).
\end{equation}
We recall that $|\partial_{x}|$ is the Fourier multiplier $\widehat{|\partial_{x}| u}(n)=|n|\widehat{u}(n)$.

The Benjamin-Ono equation~\cite{Benjamin1967,Ono1977} describes a certain regime of long internal waves in a two-layer fluid of great depth. For $\varepsilon=0$, this equation becomes the inviscid Burgers equation, for which shocks appear in finite time. When $\varepsilon>0$, there is global well-posedness in $H^2$~\cite{Saut1979} (see~\cite{GerardKappelerTopalov2020} for lower regularity results), hence the dispersive term prevents the shock formation. The shock is replaced by a dispersive shock, which manifests as a train of waves, we refer to~\cite{MillerXu2011} for numerical simulations on the real line.

\subsection{Zero-dispersion limit for bell shaped initial data}

%\begin{figure}
%\begin{center}
%\includegraphics[scale=0.3]{multivalued.jpg}
%\end{center}
%\caption{Multivalued solution of the Burgers equation obtained by the method of characteristics, with initial data $u_0(x)=-\beta\cos(x)$}\label{fig:multivalued}
%\end{figure}

The goal of this paper is to describe the weak limit of solutions as $\varepsilon\to 0$. We consider only {\it bell shaped} initial data $u_0(x)\in\classeC^3(\T)$ as defined below.
\begin{mydef}[Bell shaped initial data]\label{def:bell shaped}
We say that $u_0\in\classeC^3(\T)$ is a bell shaped initial data if the following holds:
\begin{enumerate}
\item $u_0$ is real valued with zero mean;
\item there exist $x_{\min},x_{\max}\in (0,2\pi)$ such that $u_0'>0$ on $(x_{\min},x_{\max})$ and $u_0'<0$ on $(x_{\max},x_{\min}+2\pi)$;
\item $x_{\min}=0$;
\item there are exactly two inflection points $\xi_-\in(0,x_{\max})$ and $\xi_+\in(x_{\max},2\pi)$ such that $u_0''(\xi_{\pm})=0$, and the inflection points are simple $u_0'''(\xi_{\pm})\neq 0$.
\end{enumerate}
When $u_0\in\classeC^{0}(\T)$ only satisfies 1.\@ and 2., we say that $u_0$ is weakly bell shaped.
\end{mydef}

Note that conditions 1.\@ and 3.\@ are not restrictive because the real-valued property and the mean of the solution $\int_{\T}u(t,x)\dd x$ are preserved by the flow, and because the Benjamin-Ono equation is invariant by spatial translation. The other two conditions are more technical and aim at simplifying the calculation.

Given a bell shaped initial data, let $u^{\varepsilon}$ be  the solution  to~\eqref{eq:bo} with parameter $\varepsilon$. We prove that as $\varepsilon\to 0$, the weak limit of solutions has an explicit form in terms of the multivalued solution to Burgers' equation obtained by the method of characteristics. More precisely, we say that every point $u^B(t,x)$ is an image of the multivalued solution to Burgers' equation at $(t,x)$ as soon as it solves the implicit equation
\[
u^B(t,x)=u_0(x-2u^B(t,x)t).
\]
Given $t$ and $x$, there may be several branches of solutions that are denoted $u_0^B(t,x)<\dots<u_{2P(t,x)}^B(t,x)$, see Figure~1 in~\cite{bo_zero}. We define the signed sum of branches as
\begin{equation*}%\label{eq:u_alt}
u^B_{alt}(t,x):=\sum_{n=0}^{2P(t,x)}(-1)^nu_n^B(t,x).
\end{equation*}

In~\cite{bo_zero}, we established that given a bell shaped initial data $u_0$, there exists a family of approximate initial data $u_0^{\varepsilon}$ such that the solution to~\eqref{eq:bo} with initial data $u_0^{\varepsilon}$ is weakly convergent to $u^{B}_{alt}$ in $L^2(\T)$, uniformly on compact time intervals, and such that $u_0^{\varepsilon}\to u_0$ in $L^2(\T)$. In this paper, we prove that we can replace the approximate initial data $u_0^{\varepsilon}$ by $u_0$ itself when $u_0$ is bell shaped.
%a sufficient decay on the Fourier coefficients of the initial data holds. The decay assumption is that for some universal constant $C$, for every $n$,
%\begin{equation}\label{eq:decay_fourier}
%|\widehat{u_0}(n)|\leq Cn^{-C n^2}.
%\end{equation}
%This corresponds to a Gevrey-type regularity of class $\frac{1}{2^+}$. 
Our main result is the following.

\begin{thm}[Zero-dispersion limit]\label{thm:zero}
Let $u_0\in\classeC^3(\T)$ be an even bell shaped initial data. Then uniformly on compact time intervals, the solution $u^{\varepsilon}$ to the Benjamin-Ono equation~\eqref{eq:bo} with parameter $\varepsilon$ and initial data $u_0$ converges weakly to $u_{alt}^B$ in~$L^2(\T)$ as $\varepsilon\to 0$: for every $T>0$, there holds
\[
u^{\varepsilon}\rightharpoonup u_{alt}^B \quad \text{ in } \quad \mathcal{C}([0,T],L^2_{weak}(\T)). 
\]
\end{thm}

In the paper, we mostly focus on the proof of the theorem when $u_0$ is an even bell shaped trigonometric polynomial. Then we rely on an explicit formula of the solution in terms of the Lax operator from Gérard~\cite{Gerard2022explicit} to extend this result to more general initial data.

\paragraph{Zero-dispersion limit for the Benjamin-Ono equation on the line}
Regarding the Benjamin-Ono equation on the line,   %A first approach to determine the right scattering data using Witham's theory can be found in~\cite{Matsuno1998}.
the first formal approaches to the zero-dispersion limit problem for the Benjamin-Ono equation have been employed to justify the Whitham modulation theory~\cite{Whitham}  in~\cite{Matsuno1998,Matsuno1998small} and in~\cite{Jorge1999}, and the case of more general nonlocal Benjamin-Ono equations which are not necessarily integrable is tackled in the paper~\cite{ElNguyenSmyth2018}. A similar result as Theorem 1.2 from~\cite{bo_zero} (using an approximate initial data $u_0^{\varepsilon}$) was established by Miller and Xu in~\cite{MillerXu2011}. An extension to the Benjamin-Ono hierarchy has been investigated in~\cite{MillerXu2011hierarchy}.

In order to remove the approximate initial data $u_0^{\varepsilon}$, a better understanding of the direct scattering transform in the zero-dispersion limit is needed. The scattering data was studied by Miller and Wetzel in~\cite{MillerWetzel2016rational} on the line when the initial data is a rational potential of the form
\[
u_0(x)=\sum_{n=1}^N\frac{c_n}{x-z_n}+c.c.
\]
where $c_n\in\C^*$ and $\Im(z_n)>0$ for $1\leq n\leq N$, moreover, the poles $z_n$ have distinct real parts, and $\sum_{n=1}^Nc_n+\overline{c_n}=0$. In particular, every $N$-soliton (see Definition 1.1 in~\cite{Sun2020}) is rational. In the zero-dispersion limit, an asymptotic expansion of the scattering data for Klaus-Shaw initial data (a variation of the bell-shaped assumption adapted to the line) has been established in~\cite{MillerWetzel2016}. As a consequence, the choice of the scattering data of the approximate initial data $u_0^{\varepsilon}$ in~\cite{MillerXu2011} is justified and it is likely that one could replace the approximate initial data $u_0^{\varepsilon}$ by the actual initial data $u_0$ itself when $u_0$ is a rational Klaus-Shaw initial data.

\paragraph{Integrability for the Benjamin-Ono equation}
To generalize the class of initial data $u_0$ for which we do not need to rely on approximate initial data $u_0^{\varepsilon}$, we make use of the integrability properties of the Benjamin-Ono equation.

On the torus, the complete integrability for the Benjamin-Ono equation has been established by Gérard, Kappeler and Topalov~\cite{GerardKappeler2019, GerardKappelerTopalov2020} for general initial data in $H^s(\T)$ as soon as $s>-\frac 12$, leading to global well-posedness in this range of Sobolev exponents. This integrability property could enable us to generalize the study of the zero-dispersion limit from trigonometric polynomials to bell-shaped initial data in our main theorem, only by estimating the error terms in the trigonometric approximation. However, in order to be successful, such an approach would require a huge decay of the Fourier coefficients of the initial data $u_0$ of the form
\begin{equation*}%\label{eq:decay_fourier}
|\widehat{u_0}(n)|\leq Cn^{-C n^2}.
\end{equation*}
This is why we rather rely on the explicit formula in~\cite{Gerard2022explicit} which requires less regularity of the initial data.

%The decay of the Fourier coefficients~\eqref{eq:decay_fourier} is necessary to ensure that the trigonometric approximation is accurate enough when passing to the limit $\varepsilon\to 0$. 
% It would be an interesting open problem to get a better understanding of the complete integrability for the Benjamin-Ono equation on the line in order to consider more general initial data.

On the line, complete integrability of the Benjamin-Ono equation is known when restricted to the $N$-soliton manifold, see the recent breakthrough from Sun in~\cite{Sun2020}. However, its extension to more general potentials is an interesting open problem. As a consequence, if we try to estimate the error term in the $N$-soliton approximation of a general initial data, there is not much hope in removing the approximate initial data $u_0^{\varepsilon}$ in the study of the zero-dispersion limit. 
With these challenges in mind, we believe that progress on the class of initial data for which we do not need approximate initial data will come from the fact that an explicit formula of solutions using the Lax operator also holds for the Benjamin-Ono equation on the line in~\cite{Gerard2022explicit}, and from similar techniques as Section~\ref{sec:bell_shaped} from our paper. 

\paragraph{Comparison with the KdV equation}

Historically, the first rigorous approach for zero-dispersion limit problem dates back to Lax and Levermore~\cite{LaxLevermore} for the Korteveg-de Vries (KdV) equation on the line
\[
\partial_t u-3\partial_x(u^2)+\varepsilon^2\partial_{xxx}u=0,
\]
with initial data $u_0\leq 0$ decaying at spatial infinity. The authors establish the existence of a weak limit up to approximate initial data, where the weak limit is different from $u_{alt}^B$ as in the Benjamin-Ono equation. The case of positive initial data was investigated in~\cite{Venakides1985}.

Many refinements of this result were then established for the KdV equation on the line, starting from a description of the oscillations in the dispersive shock~\cite{Venakides1991}, then strong asymptotics for the modulation equations using the steepest descent method~\cite{DeiftVenakidesZhou1997}. For the KdV equation on the torus, one can mention a description of the dispersive shock in~\cite{Venakides1987} and a justification of the Zabusky-Kruskal experiment for the cosine initial data in~\cite{DengBiondiniTrillo2016}. Finally, on the line, Claeys and Grava studied various asymptotic approximations in a neighborhood of the Whitham zone, in particular the gradient catastrophe in~\cite{ClaeysGrava2009universality}, the solitonic asymptotics in~\cite{ClaeysGrava2010solitonic} and the Painlevé II asymptotics in~\cite{ClaeysGrava2010painleve}.
Concerning the Benjamin-Ono equation, it is expected that the gradient catastrophe can be described via a universal profile at the breaking time~\cite{MasoeroRaimondoAntunes2015}.
We refer to the survey~\cite{Miller2016} and to~\cite{KleinSautBook} for more bibliographical information on zero-dispersion limit problems regarding the KdV equation and in other contexts. 

\subsection{Asymptotic expansion of the Lax eigenvalues}

On the torus, the scattering data is replaced by the Birkhoff coordinates introduced by Gérard and Kappeler~\cite{GerardKappeler2019}.
The goal of this paper is to better understand the Birkhoff coordinates of the initial data in the zero-dispersion limit. More precisely, in~\cite{bo_zero}, Definition 1.5, we showed that a certain distribution of {\it Lax eigenvalues} and {\it phase constants} of the initial data $u_0^{\varepsilon}$ as $\varepsilon\to 0$ would imply that the corresponding solution $u^{\varepsilon}$ is weakly convergent to $u_{alt}^B$. Our goal in this paper is to show that the initial data $u_0$ itself satisfies those conditions, so that we do not need to rely on approximate initial data. 

Let us first recall the definition of Lax eigenvalues and phase constants.
Fix $\varepsilon>0$. Denote by $L^2_+(\T)$ the Hardy space of complex-valued functions in $L^2(\T)$ with only nonnegative Fourier modes.  The Lax operator $L_u(\varepsilon)$ is defined on $L^2_+(\T)$ as: 
\[
L_u(\varepsilon)h
	=-i\varepsilon\partial_x h-\Pi(u h),\quad h\in L^2_+(\T).
\]
The operator $\Pi$ is the Szeg\H{o} projector from $L^2(\T)$ onto $L^2_+(\T)$.
We denote by $(\lambda_n(u;\varepsilon))_{n\geq 0}$  the eigenvalues of the Lax operator, and by $(f_n(u;\varepsilon))_{n\geq 0}$ its eigenfunctions. When $\langle f_n(u;\varepsilon)\mid \un\rangle\neq 0$, we define the phase constants as
\[
\theta_n(u;\varepsilon)=\mathrm{arg}(\langle\un\mid f_n(u;\varepsilon)\rangle).
\]
When $\langle f_n(u;\varepsilon)\mid \un\rangle= 0$, one can for instance use the convention $\theta_n(u;\varepsilon)=0$.

\begin{figure}
\begin{center}
\includegraphics[scale=0.4]{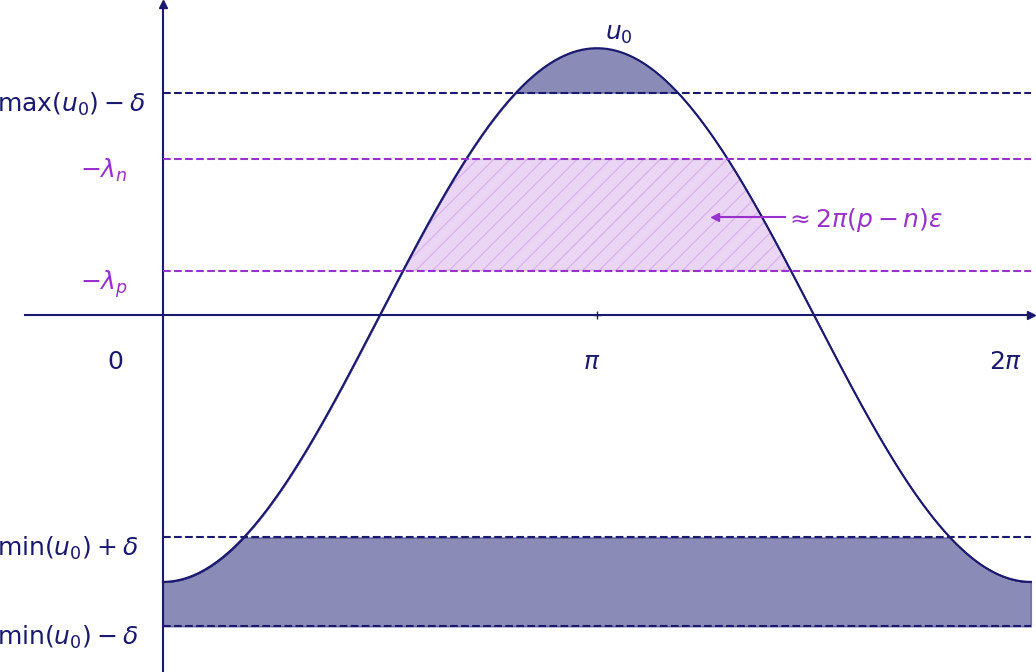}
\end{center}
\caption{Distribution of Lax eigenvalues in the zero-dispersion limit}
\label{fig:bohr}
\end{figure}

The key result of this paper is the following. We prove that the Lax eigenvalues $(\lambda_n(u))_{n\geq 0}$ satisfy a Bohr-Sommerfeld quantization rule, see Figure~\ref{fig:bohr}. Roughly speaking, the area below the curve of $u$ but above the horizontal line $y=-\lambda_n$ should be approximately be equal to $n\varepsilon$.  
These conditions are similar to the assumptions of Definition~1.5 in~\cite{bo_zero}, see also the statement of Corollary 2.2 in~\cite{bo_zero}. They are a refinement of the small-dispersion asymptotics for the Lax eigenvalues already investigated in~\cite{Moll2019-1} in the classical case, see also~\cite{Moll2019-2} in the quantum case. % Our result is the following.

%We set  for $N\geq 1$
%\[
%\delta_N:=\frac{1}{\sqrt{N}}.
%\]
We define  the distribution function as follows: for $\eta\in\R$,
\begin{equation}\label{def:F}
F(\eta):=\frac{1}{2\pi}\operatorname{Leb}(x\in[0,2\pi]\mid u_0(x)\geq\eta).
\end{equation}
We also introduce a function $\psi$ defined later in~\eqref{eq:psi} for general trigonometric polynomials. In the case of an even trigonometric polynomial of degree $N$, $\psi$ takes the simplified form~\eqref{eq:psi_even}
\[
\psi(\eta)=-\frac{1}{4}-\eta+(N+1)F(\eta).
\]
In this setting we get an approximation of Lax eigenvalues when the initial data is a trigonometric polynomial according to Definition~1.5 in~\cite{bo_zero}. We distinguish two areas:
\begin{itemize}
\item the small eigenvalues located in
\[
\Lambda_-(\delta)=[-\max(u_0)+\delta,-\min(u_0)-\delta]
\]
\item the large eigenvalues located in
\[
\Lambda_+(\delta)=[-\min(u_0)+\delta,+\infty).
\]
\end{itemize}
For technical reasons we actually restrict our study so slightly smaller sets 
\[
\widetilde{\Lambda}_{\pm}(\delta)=\Lambda_{\pm}(\delta)\setminus\bigcup_{k=1}^M(y_k-\delta,y_k+\delta).
\]
%\[
%\widetilde{\Lambda}_+(\delta)=\Lambda_+(\delta)\setminus\bigcup_{k=1}^M(y_k-\delta,y_k+\delta).
%\]

\begin{thm}[Lax eigenvalues for trigonometric polynomials]\label{thm:lax_u_trigo}
Let $u$ be a bell-shaped trigonometric polynomial. There exists $C>0$ such that for every $\delta>0$ and $\varepsilon>0$, there exists $C_u(\delta)>0$ such that the following holds. 
\begin{enumerate}
\item (Small eigenvalues) If $\lambda_n+\varepsilon,\lambda_p+\varepsilon\in\widetilde{\Lambda_-}(\delta)$, we have
\begin{equation*}%\label{eq:asymptotics}
\left|\int_{-\lambda_n(u)}^{-\lambda_p(u)}F(\eta)\dd\eta-(n-p+\psi(-\lambda_n)-\psi(-\lambda_p))\varepsilon\right|
	\leq C_u(\delta)\varepsilon\sqrt{\varepsilon}.
\end{equation*}
\item (Large eigenvalues) If $\lambda_n+\varepsilon\in\widetilde{\Lambda_+}(\delta)$, then there holds if $n,p\leq K(\delta)/\varepsilon$
\[
 |\lambda_n-\lambda_p-(n-p)\varepsilon|\leq C_u(\delta)\varepsilon\sqrt{\varepsilon}
\]
and if $n,p\geq K(\delta)/\varepsilon$,
\[
 |\lambda_n-\lambda_p-(n-p)\varepsilon|\leq  \delta.
\]
\item (Two-region eigenvalues) 
If $-\lambda_n-\varepsilon\in\Lambda_-(\delta)$ and $-\lambda_p-\varepsilon\in\Lambda_+(\delta)$, then $|p-n|\geq \frac{\delta}{C\varepsilon}$.
\end{enumerate}
\end{thm}

Compared to Definition~1.5 in~\cite{bo_zero}, there is an extra term $(\psi(-\lambda_n)-\psi(-\lambda_p))\varepsilon$, but we will see later in Section~\ref{sec:zero} that as $\psi$ is Lipschitz and $n,p$ can be restricted to the condition $|n-p|\leq\varepsilon^{-r}$, this extra term can be treated as a remainder term.
Note that the constant $C_u(\delta)$ actually depends both on the choice of bell shaped trigonometric polynomial $u$ and on $\delta$.
The parameter~$\delta$ is introduced for technical purposes. Indeed, we need to remove some bands $[y_k-\delta,y_k+\delta]$ which correspond to approaching a hidden complex stationary point of the phase with double multiplicity in order to get uniform bounds on the remainder terms. Similarly, we need to remove the bands $[-\min(u)-\delta,-\min(u)+\delta]$ and $[-\max(u)-\delta,-\max(u)]$ in order to have uniform bounds on the remainder terms when applying the stationary phase lemma when the stationary points are the antecedents of $u$ by $\lambda$.

As a corollary, we get the following approximation for general bell-shaped functions  according to Corollary~2.2 in~\cite{bo_zero}.
\begin{cor}[Lax eigenvalues in the zero-dispersion limit]\label{thm:lax_u}
Let $u\in\classeC^3(\T)$ be bell-shaped. There exists $C>0$ such that for every $N\geq 0$ and $\delta>0$, the following holds. There exists $C_N(\delta)>0$ and $\varepsilon_0(\delta)>0$ such that if and $0<\varepsilon<\varepsilon_0(\delta)$, then the following holds.
\begin{enumerate}
\item For every $n\geq 0$, we have
\begin{equation*}%\label{eq:asymptotics}
\left|\int_{-\lambda_n(u)}^{\max(u)}F(\eta)\dd\eta-n\varepsilon\right|
	\leq \frac{C}{N\sqrt{N}}+C_N(\delta)\varepsilon\sqrt{\varepsilon}+C\delta.
\end{equation*}
\item (Large eigenvalues) If $\lambda_n+\varepsilon\in\Lambda_+(\delta)=[-\min(u_0)+\delta,+\infty)$, then there holds
\[
\sum_{k= n+1}^{\infty}\gamma_k\leq \frac{C}{N\sqrt{N}}+ C_N(\delta)\varepsilon\sqrt{\varepsilon}+C\delta.
\]
\item (Two-region eigenvalues) 
If $\lambda_n+\varepsilon\in\Lambda_-(\delta)=[-\max(u_0)+\delta,-\min(u_0)-\delta]$ and $\lambda_p+\varepsilon\in\Lambda_+(\delta)=[-\min(u)+\delta,+\infty)$, then $|p-n|\geq \frac{\delta}{C\varepsilon}$.
\end{enumerate}
\end{cor}
In the statement, the constant $C_N(\delta)$ depends on the degree $N$ of the trigonometric approximation and on $\delta$, but also on the choice of the bell shaped initial data itself. Hence this statement is not uniform with respect to the choice of initial data. Moreover, the upper bound is not sufficient in general to prove Theorem~\ref{thm:zero}. This is why we rather rely on the inversion formula from~\cite{Gerard2022explicit}.

We also get some information of the phase constants of even trigonometric polynomials.
\begin{thm}[Phase constants in the zero-dispersion limit]\label{thm:phase_zero}
Assume that $u_0$ is a trigonometric polynomial, bell shaped, and  even. Then the differences of two consecutive phase constants are multiples of $\pi$: for every $n$,
\[
e^{i(\theta_{n+1}(u_0;\varepsilon)-\theta_n(u_0;\varepsilon))}=\pm 1.
\]
Moreover, let
\[
J(u_0;\varepsilon)=\left\{n\geq 1\mid \lambda_n+\varepsilon\in\Lambda_-(\delta),\quad   e^{i(\theta_{n+1}-\theta_{n})(u_0;\varepsilon)}=1\right\}.
\]
Then for some $R(\varepsilon)\to 0$ as $\varepsilon\to 0$,
\[
\varepsilon\sum_{n\in J(u_0;\varepsilon)} \sinc(\pi F(-\lambda_n))
	\leq C(\delta)R(\varepsilon)+C\delta.
\]
\end{thm}

Theorems~\ref{thm:lax_u_trigo} and~\ref{thm:phase_zero} present some differences with Definition 1.5  in~\cite{bo_zero}. Hence, once we prove this asymptotic distribution of Lax eigenvalues and phase constants, we will make precise how one can adapt the arguments in~\cite{bo_zero} in the study of the zero-dispersion limit and get Theorem~\ref{thm:zero}.

\subsection{Strategy of proof}

The strategy for establishing Theorem~\ref{thm:lax_u_trigo} is inspired from the study of spectral data for the Benjamin-Ono equation on the line in~\cite{MillerWetzel2016rational, MillerWetzel2016}.
%The strategy of proof is as follows. 
We transfer the problem into a problem of complex analysis.
We use the identification between the spaces $L^2_+(\T)$ and $L^2_+(\D)$, where
\[
L^2_+(\T)=\left\{u:x\in\T\mapsto \sum_{k=0}^{+\infty}\widehat{u}(k)e^{ikx}\mid \sum_{k=0}^{+\infty}|\widehat{u}(k)|^2<+\infty\right\},
\]
\[
L^2_+(\D)=\left\{u:z\in\D\mapsto \sum_{k=0}^{+\infty}\widehat{u}(k)z^k\mid \sup_{r<1}\int_0^{2\pi}|u(re^{ix})|^2\dd x<+\infty\right\}.
\]
As a consequence, the eigenvectors $f_n(u;\varepsilon)$ of the Lax operator $L_u(\varepsilon)$ can be interpreted as holomorphic functions on $\D$.

A general bell shaped potential $u$, however, does not identify to a holomorphic function on the complex plane. For this reason, we approximate $u$ by a trigonometric polynomial $u_N$ of order $N$, obtained by truncation of its Fourier series. Then the trigonometric polynomial $u_N$ extends to a meromorphic function on $\C$ with a multiple (but finite) pole at zero:
\[
u_N(z)=\sum_{k=1}^N\widehat{u}(k)z^k+\overline{\widehat{u}(k)}z^{-k}.
\]
Hence, we transfer the eigenvalue problem
\begin{equation}\label{eq:lax}
L_u(\varepsilon) f=\lambda f
\end{equation}
to the complex plane, where the eigenvalue equation~\eqref{eq:lax} becomes an ODE. We solve explicitly this equation in order to get an expression for $f$. However, one needs to check that the obtained formula for $f$ is holomorphic on $\D$. This leads to the vanishing of an Evans function
\begin{equation}\label{eq:evans}
\det(A_N(\lambda;\varepsilon))=0,
\end{equation}
where every coefficient of the $N\times N$ matrix $A_N$ is an oscillatory integral on a prescribed contour, with a fast oscillating phase of the form $S_{\lambda}(z)/\varepsilon$. The leading order of the integral is prescribed by the stationary points of the phase. Using the steepest descent method, we deduce an equation for $\lambda$ to be an actual eigenvalue. It turns out that the stationary points of the phase on $\partial\D$ are the antecedents of $u$ by $\lambda$, so that one can make a link with the distribution function $F$ defined in~\eqref{def:F}.

To deduce Theorem~\ref{thm:zero}, we rely on the complete integrability for the Benjamin-Ono equation on the torus established in~\cite{GerardKappeler2019,GerardKappelerTopalov2020}, and generalizes the study of the cosine initial data from part~4 in~\cite{bo_zero}. 

\paragraph{Plan of the paper} In Section~\ref{sec:eigenvalue_equation} we establish an equation of the form~\eqref{eq:evans} characterizing the eigenvalues of the Lax operator associated to a trigonometric polynomial. In Section~\ref{sec:contours}, we deform the contours involved in this equation in order to apply the method of steepest descent, then we determine an asymptotic expansion of the Lax eigenvalues and establish Theorem~\ref{thm:lax_u_trigo}. Then, in Section~\ref{sec:zero}, we determine the weak limit of solutions of~\eqref{eq:bo} as $\varepsilon\to 0$ using the asymptotic expansion of the Lax eigenvalues. Finally, we generalize this result from trigonometric polynomials to general bell shaped initial data in Section~\ref{sec:bell_shaped} to get Theorem~\ref{thm:zero}.

In Appendix~\ref{appendix:phase} we show that the phase constants of an even initial data are multiples of~$\pi$.
Appendix~\ref{sec:trigonometric_approx}  is devoted to properties of the trigonometric approximation of a bell shaped initial data, from which we deduce Corollary~\ref{thm:lax_u}. 

\paragraph{Acknowledgments}  The author would like to warmly thank Patrick Gérard for providing a proof of Theorem~\ref{thm:evans} and for useful discussions about this problem. 

%\newpage
%%%%%%%%%%%%%%%%%%%%%%%%%%%%%%%%%%%%%%%%%%%%%%%%%%%%%%%%%%%%%%%%%%%%%%%%%%%%%%%%%%%%%%%%%%%%%%%%%%%%%%%%%%%%%%%%%%%%%%%%%%%%%%%%%%%%%%%%%%%%%%%%%
\section{Eigenvalue equation for trigonometric polynomials}\label{sec:eigenvalue_equation}
%%%%%%%%%%%%%%%%%%%%%%%%%%%%%%%%%%%%%%%%%%%%%%%%%%%%%%%%%%%%%%%%%%%%%%%%%%%%%%%%%%%%%%%%%%%%%%%%%%%%%%%%%%%%%%%%%%%%%%%%%%%%%%%%%%%%%%%%%%%%%%%%%

In this part, we consider the eigenvalue equation~\eqref{eq:lax}
%\[
%L_u(\varepsilon) f=\lambda f
%\]
when $u$ is a real-valued trigonometric polynomial of order $N$ with zero mean
\[
u(x)=\sum_{k=-N}^Nc_ke^{ikx}.
\]
Since $u$ has zero mean, then $c_0=0$, moreover, since $u$ is real-valued, we have $c_{-k}=\overline{c_k}$, finally, one can assume that $c_N\neq 0$.  Using the spatial translation invariance of equation~\eqref{eq:bo}, one can moreover assume that $c_N>0$ (we do not assume that $u$ is bell shaped here).

We first establish the eigenvalue equation, then we study the stationary points of the phase when $-\lambda$ has either zero or two antecedents by $u$ on $\T$.

%\newpage
%%%%%%%%%%%%%%%%%%%%%%%%%%%%%%%%%%%%%%%%%%%%%%%%%%%%%%%%%%%%%%%%%%%%%%%%%%%%%%%%%%%%%%%%%%%%%%%%%%%%%%%%%%%%%%%%%%%%%%%%%%%%%%%%%%%%%%%%%%%%%%%%%
\subsection{Evans function and eigenvalue equation}
%%%%%%%%%%%%%%%%%%%%%%%%%%%%%%%%%%%%%%%%%%%%%%%%%%%%%%%%%%%%%%%%%%%%%%%%%%%%%%%%%%%%%%%%%%%%%%%%%%%%%%%%%%%%%%%%%%%%%%%%%%%%%%%%%%%%%%%%%%%%%%%%%

In  this part, we state the eigenvalue equation in term of the vanishing of the Evans function $\det( A(\lambda;\varepsilon))$.
The coefficients of the matrix $ A(\lambda;\varepsilon)$ are oscillatory integrals on the following contours.

\begin{mydef}[Contours]\label{def:contours}
For $1\leq k\leq N$, we denote $\theta_k:=\frac{(2k-1)\pi}{N}$.
\begin{itemize}
\item For $1\leq k\leq N$, $\Gamma_k$ is the closed contour made of the juxtaposition of the segment $[0,e^{i\theta_k}]$, the arc of circle of radius $1$ where $\arg(\zeta)$ varies from $\theta_k$ to $\theta_{k+1}$, and the segment $[e^{i\theta_{k+1}},0]$.

\item The contour $\Gamma_N^-$ is the juxtaposition of the segment $[0,e^{i\theta_N}]$ and the arc of circle of radius~$1$ where $\arg(\zeta)$ varies from $\theta_N$ to $2\pi$.

\item The contour $\Gamma_N^+$ is the juxtaposition the arc of circle of radius $1$ where $\arg(\zeta)$ varies from $0$ to $\theta_1$ and of the segment $[e^{i\theta_1},0]$.
\end{itemize}
We orient these contours counterclockwise, see Figure~\ref{fig:contours}.
\end{mydef}

\begin{figure}
\begin{center}
\begin{tikzpicture}
\draw[->] (-2.5,0)--(2.5,0);
\draw[->] (0,-2.5)--(0,2.5);
\draw (2,0) arc (0:360:2);
\begin{scope}[very thick,decoration={
    markings,
    mark=at position 0.5 with {\arrow{>}}}
    ] 
\draw[purple, densely dashed, line width=2,postaction={decorate}] (2,0) arc (0:30:2) -- +(30+180:2);
\draw[violet, line width=2,postaction={decorate}] (0,0)-- (0,2) arc (90:150:2) -- +(-30:2);
\end{scope}
\begin{scope}[very thick,decoration={
    markings,
    mark=at position 0.5 with {\arrow{<}}}
    ] 
\draw[blue, line width=2,postaction={decorate}] (2,0) arc (0:-30:2) -- +(-30+180:2);
\end{scope}
%\draw[blue] (3,0) arc (0:-30:3);

\draw (2.4,0.6) node{{\color{purple}$\Gamma_N^+$}};
\draw (2.4,-0.6) node{{\color{blue}$\Gamma_N^-$}};
\draw (1.8,-1.4) node{{$\theta_N$}};
\draw (1.8,1.4) node{{$\theta_1$}};
\draw (0.3,2.3) node{{$\theta_k$}};
\draw (-2.1,1.3) node{{$\theta_{k+1}$}};

\draw (-1,2.2) node{{\color{violet}$\Gamma_k$}};
\draw (0,0) node[below left]{$0$};
\draw (2,0) node[below left]{$1$};
\end{tikzpicture}
\end{center}
\caption{Contours of the eigenvalue equation}\label{fig:contours}
\end{figure}
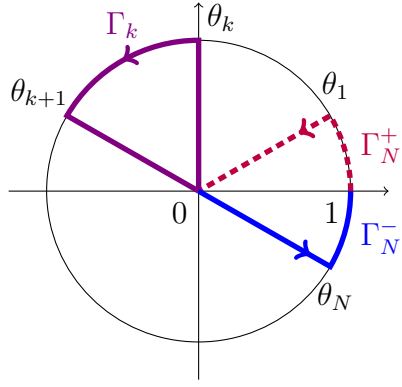

We extend the function $u$ as a meromorphic function on the complex plane with only one multiple pole at $0$: for $z\in\C\setminus\{0\}$,
\[
u(z):=\sum_{k=-N}^Nc_kz^k.
\]
Let us define $Q$ in $\C\setminus\{0\}$ by the formula
\begin{equation}\label{def:Q}
Q(z):=\sum_{k=1}^N\left(-c_k\frac{z^k}{k}+\overline{c_k}\frac{z^{-k}}{k}\right),
\end{equation}
so that $-zQ'(z)=u(z)$.
The coefficients of the matrix $ A(\lambda;\varepsilon)$ are the following oscillatory integrals.
\begin{mydef}[Oscillatory integrals]\label{def:A}
For $1\leq k\leq N-1$ and $1\leq \ell\leq N$, we define
\[
A_{k,\ell}:=\int_{\Gamma_k}e^{Q(\zeta)/\varepsilon}\zeta^{-\ell-\lambda/\varepsilon}\frac{\dd\zeta}{\zeta},
\]
\[
A_{N,\ell}^{\pm}:=\int_{\Gamma_N^{\pm}}e^{Q(\zeta)/\varepsilon}\zeta^{-\ell-\lambda/\varepsilon}\frac{\dd\zeta}{\zeta},
\]
\[
A_{N,\ell}:=A_{N,\ell}^++e^{-2i\pi\lambda/\varepsilon}A_{N,\ell}^-.
\]
Let $A(\lambda;\varepsilon)=(A_{k,\ell})_{1\leq k,\ell\leq N}$ be the $N\times N$ matrix with coefficients $A_{k,\ell}$.
\end{mydef}

%\begin{mydef}[Evans function]
 %Then we denote
%\[
%D(\lambda;\varepsilon):=\det(J(\lambda;\varepsilon)).
%\]
%\end{mydef}

We show that $\lambda$ is an eigenvalue of the Lax operator $L_u(\varepsilon)$ if and only if the Evans function $\det(A(\lambda;\varepsilon))$ vanishes.
\begin{thm}[Eigenvalue equation]\label{thm:evans}
Let $u$ be a real-valued trigonometric polynomial of order $N$ with zero mean. A real number $\lambda$ is an eigenvalue of $L_u(\varepsilon)$ if and only if
\begin{equation}\label{eq:evans1}
\det(A(\lambda;\varepsilon))=0.
\end{equation}
\end{thm}

%\newpage
%%%%%%%%%%%%%%%%%%%%%%%%%%%%%%%%%%%%%%%%%%%%%%%%%%%%%%%%%%%%%%%%%%%%%%%%%%%%%%%%%%%%%%%%%%%%%%%%%%%%%%%%%%%%%%%%%%%%%%%%%%%%%%%%%%%%%%%%%%%%%%%%%
\subsection{Proof of the eigenvalue equation}
%%%%%%%%%%%%%%%%%%%%%%%%%%%%%%%%%%%%%%%%%%%%%%%%%%%%%%%%%%%%%%%%%%%%%%%%%%%%%%%%%%%%%%%%%%%%%%%%%%%%%%%%%%%%%%%%%%%%%%%%%%%%%%%%%%%%%%%%%%%%%%%%%

Let us consider an eigenfunction $f$ of $L_u(\varepsilon)$ with eigenvalue $\lambda$: $L_u(\varepsilon)f=\lambda f$,
where we recall that
\[
L_u(\varepsilon)f=-i\varepsilon\partial_x f-\Pi(uf)= -i\varepsilon\partial_x f-uf+(\mathrm{Id}-\Pi)(uf).
\]
We note that $L_u(\varepsilon)f=L_{u/\varepsilon}(1)(\varepsilon f)$. Hence, if $L_u(\varepsilon)f=\lambda f$, then $L_{u/\varepsilon}(1)(\varepsilon f)=\frac{\lambda}{\varepsilon}(\varepsilon f)$. Up to replacing $f$ by $\varepsilon f$, $u$ by $\frac{u}{\varepsilon}$ and $\lambda$ by $\frac{\lambda}{\varepsilon}$ at the end of the argument, we therefore assume that $\varepsilon=1$.  

Since $f\in L^2_+(\T)$, we express $f$ as a holomorphic function on the unit disk $\D$. The eigenvalue equation becomes: for every $0<|z|<1$,
\[
\varepsilon zf'(z)-\left(\sum_{k=1}^Nc_kz^k+\overline{c_k}z^{-k}\right)f(z)+\sum_{k=1}^N\overline{c_k}z^{-k}\left(\sum_{j=0}^{k-1}\frac{f^{(j)}(0)}{j!}z^j\right)=\lambda f(z).
\]
This implies the existence of complex numbers $v_0,\dots,v_N$ such that
\begin{equation}\label{eq:f}
zf'(z)-\left(\sum_{k=1}^Nc_kz^k+\overline{c_k}z^{-k}\right)f(z)-\lambda f(z)=\sum_{j=1}^N v_jz^{-j}.
\end{equation}
Conversely, if $f$ is a holomorphic function near the origin solution to~\eqref{eq:f}, then the constants $v_1,\dots, v_N$ are uniquely expressed in terms of $f(0),\dots,f^{(N-1)}(0)$ by the triangular system
\[
v_j=\sum_{k=j+1}^N\overline{c_k} \frac{f^{(k-j)}(0)}{(k-j)!}.
\]
Moreover, if $f$ is a solution to~\eqref{eq:f} on $\D$, and if $f$ is holomorphic on $\D$ and bounded on $\partial \D$, then $f$ has a representative in $L^2_+(\T)$ so that $\lambda$ is an eigenvalue of $L_u(1)$ associated to the eigenfunction $f$. We deduce that the eigenvalues of $L_u$ are the real numbers $\lambda$ such that equation~\eqref{eq:f} admits a nonzero solution $f\in L^2_+(\D)$, for some complex numbers $v_1,\dots, v_N$.

We first observe that equation~\eqref{eq:f} is a linear differential equation of order $1$. On the simply connected open subset
\[
\Omega=\C\setminus[0,+\infty),
\]
this equation can be solved explicitly and the solutions are completely characterized by their value at one given point of $\Omega$. We choose the determination of $\mathrm{arg}(z)\in(0,2\pi)$ for $z\in\Omega$. Note that in the rest of the argument, one could choose the determination of the argument to be defined on $\C\setminus e^{i\theta}\R_+$ for any $\theta\in\T$ up to rotating everything by the angle $\theta$ in the proof, including the contours $\Gamma_k$. We get that~\eqref{eq:f} is equivalent to the equation
\begin{equation}\label{eq:f_factorized}
\frac{\dd}{\dd z}\left[z^{-\lambda}e^{Q(z)}f(z)\right]=z^{-\lambda}e^{Q(z)}\sum_{j=1}^Nv_jz^{-j-1},
\end{equation}
where we recall~\eqref{def:Q}
\[
Q(z)=\sum_{k=1}^N\left(-c_k\frac{z^k}{k}+\overline{c_k}\frac{z^{-k}}{k}\right).
\]
 Since we have assumed that $c_N>0$, we know that for every $1\leq k\leq N$, the real part of $Q$ around the angle $\theta_k=\frac{(2k-1)\pi}{N}$ has the following asymptotic expansion as $r\to 0^+$:
\begin{equation}\label{eq:ReQ_theta}
\mathrm{Re}[Q(re^{i\theta_k})]\sim -\frac{c_N}{Nr^N}.
\end{equation}
Moreover, for $z\to 0$ in the angular sector 
\[
|\arg(z)-\theta_k|\leq \frac{\pi}{2N}-\delta,
\quad\delta\in\left(0,\frac{\pi}{2N}\right),
\]
there exists $a_N(\delta)>0$ such that
\begin{equation}\label{eq:cv_theta_k}
\mathrm{Re}[Q(z)]\leq -\frac{a_N(\delta)}{|z|^N}.
\end{equation}

Fix $1\leq k\leq N$. For every $z\in\Omega$, we choose a path $\gamma_{k,z}$ joining $0$ to $z$ in $\Omega$ such that for $t\geq 0$ small enough, we have
\[
\gamma_{k,z}(t)=te^{i\theta_k}.
\]
Thanks to~\eqref{eq:cv_theta_k}, we get a well-defined and holomorphic solution $f_k$ of~\eqref{eq:f} on $\Omega$ as follows:
\begin{equation}\label{eq:fk_vj}
f_k(z)=z^{\lambda}e^{-Q(z)}\int_{\gamma_{k,z}}e^{Q(\zeta)}\sum_{j=1}^Nv_j\zeta^{-j-\lambda}\frac{\dd\zeta}{\zeta}.
\end{equation}

\begin{prop}[Boundedness of $f_k$]\label{prop:fk}
The function $f_k$ is bounded on the angular sector
\[
A_k=\left\{z\in\Omega\mid |z|\leq 1,\quad \mathrm{arg}(z)\in\left[\theta_k-\frac{\pi}{N},\theta_k+\frac{\pi}{N}\right]\right\}.
\]
\end{prop}
\begin{proof}
In the proof, we first Taylor expand $f_k$ around zero. Then we show that $f_k$ is bounded in the angular sector $|\arg(z)-\theta_k|\leq \frac{\pi}{2N}-\delta$ for $\delta>0$. Finally we show that $f_k$ is also bounded in the angular sector  $\frac{\pi}{2N}+\delta\leq |\arg(z)-\theta_k|\leq \frac{\pi}{N}$. We conclude to the boundedness for the missing angles by using the Phragmen-Lindelöf principle.

$\bullet$ We first observe that for every $M>N$, there exist coefficients $b_0,\dots,b_M$ and a polynomial function $P_M$ such that
\[
e^{Q(z)}\sum_{j=1}^Nv_jz^{-j-\lambda-1}=\frac{\dd}{\dd z}\left(\sum_{\ell=0}^Mb_{\ell}z^{\ell-\lambda}e^{Q(z)}\right)+z^{M-N-\lambda}P_M(z)e^{Q(z)}.
\]
Indeed, the coefficients $b_0,\dots,b_M$ are the solutions of a triangular system when plugged into~\eqref{eq:f}. Hence $f_k$ has the following expression:
\begin{equation}\label{eq:fk}
f_k(z)=\sum_{\ell=0}^Mb_{\ell}z^{\ell}+z^{\lambda}e^{-Q(z)}\int_{\gamma_{k,z}}e^{Q(\zeta)}\zeta^{M-N-\lambda}P_M(\zeta)\dd \zeta.
\end{equation}

$\bullet$ We first assume  that there exists $\delta>0$ such that
\[
|\arg(z)-\theta_k|\leq\frac{\pi}{2N}-\delta.
\]
Then we can choose the path $\gamma_{k,z}$  to stay inside the sector $|\arg(\zeta)-\theta_k|\leq\frac{\pi}{2N}-\delta$. In this case, estimate~\eqref{eq:cv_theta_k} combined with the holomorphy of $f_k$ implies that one can transform the contour $\gamma_{k,z}$ into the segment $[0,z]$ without changing the value of the integral in~\eqref{eq:fk}. Moreover, when $|z|$ is small enough in this angular sector, there exists $a_N(\delta)>0$ such that
\[
\mathrm{Re}[\zeta Q'(\zeta)]\geq  \frac{a_N(\delta)}{|\zeta|^N},
\]
therefore the function $t\in(0,1]\mapsto\mathrm{Re}[Q(tz)]$ is increasing  and $\mathrm{Re}[Q(\zeta)-Q(z)]\leq 0$ on $[0,z]$. We deduce that when $M$ is large enough, the right hand side of~\eqref{eq:fk} is bounded in the angular sector \(
|\mathrm{\arg(z)}-\theta_k|\leq\frac{\pi}{2N}-\delta.
\)

$\bullet$ We now assume that there exists $\delta>0$ such that
\[
\frac{\pi}{2N}+\delta\leq |\arg(z)-\theta_k|\leq\frac{\pi}{N}.
\]
Then one has that for some $a_N(\delta)>0$, when $z$ is in this angular sector,
\begin{equation}\label{eq:ReQ_min}
\mathrm{Re}[Q(z)]\geq \frac{a_N(\delta)}{|z|^N}.
\end{equation}
We choose $\gamma_{k,z}$ to be the juxtaposition of the following three paths:
\begin{enumerate}
\item the segment $[0,e^{i\theta_k}]$;
\item the arc of circle of radius $1$ where $\arg(\zeta)$ goes from $\theta_k$ to $\arg(z)$;
\item the segment $[z/|z|,z]$.
\end{enumerate}
Now, we show that the right hand side of equality~\eqref{eq:fk} is bounded. Given the expression of $f_k(z)$ in~\eqref{eq:fk}, it is enough to prove that this expression is bounded when $z\to 0$.
\begin{enumerate}
\item 
Because of inequality~\eqref{eq:ReQ_min}, one can see that $z^{\lambda}e^{-Q(z)}\to 0$ as $z\to 0$. 
On $[0,e^{i\theta_k}]$, one has seen in~\eqref{eq:ReQ_theta} that $\mathrm{Re}[Q(re^{i\theta_k})]\sim -\frac{c_N}{Nr^N}$ as $r\to 0^+$. In particular the integral $\int_{[0,e^{i\theta_k}]}e^{Q(\zeta)}\zeta^{M-N-\lambda}P_M(\zeta)\dd\zeta$ stays bounded on $[0,e^{i\theta_k}]$.

\item
In the arc of  circle of radius $1$ where $\arg(\zeta)$ goes from $\theta_k$ to $\arg(z)$, there is nothing to prove.

\item 
We now consider the segment $[z/|z|,z]$. Our goal is to show that when $M$ is large enough, the following integral stays bounded as $z\to 0$:
\begin{equation}\label{eq:integral_bdd}
z^{\lambda}e^{-Q(z)}\int_{[z/|z|,z]}e^{Q(\zeta)}\zeta^{M-N-\lambda}P_M(\zeta)\dd\zeta.
\end{equation}
We use the fact that if $|\zeta|\leq 1$ and $\frac{\pi}{2N}+\delta\leq |\arg(\zeta)-\theta_k|\leq\frac{\pi}{N}$, then
\[
\mathrm{Re}[\zeta Q'(\zeta)]\leq-\frac{a_N(\delta)}{|\zeta|^N}
\]
to deduce that the function $t\in[1,|z|^{-1}]\mapsto \mathrm{Re}[Q(tz)]$ is decreasing. As a consequence, for every $\zeta\in [z/|z|,z]$, there holds $\mathrm{Re}[Q(\zeta)-Q(z)]\leq 0$. We conclude that when $M>N+\lambda$, the quantity~\eqref{eq:integral_bdd} is bounded.
\end{enumerate}

$\bullet$ We have shown in the latter steps that  for every $\delta>0$, $f_k$ is bounded on the complement of $|\mathrm{arg}(z)-\theta_k-\frac{\pi}{2N}|\leq\delta$ in $A_k$. Furthermore, we see that there exists $C,A>0$ such that for every $z\in A_k$, 
\[
|f_k(z)|\leq Ce^{A|z|^{-N}}.
\]
Making the change of variable $z\mapsto \frac{1}{z}$, Proposition~\ref{prop:fk} is  now a consequence of the following lemma applied to $\theta=\theta_k+\frac{\pi}{2N}$, $\beta=\frac{\pi}{2N}$ and $\delta>0$.

\begin{lem}
Let $f$ be a holomorphic function in a neighborhood of the sector
\[
\Sigma_{\beta}:=\{z\in\C\mid |z|>1,\quad \arg(z)\in[\theta-\beta,\theta+\beta]\}
\]
for some $\beta>0$ and $\theta\in\T$. Assume that there exists $C,A,N>0$ such that the following holds:
\begin{enumerate}
\item for every $z\in\Sigma_{\beta}$, $|f(z)|\leq Ce^{A|z|^N}$;
\item for every $\delta>0$, $f$ is bounded on $\Sigma_{\beta}\setminus \Sigma_{\delta}$. 
\end{enumerate}
Then $f$ is bounded on $\Sigma_{\beta}$.
\end{lem}

The proof of this lemma is as follows. Let $\delta>0$ such that $2\delta N<\pi$. We apply the Phragmen-Lindelöf principle to $f$ on $\Sigma_{\delta}$. Therefore $f$ is bounded on $\Sigma_{\delta}$, hence on $\Sigma_{\beta}$.
\end{proof}

As a consequence of Proposition~\ref{prop:fk}, we have defined $N$ solutions $f_1,\dots,f_N$ of~\eqref{eq:f} on $\Omega$ with the same Taylor expansion around $0$, and such that for every $k$, $f_k$ is bounded on $A_k$. If there exists a solution to~\eqref{eq:f} which is holomorphic near the origin, then every function $f_k$ should coincide with $f$ on $A_k$. This implies the following system of $N-1$ equations
\begin{equation}\label{eq:syst1}
f_k(e^{i(\theta_k+\frac{\pi}{N})})=f_{k+1}(e^{i(\theta_k+\frac{\pi}{N})}), \quad 1\leq k\leq N-1.
\end{equation}
For $1\leq k\leq N-1$, we note that from the integral expression~\eqref{eq:fk_vj}, we have
\[
f_k(e^{i(\theta_k+\frac{\pi}{N})})-f_{k+1}(e^{i(\theta_k+\frac{\pi}{N})})
	=e^{i\lambda(\theta_k+\frac{\pi}{N})-Q(e^{i(\theta_k+\frac{\pi}{N})})} \int_{\Gamma_k}e^{Q(\zeta)}\sum_{j=1}^Nv_j\zeta^{-j-\lambda}\frac{\dd\zeta}{\zeta},
\]
where we recall that the contours $\Gamma_k$ have been introduced in Definition~\ref{def:contours}. The equations~\eqref{eq:syst1} mean that all these quantities cancel, meaning that all the integrals on $\Gamma_k$ cancel:
\begin{equation}\label{eq:syst1_bis}
\int_{\Gamma_k}e^{Q(\zeta)}\sum_{j=1}^Nv_j\zeta^{-j-\lambda}\frac{\dd\zeta}{\zeta}=0.
\end{equation}
Given expression~\eqref{eq:fk_vj}, this is a system of $N-1$ linear equation on $v_1,\dots,v_N$.

Moreover, integrating~\eqref{eq:f_factorized} on the circle of radius $1$, we also get the equation
\begin{equation}\label{eq:syst2}
f(1)(e^{-2i\pi\lambda}-1)e^{Q(1)}=\int_{\partial \D}z^{-\lambda}e^{Q(z)}\sum_{j=1}^Nv_jz^{-j-1}\dd z.
\end{equation}
Replacing $f(1)$ by $f_1(1+i0)$, we get a new linear equation on $v_1,\dots,v_N$.

\begin{prop}%\label{prop:eigenvalues}
The linear system of $N$ equations~\eqref{eq:syst1_bis} and~\eqref{eq:syst2} admits a nontrivial solution $(v_1,\dots,v_N)$ if and only if $\lambda$ is an eigenvalue of $L_u(1)$. 
\end{prop}

\begin{proof}
It only remains to prove the converse. Assume that the linear system of $N$ equations~\eqref{eq:syst1} and~\eqref{eq:syst2} admits a nontrivial solution $(v_1,\dots,v_N)$, and let us prove that $\lambda$ is an eigenvalue of $L_u(1)$.

Since the solutions to equation~\eqref{eq:f} in $\Omega$ are unique, we get from~\eqref{eq:syst1} that all the $f_k$ coincide, so that Proposition~\ref{prop:fk} implies that $f$ is bounded on the union of the angular sectors $S_k$, hence on $\Omega\cap\D$.

$\bullet$ We first assume that $\lambda$ is not an integer. Define for $z\neq 0$
\[
g(z):=\frac{ie^{-Q(z)}}{e^{-2i\pi\lambda}-1}\int_0^{2\pi}e^{Q(ze^{i\theta})-i\lambda\theta} \sum_{j=1}^Nv_j z^{-j}e^{-ij\theta}\dd\theta.
\]
Then $g$ is holomorphic on $\C\setminus\{0\}$ and one can check that $g$ satisfies equation~\eqref{eq:f}.

Moreover, the choice of constant term implies that $g(1)=f(1)$, therefore $f=g$ on $\D\cap\Omega$. Moreover, $f$ is bounded on $\D\cap\Omega$, therefore $g$ is bounded on $\D$ by continuity. Hence $g$ is also holomorphic near the origin. We conclude that $\lambda$ is an eigenvalue of $L_u(1)$.

$\bullet$ We now assume that $\lambda$ is an integer. 
Withdrawing the sum of the integrals in~\eqref{eq:syst1_bis} from the integral~\eqref{eq:syst2} on $\partial\D$, which cancels since equation~\eqref{eq:syst2} becomes
\[
0=\int_{\partial \D}z^{-\lambda}e^{Q(z)}\sum_{j=1}^Nv_jz^{-j-1}\dd z,
\]
we deduce that
\[
\int_{\Gamma_N}e^{Q(\zeta)}\sum_{j=1}^Nv_j\zeta^{-j-\lambda}\frac{\dd\zeta}{\zeta}=0.
\]
For $0<r<1$, we consider the contour $\Gamma_N(r)$ defined the closed contour made of the juxtaposition of the segment $[0,re^{i\theta_k}]$, the arc of circle of radius $r$ where $\arg(\zeta)$ varies from $\theta_k$ to $\theta_{k+1}$, and the segment $[re^{i\theta_{k+1}},0]$. Since the above integrand is holomorphic outside the origin, the integral on $\Gamma_N$ is equal to the one on $\Gamma_N(r)$:
\[
\int_{\Gamma_N(r)}e^{Q(\zeta)}\sum_{j=1}^Nv_j\zeta^{-j-\lambda}\frac{\dd\zeta}{\zeta}=0.
\]
This implies that $f(r+i0)=f(r-i0)$ for $0<r\leq 1$. Since $f$ solves~\eqref{eq:f} on $\Omega$, we deduce that $f$ is holomorphic outside the origin. Moreover, since $f$ is bounded on $\Omega\cap\D$, then $f$ is also holomorphic near the origin. We conclude that $\lambda$ is an eigenvalue of $L_u(1)$.
\end{proof}

Hence, we have seen that $\lambda$ is an eigenvalue of $L_u(1)$ if and only if there exists a nontrivial solution $(v_1,\dots,v_N)$ such that~\eqref{eq:syst1_bis} and~\eqref{eq:syst2} hold. 

We now replace $f(1)$ by its integral expression~\eqref{eq:fk_vj} on $\gamma_{1,1}=\Gamma_N^+$ in equation~\eqref{eq:syst2}. Then, removing~\eqref{eq:syst1_bis} for $1\leq k\leq N-1$ on the right hand side of equation~\eqref{eq:syst2}, equation~\eqref{eq:syst2} becomes equivalent to
\[
-(e^{-2i\pi\lambda}-1)\sum_{\ell=1}^NA_{N,\ell}^+v_{\ell}=\sum_{\ell=1}^N (A_{N,\ell}^++A_{N,\ell}^-)v_{\ell}.
\] 
We conclude that the linear system~\eqref{eq:syst1_bis},~\eqref{eq:syst2} is equivalent to the existence of  $V=(v_1,\dots,v_N)$ such that
\[
A(u;1)V=0
\]
where we recall that $A(u;\varepsilon)$ was introduced in Definition~\ref{def:A}. Moreover, $f\not\equiv 0$ if and only if $V\not\equiv 0$, therefore this property is equivalent to
\[
\det(A(u;1))=0.
\]

%\newpage
%%%%%%%%%%%%%%%%%%%%%%%%%%%%%%%%%%%%%%%%%%%%%%%%%%%%%%%%%%%%%%%%%%%%%%%%%%%%%%%%%%%%%%%%%%%%%%%%%%%%%%%%%%%%%%%%%%%%%%%%%%%%%%%%%%%%%%%%%%%%%%%%%
\subsection{Stationary points of the phase}%\label{part:shape_u}
%%%%%%%%%%%%%%%%%%%%%%%%%%%%%%%%%%%%%%%%%%%%%%%%%%%%%%%%%%%%%%%%%%%%%%%%%%%%%%%%%%%%%%%%%%%%%%%%%%%%%%%%%%%%%%%%%%%%%%%%%%%%%%%%%%%%%%%%%%%%%%%%%

In order to find an asymptotic expansion of the Lax eigenvalues, the general strategy is to now apply the method of steepest descent to every coefficient of $A(u;\varepsilon)$ in the determinant formula $\det(A(u;\varepsilon))=0$. Intuitively, in the limit $\varepsilon\to 0$, the leading asymptotics are driven by the stationary points of the phase $e^{iS_{\lambda}(z)/\varepsilon}$, where
\[
S_{\lambda}(z)=Q(z)-\lambda\log(z).
\]
Since $-zQ'(z)=u(z)$, the stationary points satisfy $S_{\lambda}'(z)=0$ whenever
\[
u(z)+\lambda=0.
\]
%Note that we will allow $N$ to depend on $\varepsilon$ later in the paper, therefore we search for asymptotic expansions both in $\varepsilon$ and in $N$.

In this part, we describe basic properties of the holomorphic extension of $u+\lambda$ in $\C^*$ when $u+\lambda$ has zero or two antecedents on $\T$. %, which writes in complex form as $u(e^{ix_{\pm}(-\lambda)})=\lambda$. 

\begin{lem}[Critical points]\label{lem:def_pk}~
\begin{itemize}
\item (Small eigenvalues) Assume that $p_{\pm}(\lambda)=e^{ix_{\pm}(-\lambda)}$ are the only solutions to the equation $u(p_{\pm}(\lambda))=-\lambda$ on $\partial\D$. Then there exist $p_1(\lambda),\dots, p_{N-1}(\lambda)\in\D$ such that for every $z\in\C^*$,
\begin{equation}\label{eq:u+lambda}
u(z)+\lambda=c_Nz^{-N}(z-p_+)(z-p_-)\prod_{k=1}^{N-1}(z-p_k)\left(z-1/\overline{p_k}\right).
\end{equation}
\item (Large eigenvalues)
Similarly, assume that the equation $u+\lambda=0$ has no solution on $\partial\D$. Then there exist $p_1(\lambda),\dots, p_{N}(\lambda)\in\D$ such that for every $z\in\C^*$,
\[
u(z)+\lambda=c_Nz^{-N}\prod_{k=1}^{N}(z-p_k)\left(z-1/\overline{p_k}\right).
\]
\end{itemize}
\end{lem}

\begin{lem}[Spacing of roots]\label{lem:spacing}
Let $u$ be a bell-shaped trigonometric polynomial of order $N$. There exists $y_1<\dots<y_M$, $M\leq 2N$, such that  if 
\[\lambda\not\in \bigcup_{k=1}^M(y_k-\delta,y_k+\delta)\cup(K(\delta),+\infty),
\] then all the roots $p_1,\dots,p_{N-1},p_+,p_-$ in the small eigenvalues case (or $p_1,\dots,p_N$ in the large eigenvalues case)  have single multiplicity, moreover they satisfy the following uniform bounds:
\begin{itemize}
\item if $\lambda\in\Lambda_-(\delta)$, then
\[
|p_k(\lambda)-p_{\pm}(\lambda)|\geq \frac{1}{C_u(\delta)};
\]
\item if $k\neq l$, then
\[
|p_k(\lambda)-p_l(\lambda)|\geq \frac{1}{C_u(\delta)},
\quad
|p_k(\lambda)|\geq \frac{1}{C_u(\delta)}.
\]
\end{itemize}
\end{lem}

\begin{proof}%[Single multiplicity]\label{rk:critical points}
%One only needs to check the single multiplicity of the roots.

Since $P_{\lambda}(z)=z^N(u(z)+\lambda)$ is a polynomial of order $2N$, one can see that its derivative $P_{\lambda}'(z)=Nz^{N-1}(u(z)+\lambda)+z^N u'(z)$ is a polynomial of order $2N-1$. We note that a root  $z$ in $\C\setminus\{0\}$ of $P_{\lambda}$ is a multiple nonzero root   if and only if $u(z)=-\lambda$ and $u'(z)=0$. This happens at most $2N$ times since $P_{\lambda}$ is a polynomial of order $2N$. Denote $z_1,\dots,z_{N'}$, $N'\leq 2N$, the nonzero roots of $z^{N+1}u'(z)$, where we remove the eventual multiplicity. If $P_{\lambda}$ has a multiple nonzero root, then it must be one of the $z_k$, hence $u(z_k)+\lambda=0$. Let us denote $y_1< \dots< y_M$, $M\leq 2N$, the real numbers $y$ such that there exists a multiple nonzero root $z_k$  of $P_{\lambda}$, satisfying $u(z_k)=-y$.

Without loss of generality, we replace the set of small eigenvalues $\Lambda_-(\delta)=(-\max(u)+\delta,-\min(u)-\delta)$ by the compact set
\[
\widetilde{\Lambda}_-(\delta):=[-\max(u)+\delta,-\min(u)-\delta]\setminus \bigcup_{k=1}^{M}(y_k-\delta,y_k+\delta).
\]

On $\widetilde{\Lambda}_-(\delta)$, we know that $p_+(\lambda),p_-(\lambda), p_k(\lambda)$ are simple roots of the polynomial $z^N(u(z)+\lambda)$ with smooth coefficients with respect to the parameter $\lambda$, therefore, these roots are smooth with respect to the parameter $\lambda$. In particular, they lie in a compact set of $\C\setminus\{0\}$ and have a bounded derivative for $\lambda\in\Lambda_-(\delta)$.

The same applies by removing the problematic bands to the set $\Lambda_+(\delta)=[-\min(u)+\delta,K(\delta)]$ of large eigenvalues:
\[
\widetilde{\Lambda}_+(\delta):=[-\min(u)+\delta, K(\delta)]\setminus \bigcup_{k=1}^{M}(y_k-\delta,y_k+\delta).\qedhere
\]
\end{proof}

We recall that the oscillatory phase satisfies
\[S_{\lambda}(z)=Q(z)-\lambda\log(z),
\quad
S_{\lambda}'(z)=-\frac{u(z)+\lambda}{z}.
\]
%\[
%S_{\lambda}''(z)=\frac{u(z)+\lambda}{z^2}-\frac{u'(z)}{z}.
%\]
As a consequence of Lemma~\ref{lem:spacing}, when $\lambda\in\widetilde{\Lambda_-}(\delta)\cup \widetilde{\Lambda_+}(\delta)$, we get by compactness and continuity the following uniform bounds on the phase on a vicinity of the unit disk and far enough from $z=0$.
\begin{lem}[Bounds for the phase]
Consider the constant $C_u(\delta)>0$ from Lemma~\ref{lem:spacing}. Let $V(\delta):=B_{\C}(0,2)\setminus B_{\C}(0,\frac{1}{2C_u(\delta)})$. There exists $C_u'(\delta)>0$ such that for every  $\lambda\in\widetilde{\Lambda_-}(\delta)\cup \widetilde{\Lambda_+}(\delta)$, for every $z\in V(\delta)$,
\begin{itemize}
\item 
 $\|S_{\lambda}-S_{\lambda}(p_k)\|_{\classeC^4(V(\delta))}\leq C'_u(\delta)$;
\item $\frac{|z-p_k|}{|S'_{\lambda}(z)|}\leq C'_u(\delta)$;
\item $\|z^{-\ell-1}\|_{\classeC^2(V(\delta))}\leq C'_u(\delta)$ for $1\leq \ell\leq N$. 
\end{itemize}
\end{lem}

We also note that at the critical points $p_k$ and $p_{\pm}$, the real part $\Re(S_{\lambda})$ has at saddle point, whereas for $z$ outside of these critical points, $\Re(S_{\lambda})$ splits the vicinity of $z$ into two parts $\Re(S_{\lambda})>\Re(S_{\lambda})(z)$ and $\Re(S_{\lambda})<\Re(S_{\lambda})(z)$.
\begin{lem}[Real part of the phase]\label{lem:fusion} Assume that $\lambda\in\widetilde{\Lambda_-}(\delta)\cup\widetilde{\Lambda_+}(\delta)$.
\begin{itemize}
\item At the vicinity of the critical points $z=p_k$ and $z=p_{\pm}$, the level set $\Re(S_{\lambda}(\zeta))=\Re(S_{\lambda}(z))$ is composed of two paths which intersect perpendicularly at $z$. Moreover, $z$ is at saddle point.
\item At the vicinity of $z$ outside of these critical points, the level set $\Re(S_{\lambda}(\zeta))=\Re(S_{\lambda}(z))$ is a path, which splits the vicinity of $z$ into two parts $\Re(S_{\lambda})>\Re(S_{\lambda})(z)$ and $\Re(S_{\lambda})<\Re(S_{\lambda})(z)$.
\end{itemize}
\end{lem} 
 
%\begin{proof}
The proof of this lemma consists in a Taylor expansion of $S_{\lambda}$ around $z$.
%\end{proof}

We now simplify the expressions of $S_{\lambda}$ and $S_{\lambda}''$ that we will obtain when applying the steepest descent method, as a direct consequence of the factorization of $u+\lambda$.
\begin{lem}[Second derivative of the phase]\label{lem:S''}
Assume that $u+\lambda$ has two zeroes on $\T$ and that $p_1(\lambda),\dots,p_{N-1}(\lambda),p_+(\lambda),p_-(\lambda)$ are distinct. Then at the stationary points $p_{\pm}(\lambda)$, we have
\[
|S_{\lambda}''(p_{\pm}(\lambda))|
	=c_N|p_+(\lambda)-p_-(\lambda)|\prod_{k=1}^{N-1}\frac{|p_k(\lambda)-p_{\pm}(\lambda)|^2}{|p_k(\lambda)|}.
\]
%Similarly,
%\[
%|S_{\lambda}''(p_k)|
%	=c_N|p_+-p_k||p_--p_k||p_k-1/\overline{p_k}|\prod_{\substack{j=1\\j\neq k}}^{N-1}|p_k-p_j||p_k-1/\overline{p_j}|.
%\]
\end{lem}

\begin{proof}
We treat the case $p_{\pm}=p_+$ to simplify the notation. Since $u(p_{+})+\lambda=0$ and $|p_{+}|=1$, we compute
\begin{align*}
|S_{\lambda}''(p_{+})|
	=|u'(p_{+})|
	=c_N|p_+-p_-|\prod_{k=1}^{N-1}|p_+-p_k||p_+-1/\overline{p_k}|.
\end{align*}
Finally, we note that
\[
|p_+-1/\overline{p_k}|=|p_+||p_k|^{-1}|\overline{p_k}-1/p_+|=|p_k|^{-1}|\overline{p_k}-\overline{p_+}|.\qedhere
\]
%Similarly, we compute
%\[
%|S_{\lambda}''(p_k)|
%	=c_N|p_+-p_k||p_--p_k||p_k-1/\overline{p_k}|\prod_{\substack{j=1\\j\neq k}}^{N-1}|p_k-p_j||p_k-1/\overline{p_j}|.
%\]
\end{proof}

Finally, we make a link between the integral in the Bohr-Sommerfeld quantization formulation for the Lax eigenvalues in Theorem~\ref{thm:lax_u}, and the value of the phase  at the critical points $p_{\pm}(\lambda)=e^{ix_{\pm}(-\lambda)}$.

\begin{lem}[Phase and distribution function]\label{lem:aire}
Assume that $-\lambda\in(\min(u),\max(u))$. Then there holds
\[
S_{\lambda}(p_+(\lambda))-S_{\lambda}(p_-(\lambda))
	=2i\pi \int_{-\lambda}^{\max(u)}F(\nu)\dd\nu.
\]
\end{lem}

\begin{proof}
We write the difference of phase values as an integral from $x_-=x_-(-\lambda)$ to $x_+=x_+(-\lambda)$. We get by using the original notation for the function $u$ defined on $\T$ that
\begin{align*}
S_{\lambda}(p_+(\lambda))-S_{\lambda}(p_-(\lambda))
	&=\int_{x_-(\lambda)}^{x_+(\lambda)}ie^{ix}S_{\lambda}'(e^{ix})\dd x\\
%	&=-i\int_{x_-(\lambda)}^{x_+(\lambda)}(u(e^{ix})+\lambda)\dd x\\
	&=-i\int_{x_-(\lambda)}^{x_+(\lambda)}(u(x)+\lambda)\dd x.
\end{align*}
Let $x_{\max}\in(x_-(\lambda),x_+(\lambda))$ be the point on $\T$ such that $u(x_{\max})=\max(u)$.
Using an integration by parts then a change of variables $u(x)=\nu$, leading to $u'(x)\dd x=\dd \nu$ and $x_{\pm}(\nu)=x$, we get
\begin{align*}
\int_{x_-(\lambda)}^{x_+(\lambda)}u(x)\dd x
	&=\Big[u(x)\Big]_{x_-(\lambda)}^{x_+(\lambda)}-\int_{x_-(\lambda)}^{x_{\max}}xu'(x)\dd x-\int_{x_{\max}}^{x_+(\lambda)}xu'(x)\dd x\\
	&=-\lambda(x_+(\lambda)-x_-(\lambda))-\int_{-\lambda}^{\max(u)}(x_-(\nu)-x_+(\nu))\dd\nu.
\end{align*}
This leads to the formula.
\end{proof}

%\newpage
%%%%%%%%%%%%%%%%%%%%%%%%%%%%%%%%%%%%%%%%%%%%%%%%%%%%%%%%%%%%%%%%%%%%%%%%%%%%%%%%%%%%%%%%%%%%%%%%%%%%%%%%%%%%%%%%%%%%%%%%%%%%%%%%%%%%%%%%%%%%%%%%%
\section{Deformation of contours}\label{sec:contours}
%%%%%%%%%%%%%%%%%%%%%%%%%%%%%%%%%%%%%%%%%%%%%%%%%%%%%%%%%%%%%%%%%%%%%%%%%%%%%%%%%%%%%%%%%%%%%%%%%%%%%%%%%%%%%%%%%%%%%%%%%%%%%%%%%%%%%%%%%%%%%%%%%

%Therefore, the relevant points are the zeroes of $U+\nu$.
 In order to apply the steepest descent method, we will choose contours such that the stationary points maximize the real part of the phase $\Re(S_{\lambda}(z))=\Re(Q(z))+\lambda\log(|z|)$ on the  contour instead of $\Gamma_k$ or $\Gamma_N^{\pm}$. %, where
%\[
%\Re(S_{\lambda}'(z))=\Re(z^{-1}(U(z)+\lambda)).
%\]
We therefore study the level sets of $\Re(S_{\lambda})$ in order to deform the contours $\Gamma_k$ and $\Gamma_N^{\pm}$ into more suitable ones in parts~\ref{part:contours_inside} and~\ref{part:contours_outside}. Then we apply the method of steepest descent in parts~\ref{part:steepest_descent} and~\ref{part:steepest_descent_outside}.

%%%%%%%%%%%%%%%%%%%%%%%%%%%%%%%%%%%%%%%%%%%%%%%%%%%%%%%%%%%%%%%%%%%%%%%%%%%%%%%%%%%%%%%%%%%%%%%%%%%%%%%%%%%%%%%%%%%%%%%%%%%%%%%%%%%%%%%%%%%%%%%%%
\subsection{Eigenvalues inside the bulk}\label{part:contours_inside}
%%%%%%%%%%%%%%%%%%%%%%%%%%%%%%%%%%%%%%%%%%%%%%%%%%%%%%%%%%%%%%%%%%%%%%%%%%%%%%%%%%%%%%%%%%%%%%%%%%%%%%%%%%%%%%%%%%%%%%%%%%%%%%%%%%%%%%%%%%%%%%%%%

In this part, we choose a small eigenvalue $\lambda\in\Lambda_-(\delta)$. We justify that we can deform the contours so that every coefficient in the matrix $A$ is suitable for a steepest descent expansion.
For the same equation on the line, this work was the tackled in Proposition 5.1 in~\cite{MillerWetzel2016}.

\begin{mydef}[Suitable contours]\label{def:suitable_W}
The sequence of loops $W_1,\dots,W_N$ is suitable if:
\begin{enumerate}
\item for every $1\leq k\leq N$, the loop $W_k$ starts along the segment $[0,r_ke^{i\theta_k}]$ and finishes along the segment $[r_k e^{i\theta_{k+1}},0]$ for small enough $r_k>0$;
\item when $k\leq N-1$, each contour $W_k$ passes through exactly one critical point $p_k(\lambda)$ of $S_{\lambda}$, and $\Re(S_{\lambda})$ is maximal along $W_k$ exactly at $p_k(\lambda)$;
\item when $k=N$, the contour passes only through the two critical points $p_+(\lambda)$ and $p_-(\lambda)$, and $\Re(S_{\lambda})$ is maximal along $W_N$ exactly at $p_+(\lambda)$ and $p_-(\lambda)$.
\end{enumerate}
\end{mydef}
By construction of the suitable contours, we will get that for every critical point of $S_{\lambda}$ in $\overline{\D}$, there is exactly one contour which passes through this point.

\begin{prop}[Existence of suitable contours]\label{prop:contours}
There exists a family of suitable contours $(W_k)_{1\leq k\leq N}$ according to Definition~\ref{def:suitable_W}, with $W_N=W_N^+\cup W_N^-$, and such that the following holds. Let $B(\lambda;\varepsilon)$ be the matrix with coefficients  given for $1\leq k\leq N-1$ and $1\leq \ell\leq N$ by
\[
B_{k,\ell}=\int_{W_k}e^{Q(\zeta)/\varepsilon}\zeta^{-\ell-1-\lambda/\varepsilon}\dd \zeta,
\]
\[
B_{k,N}^{\pm}=\int_{W_N^{\pm}}e^{Q(\zeta)/\varepsilon}\zeta^{-\ell-1-\lambda/\varepsilon}\dd \zeta,\]
\[
B_{k,N}=B_{k,N}^++e^{-2i\pi\lambda/\varepsilon}B_{k,N}^-.
\]
Then $\det(B(\lambda;\varepsilon))=0$ if and only if $\det(A(\lambda;\varepsilon))=0$.
\end{prop}

The rest of this part is devoted to the proof of Proposition~\ref{prop:contours}.

\paragraph{Description of the level sets}
We first describe the level sets of $\Re(S_{\lambda})$. This description is inspired from Appendix B of~\cite{MillerWetzel2016} for the same equation on the real line. One example of illustration of the level sets of $\Re(S_{\lambda})$ on the complex plane is given in Figure~\ref{fig:level0}.

 We consider the level curve $\Re(S_{\lambda})=L$ for $L$ decreasing from $+\infty$ to $-\infty$. For fixed $L$, we consider that $L$ is the sea level, therefore $\{z\in\C^*\mid\Re(S_{\lambda})<L\}$ will be the part of the landscape under the water and $\{z\in\C^*\mid\Re(S_{\lambda})>L\}$ will be the part above water.

We first note that for $z=e^{ix}\in\partial \D$, there holds $Q(z)\in i\R$, therefore $\Re(S_{\lambda})=0$ on $\partial\D$.
Moreover, in a neighborhood $B(0,r_0)$ of $0$, write $z=re^{ix}$, $0<r<r_0$. Then the main contribution to  $\Re(S_{\lambda})$ is the term 
\[
\Re\left(c_N\frac{z^{-N}}{N}\right)=c_N\frac{r^N}{N}\cos(Nx).
\]
Choosing $r_0$ small enough, one has that $\Re(S_{\lambda})\gg 0$ on the branches $(0,r_0 e^{2ik\pi/N})$ whereas $\Re(S_{\lambda})\ll 0
$ on the branches $(0,r_0 e^{(2k+1)i\pi/N})$.
 As a consequence, inside of $B(0,r_0)$, there is a alternation of lakes and islands when turning around $0$. Let $L\gg 0$  be large enough so that the compact set~$\overline{\D}\setminus B(0,r_0)$ is covered by water. There are isolated infinite mountains which are localized around the directions $(0,r_0 e^{i\widetilde{\theta_k}})$, $\widetilde{\theta_k}=\frac{2k\pi}{N}$, see Figure~\ref{fig:sfig1}. The same happens when $L\ll 0$ is small enough, there are isolated infinite lakes localized around the directions $(0,r_0 e^{i\theta_k})$, $\theta_k=\frac{(2k+1)\pi}{N}$.

%\paragraph{Eigenvalues in the bulk}

\begin{figure}
\begin{center}
\includegraphics[scale=0.4]{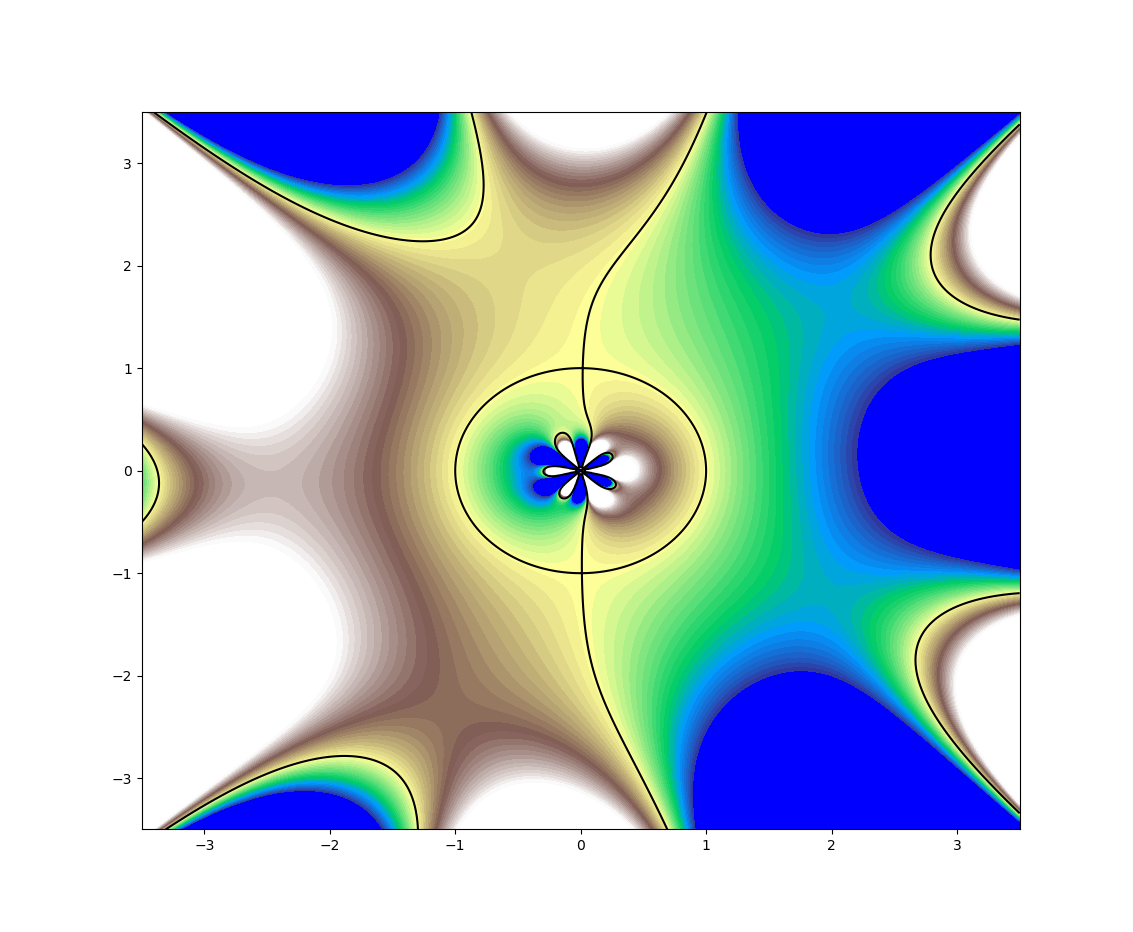}
\caption{Map of the landscape when $u(x)=8\cos(x)+\sin(2x)+\cos(6x)/5$ and $\lambda=0$. The green to dark blue color represents very low values of $\Re(S)$ (lakes) whereas the brown to white color represents very high values of $\Re(S)$ (mountains with snow). The black curve is the zero level set.}\label{fig:level0}
\end{center}
\end{figure}

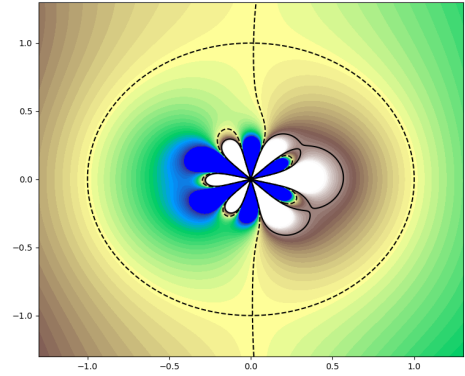
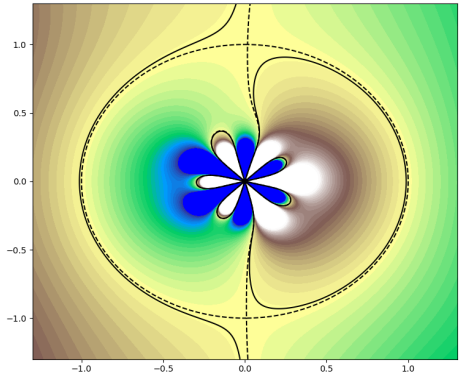
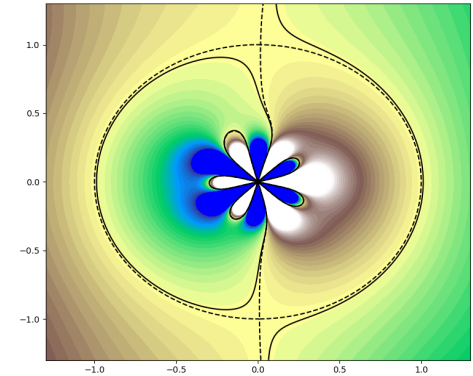
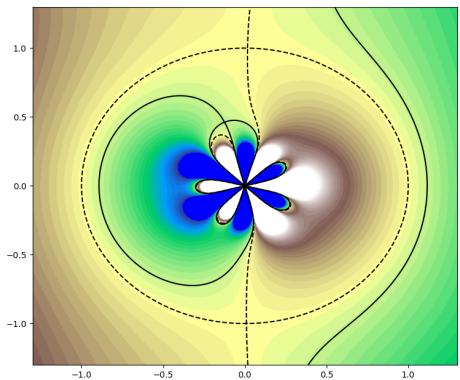
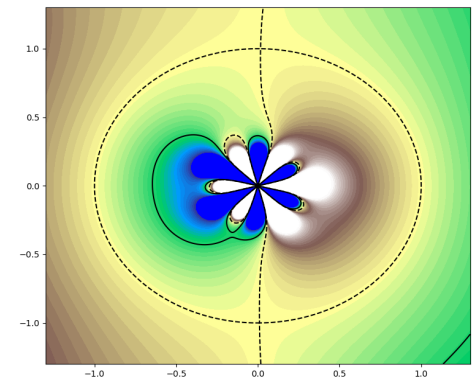
\begin{figure}
\begin{subfigure}{.5\textwidth}
  \centering
  \begin{tikzpicture}
  \node[anchor=south west,inner sep=0] at (0,0) 
  {\includegraphics[width=.8\linewidth]{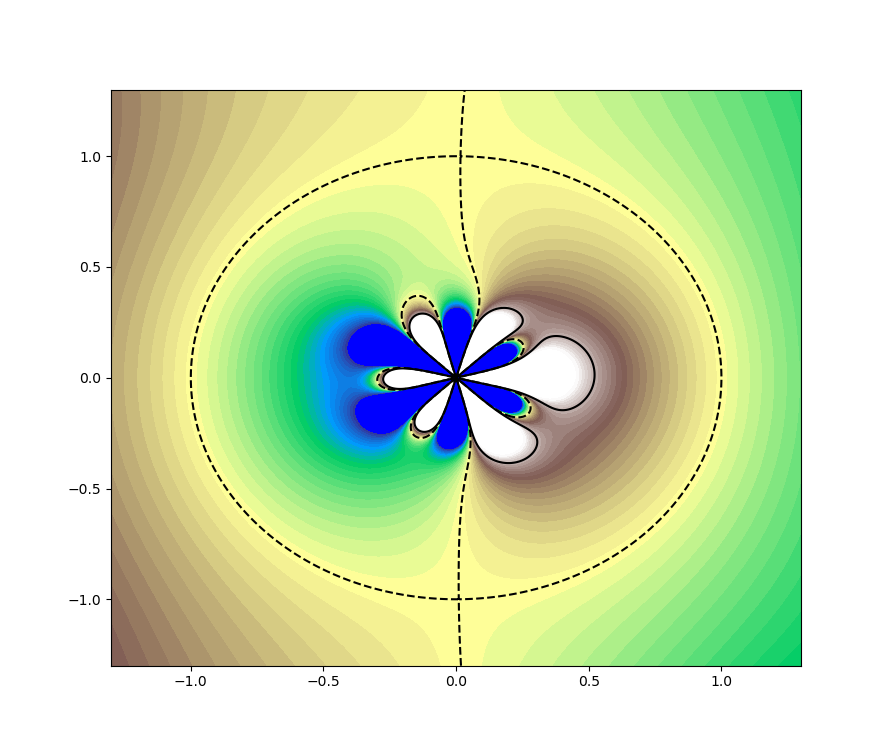}};
  \draw   (5,3.1) node{{ $\ell_6$}};
  \draw   (4.3,3.9) node{$\ell_1$};
  \draw   (3.2,3.9) node{$\ell_2$};
  \draw   (2.6,3) node{$\ell_3$};
  \draw   (3.2,2.2) node{$\ell_4$};
  \draw   (4.3,2.1) node{$\ell_5$};
  \end{tikzpicture}
  \caption{$L>L_1>0$\\ Isolated mountains}
  \label{fig:sfig1}
\end{subfigure}%
\begin{subfigure}{.5\textwidth}
  \centering
  \includegraphics[width=.8\linewidth]{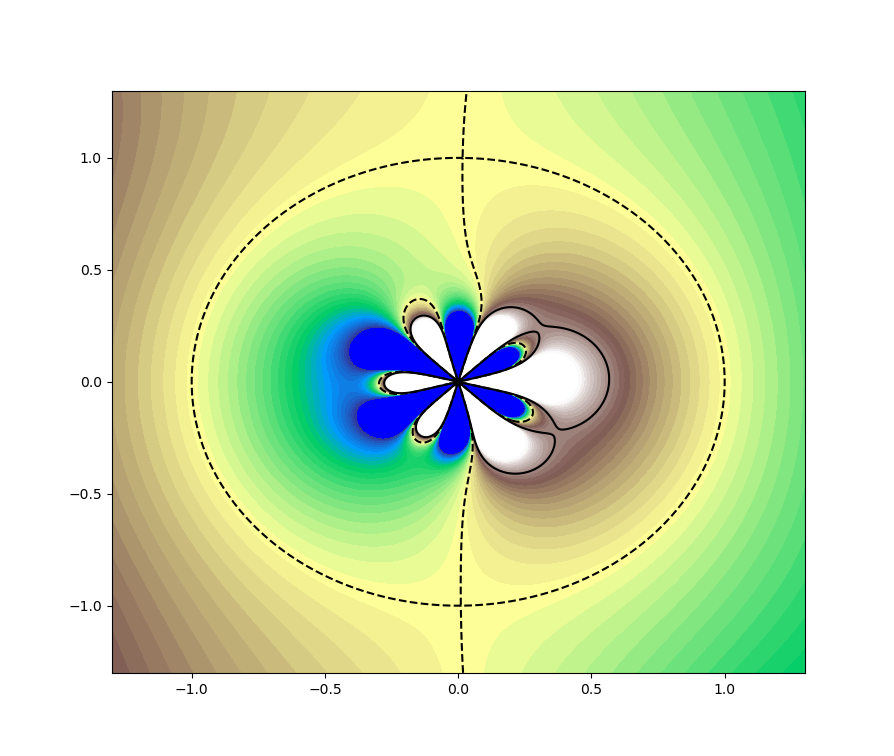}
  \caption{$L_1>L>L_2$\\ Two mountains merge}
%  \label{fig:sfig2}
\end{subfigure}%

\begin{subfigure}{.5\textwidth}
  \centering
  \includegraphics[width=.8\linewidth]{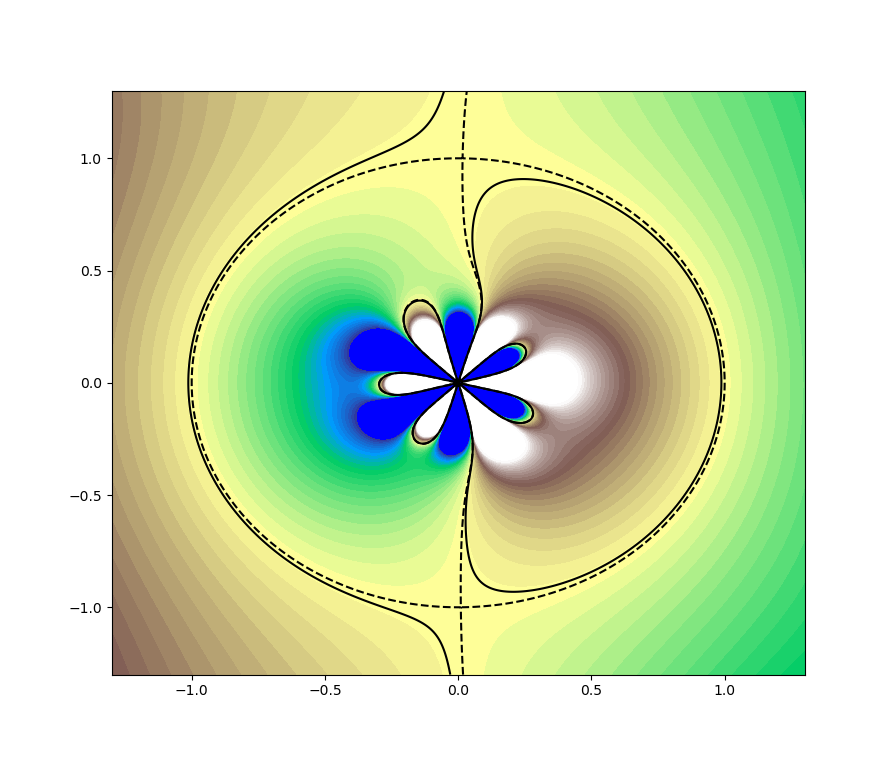}
  \caption{$L_2>L>0$\\ Three mountains have merged}
 % \label{fig:sfig3}
\end{subfigure}
\begin{subfigure}{.5\textwidth}
  \centering
  \includegraphics[width=.8\linewidth]{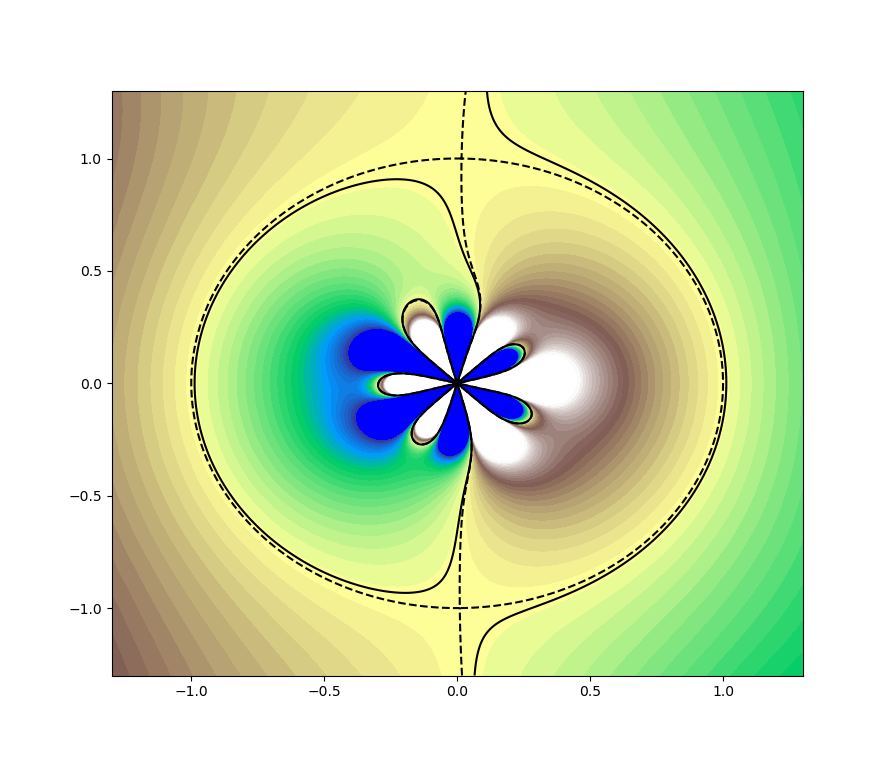}
  \caption{$L_3<L<0$\\ A continent joins $\partial\D$}
  %\label{fig:sfig4}
\end{subfigure}%

\begin{subfigure}{.5\textwidth}
  \centering
  \includegraphics[width=.8\linewidth]{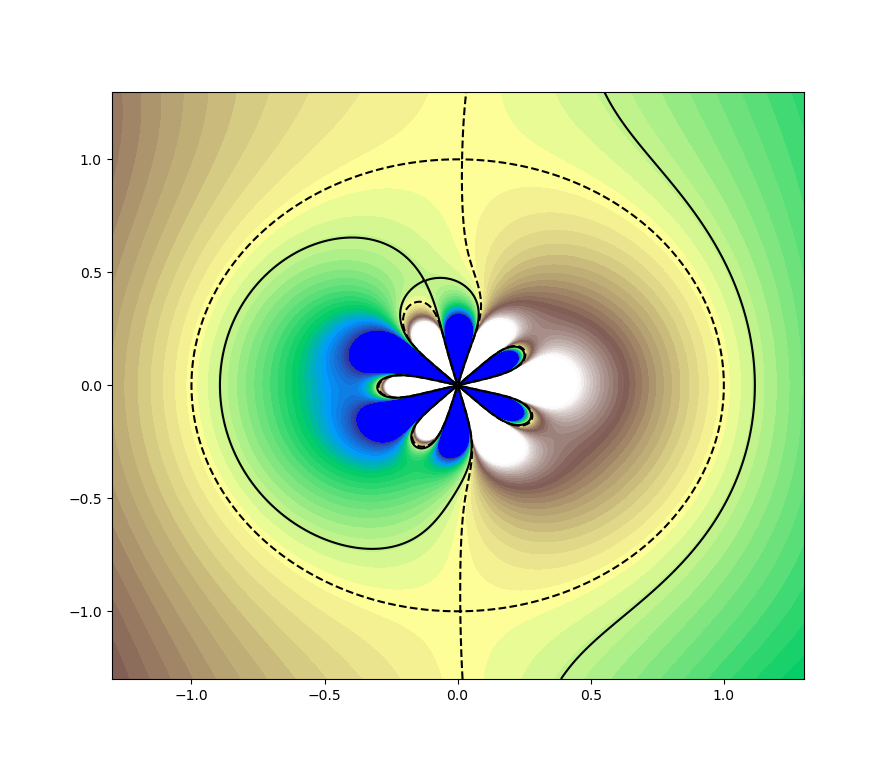}
\caption{$L=L_3<0$\\ Two groups of lakes separate}
  %\label{fig:sfig5}
\end{subfigure}
\begin{subfigure}{.5\textwidth}
  \centering
  \includegraphics[width=.8\linewidth]{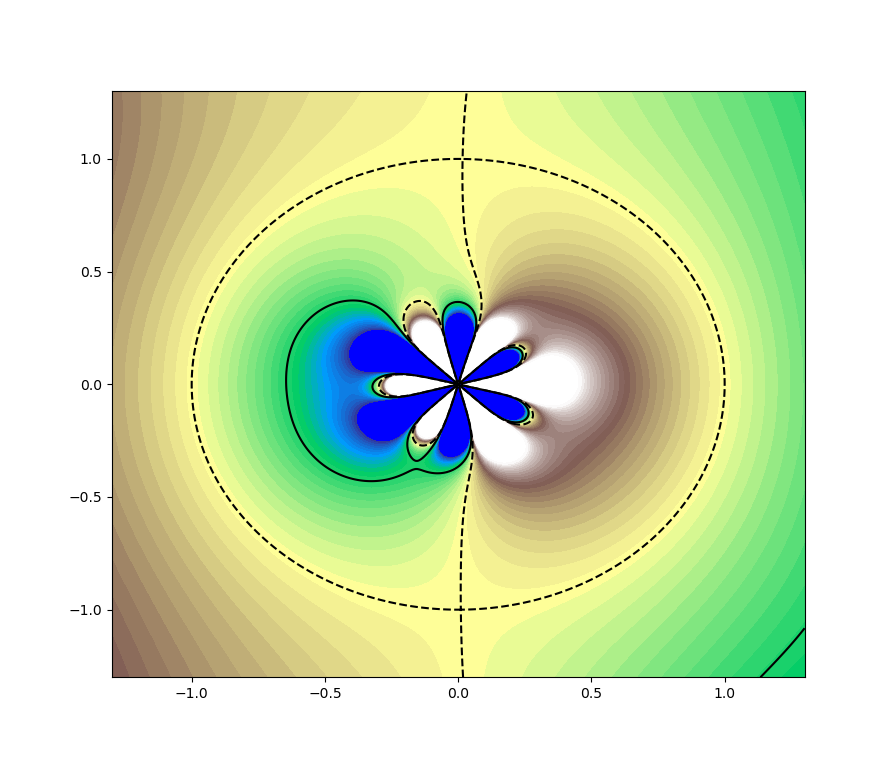}
\caption{$L_4<L<0$\\ Low sea level}
  %\label{fig:sfig6}
\end{subfigure}
\caption{Plots of several level sets of $\Re(S_{\lambda})$. The solid line is the level set $L$, the dotted line is the level set $L_4=0$.}
\label{fig:level_sets}
\end{figure}

We start from a very high sea level $L\gg 0$, where there are isolated infinite mountains localized around $(0,r_0e^{i\widetilde{\theta_k}})$. Under a long drought the sea level will decrease, see Figure~\ref{fig:level_sets}. Therefore at the critical values $L=L_k>0$, some of the infinite mountains will become connected. According to Lemma~\ref{lem:fusion}, the fusion takes place at one of the points $p_k$, $2\leq k\leq N$, and at most two connected groups of mountains fuse at this point. At the critical value $L=0$, one of the groups of mountains becomes connected to a vicinity of the disk in the complex plane, % (actually the connected component of this group of mountain goes to infinity for $L=0^+$),
 totally encircling the remaining lakes in $\D$.  Then, at critical values $L=L_k<0$, groups of mountains fuse together again, and similarly, at every point $p_k$ where this happens, at most two groups mountains become connected. Given that initially there are $N$ infinite pits and $N$ infinite peaks, a counting argument implies that a connected component of mountains does not fuse with itself again, and every infinite mountain still isolated fuses exactly once with a group of mountains in the process. Up to renumbering the critical points, the critical values $L_k$ ordered such that $L_1>\dots>L_{K-1}> L_K=0>L_{K+1}>\dots> L_N$ each correspond to the critical point $p_k$ if $k\neq K$, or to the critical points $p_{\pm}$ if $k=K$.

\paragraph{Tree pruning algorithm}
We use the ``tree pruning algorithm'' construction from Appendix~B of~\cite{MillerWetzel2016} in order to properly describe how the mountains merge. The goal is to show that the steepest descent contours $W_k$ will form a bijection with the original contours $\Gamma_k$.

%We first construct a tree as follows. We denote by $\ell_k$, $1\leq k\leq N$, the infinite pits localized around the lines $(0,2\delta e^{i\theta_k}]$ for small enough $\delta$.
%
%
%We link each path $W_k$ to the lakes that they connected during the fusion in order of increasing height $L$. More precisely, at the level set $L_1$, two lakes $\ell_m,\ell_n$ fuse together through the path $W_1$, so we connect the leaves $\ell_m$ and $\ell_n$ through the internal mode $W_1$ and the two leaves become a group of lakes. Similarly, at the height $L_k$ for $k\neq K$, two groups of lakes fuse together through the path $W_k$, so we connect these groups through the internal mode $W_k$.
%
%At the node $p_{\pm}$, no groups of lakes inside of $\D$ fuse with each other, but one group of lakes becomes connected with a vicinity of $\partial\D$ for $L=0^+$.  % in particular, it becomes connected to $1\in\C$. 
%We put an extra node $p_{\pm}$ on the branch involved in the tree. %, and we link it to a new additional leaf labeled $1$.
%   
%   We get a tree with $N$ leaves $\ell_1,\dots,\ell_N$, $(N-1)$ internal nodes $p_k$ for $k\neq K$ and one extra node $p_{\pm}$.
%Now for every $k\ne K$, we define $S_k$ as the set of leaves that are a descendant of the node $p_k$ in the tree. 

We first construct a tree as follows. We label by $\ell_k$, $1\leq k\leq N$, the infinite mountains localized around the lines $(0,r_0 e^{2ik\pi/N})$ for small $r_0$. The mountains $\ell_k$ will be leaves of the tree, and the critical points $p_k$ will be internal modes.
At the end of the construction, the descendants of each internal mode $p_k$ are the mountains that get connected when the sea level goes from $L_k^+$ to $L_k^-$.

We proceed in order of decreasing sea level $L_k$. At the level set $L_k$ for $k\neq K$, because the sea level decreases, two distinct connected groups of mountains fuse together through the critical point~$p_k$. If the connected group is composed of several mountains, we draw an edge from $p_k$ to the internal mode $p_l$ at which this group became fully connected. If the connected group is a single leaf $\ell_m$, we draw an edge from $p_k$ to this leaf. As a consequence $p_k$ has two direct children.

We note that when $L_k>0$, then the mountains are subsets of $\{z\in\D\mid \Re(S_{\lambda}(z))\geq L_k\}$, therefore a connected group of mountains does not cross the zero level set.
At the level set $L_K=0$, no groups of mountains inside of $\D$ fuse with each other, but one group of mountains becomes connected with a vicinity of $\partial\D$. Before this, this group of mountains either became fully connected at some node $p_l$, or is an isolated mountain $\ell_l$.  % in particular, it becomes connected to $1\in\C$. 
We decide that the  node labeled $p_{\pm}$ has exactly one child which is either the internal mode $p_l$ or to the leaf $\ell_l$. In particular, the leaves that descend from $p_{\pm}$ form a connected component of $\{z\in\D\mid \Re(S_{\lambda}(z))>0\}$ called {\it continent}. Its characteristic is that it becomes connected as $L\to 0^-$ to a vicinity of the level set $\{z\in\C^*\mid \Re(S_{\lambda}(z))>0^-\}$, hence to a vicinity of $\partial\D$ .
 
We get a tree with $N$ leaves $\ell_1,\dots,\ell_N$ and $(N-1)$ internal nodes $p_k$ for $1\leq k\neq N$ and $k\neq K$ with two children each, plus one extra node $p_{\pm}$ with only one child.
   
Now we chop the tree at the node $p_{\pm}$, which becomes both a leaf for one sub-tree and a root for the other. Therefore we have one sub-tree of $M\geq 1$ leaves and $(M-1)$ internal modes $p_k$, rooted at  $p_{\pm}$, and a sub-tree with $(N-M+1)$ leaves including $p_{\pm}$ and $(N-M)$ internal modes.

For the sub-tree of $M$ leaves and then the sub-tree of $(N-M+1)$ leaves, we apply the tree-pruning algorithm. At one step, we only keep a sub-tree of the tree such that each child of a node $p_k$ is represents one of the two groups of mountains that become connected at $p_k$. We obtain this sub-tree by repeating the following operation:
\begin{enumerate}
\item\label{step1} We consider an internal node $p_k$ connected to two leaves. By construction, the two  leaves that come from the two edges in the original sub-tree are the two connected groups of mountains that fuse together when $L=L_k^-$.
\item We remove one of the two leaves $\ell_m$, the internal node $p_k$ and the edge between them. We can always choose this leaf to be different from $p_{\pm}$. Note that at the critical point $p_k$, one could construct a contour staying in the level set $\{z\in\C\mid \Re(S_{\lambda}(z))<L_k\}$ in two ways by encircling either one of the two groups of mountains that are merging. The choice of a leaf $\ell_m$ represents the choice of the group of mountains that we decide to encircle.

\item  We define the set $D_k$ as the set of leaves that have been pruned from the sub-tree up to now, and that are descendant of $p_k$ in the original sub-tree. This represents the set of mountains that are allowed to be encircled in our choice of contour.

\item
Let $p_l$ be the parent of $p_k$, the other leaf which was connected to $p_k$ becomes a leaf connected $p_l$. If $p_l$ does not exist, then the process is over. Otherwise, we search for an internal mode which has $p_l$ as an ancestor and which is connected to two leaves. We apply step~\ref{step1} to this internal mode.
\end{enumerate}

We refer to Figure~\ref{fig:tree} for an example of tree and tree pruning in the context of Figure~\ref{fig:level_sets}.

\begin{figure}
\begin{center}
\begin{tikzpicture}[scale=0.8]
\node {$p_6$}
    child {node {$p_5$} 
		child {node {$p_4$}
			child {node{$p_{\pm}$} 
				child {node{$p_2$}
					child {node{$p_1$}
						child[red] {node[black]{$\ell_1$}}
						child {node {$\ell_6$}}} 
					child[red] {node[black]{$\ell_5$}}}}
			child[red] {node[black]{$\ell_2$}} }
		child[red] {node[black] {$\ell_4$}}}
    child[red] {node[black] {$\ell_3$}};
\end{tikzpicture}
\end{center}
\caption{Construction of a tree from Figure~\ref{fig:level_sets}. The colored edges represent the pairings (leaves and internal modes removed together) obtained during the tree pruning algorithm.}
\label{fig:tree}
\end{figure}
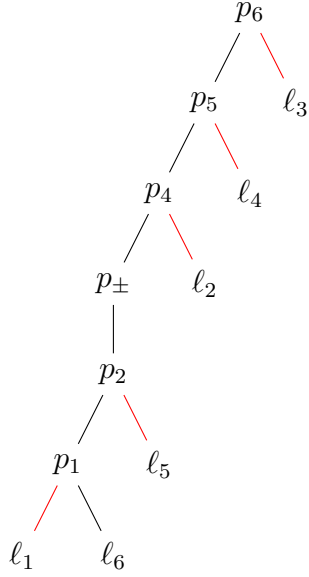

At the end of the process, all the leaves have been pruned, except for one leaf $\ell'_N=\ell_{\sigma(N)}$ for the first sub-tree and the leaf $p_{\pm}$ for the second sub-tree.  Since $\ell'_N$ is a descendant of $p_{\pm}$ in the big tree, it belongs to the ``continent'', in particular, the complex points $p_+
$,  $e^{i\theta_{\sigma(N)}}$ and $p_-$ are placed in this order in the counterclockwise orientation. We choose the determination of the logarithm to have its branch cuts at the line $e^{i\theta_{\sigma(N)}}\R_+$. This implies that we rotate everything and instead of splitting $\Gamma_N=\Gamma_N^{+}\cup\Gamma_N^-$, we split the contour encircling $e^{i\theta_{\sigma(N)}}$ instead.

Moreover, there is a permutation $\ell'_k=\ell_{\sigma(k)}$ of the leaves and a permutation $p'_k=p_{\tau(k)}$ of the critical points with $p'_K=p_N=p_{\pm}$, such that the leaves $\ell'_1,\dots,\ell'_{N-1}$ and internal modes $p'_1,\dots,p'_{N}$ are pruned in this order. We denote $D'_k=D_{\tau(k)}$.

We will construct steepest descent contours $W'_k$ such that the following properties hold.
\begin{itemize}
\item If $k\leq K$, the contour $W'_k$ stays inside the level set $\{z\in\C^*\mid \Re(S_{\lambda}(z))\leq L_k'\}$, it encloses $\ell'_k$ and eventually some other leaves in $D'_k\subseteq\{\ell'_1,\dots,\ell'_{k-1}\}$, and eventually some leaves in $\{\ell'_K,\dots,\ell'_{N-1}\}$. This is possible by construction: at $p_k'$, the mountain at $\ell_k'$ merges with all the leaves in $S_k'$ (and eventually $\ell'_N$), so that they form a connected component of $\{z\in\C^*\mid \Re(S_{\lambda}(z))\leq L_k'\}$. There is always a way not to encircle $\ell'_N$ in our choice of contour.

% Note that since this sub-tree descends from $p_{\pm}$, the mountains that group together form a connected component inside of $\{z\in\C^*\mid\Re(S_{\lambda}(z))>0\}$. Hence they all belong to the same portion of $\D$ delimited by the level set $\{z\in\C^*\mid\Re(S_{\lambda}(z))=0\}$. This is why we can assume that the contour stays inside this portion of $\D$, hence it does not enclose any mountains in the other portion of $\D$, in particular it does not enclose $\{\ell'_K,\dots,\ell'_{N}\}$.
\item if $K+1\leq k\leq N$, the contour $W'_k$ stays inside the level set $\{z\in\D^*\mid \Re(S_{\lambda}(z))\leq L_k'\}$, and encloses $\ell'_k$, plus eventually some other leaves in $S'_k\subseteq\{\ell'_K,\dots,\ell'_{k-1}\}$.
Note that when $k\geq K+1$, then the land is very dry so every group of lakes is isolated inside one portion of $\D$ delimited by $\{z\in\C^*\mid\Re(S_{\lambda}(z))=0\}$. Hence we can assume that the corresponding contour stays inside this portion of $\D$ and does not encircle any leaf in $\{\ell'_1,\dots,\ell'_{K-1}\}$.
\end{itemize}

\paragraph{Choice of steepest descent contours}
It now remains to define the steepest descent contours $W_l$. Let $l\neq N$, and $k\neq K$ such that $W_l=W'_k$.

We recall that if $k\leq K-1$, by construction, the connected group of mountains containing $\ell'_k$ fuses with an other group of mountains when the sea level decreases from $L_l^+$ to $L_l^-$. Moreover, the connected group of mountains containing $\ell'_k$ after fusion (when $L=L_l^-$) is only comprised from mountains that belong to $D_k$.
Similarly, if $K+1\leq k\leq N$,  the connected group of mountains containing $\ell'_k$ fuses with an other group of mountains. The connected group of mountains containing $\ell'_k$ after fusion is comprised from mountains that belong to $D_k$, maybe some mountains in the first sub-tree $\{\ell'_1,\dots,\ell'_{M-1}\}$.

As a consequence, in any of these cases, it is possible to enclose the mountain $\ell'_k$ while enclosing only the other leaves mentioned here by some path. We modify this path a little so that this path is a steepest descent path by taking:
\begin{itemize}
\item a short arc of the steepest decent path at the vicinity of $p_k'$, descending to the level $L_k-\delta_k$, hence staying in $\{z\in \D^*\mid \Re(S_{\lambda}(z))\leq L_k\}$;
\item going from each of the two extremities of the steepest descent path, any path staying inside the set $\{z\in\D^*\mid \Re(S_{\lambda}(z))\leq L_k-\delta_k\}$, and enclosing the mountain $\ell'_k$, then going to a neighborhood of $0$, then going to two lakes $r_k e^{i\theta_1}$ resp. $r_k e^{i\theta_2}$;
\item the segments of the form $[0,r_k e^{i\theta_m}]$ and  $[0,r_k e^{i\theta_n}]$.
\end{itemize}
We orient this loop counterclockwise. The contour does not enclose the leaf $\ell'_N$ by assumption.
% As a consequence, this contour can cross the branch cut $e^{i\theta_{\sigma(N)}}\R_+$ an even number of times, but does not circle around it.

If $k=K$ and $l=N$, then as $L$ goes to $0^+$ to $0^-$, we know that one continent  fuses with the vicinity of the unit disk, whereas one lake localized around $(0,r_0e^{i\theta_{m}})$ becomes separated from it. We define two paths $(W'_{K})^{\pm}=W_N^{\pm}$ so that $W'_K=(W'_K)^+\cup(W'_K)^-$ encircles $\ell_N'$ and a subset of $\{\ell_1',\dots,\ell_{K-1}'\}$, and:
\begin{itemize}
\item a short arc of the steepest descent path at the vicinity of $p_{\pm}$, descending to the level set $\{z\in\C^*\mid \Re(S_{\lambda}(z))\leq -\delta_K\}$;
\item  any path staying inside the set $\{z\in\C^*\mid \Re(S_{\lambda}(z))\leq -\delta_K\}$  and joining the outside extremity of the steepest descent path around $p_{\pm}$ to $e^{i\theta_{\sigma(N)}}$;
\item any path staying inside the set $\{z\in\C^*\mid \Re(S_{\lambda}(z))\leq -\delta_K\}$  and joining the inside extremity of the steepest descent path around $p_{\pm}$ to $r_{\pm} e^{i\theta_{m}}$;
\item the segment $[0,r_{\pm} e^{i\theta_m}]$.
%\item one path staying inside the set $\{z\in\C\mid \Re(S)\leq -\delta_K\}$, connecting the two ends of the steepest decent arc outside of $\D$;
%\item two paths staying inside the set $\{z\in\C\mid \Re(S)\leq -\delta_K\}$, connecting the two ends of the steepest descent arc inside of $\D$ to $\delta_0e^{i\theta_m+\delta'}$ and $\delta_0 e^{i\theta_m-\delta'}$, without crossing the segment $[0,e^{i\theta_m}]$;
%\item  the segments $[0,\delta_0e^{i\theta_m-\delta'}]$ and $[0,\delta_0e^{i\theta_m+\delta'}]$.
\end{itemize} 
We also orient these paths counterclockwise.

\paragraph{Vanishing of the Evans function}

Let $V=(v_1,\dots,v_N)$. We have seen that $AV=0$ if and only if~\eqref{eq:syst1_bis} and~\eqref{eq:syst2} combined with~\eqref{eq:fk_vj} hold, i.e. for $l\leq N-1$,
\begin{equation*}%\label{eq:syst1_bis}
\int_{\Gamma_l}e^{Q(\zeta)/\varepsilon}\sum_{j=1}^Nv_j\zeta^{-j-\lambda/\varepsilon}\frac{\dd\zeta}{\zeta}=0,
\end{equation*}
\begin{equation*}%\label{eq:syst2}
(e^{-2i\pi\lambda/\varepsilon}-1)\int_{\Gamma_{\sigma(N)}^+}e^{Q(\zeta)/\varepsilon}\sum_{j=1}^Nv_j\zeta^{-j-\lambda/\varepsilon-1}\dd\zeta=\int_{\partial \D}z^{-\lambda}e^{Q(z)/\varepsilon}\sum_{j=1}^Nv_jz^{-j-1}\dd z.
\end{equation*}

Since each contour $\Gamma_l$ encloses exactly the mountain $\ell_l$, the $W'_k$ for $k\neq K$ can be written up to deformation (outside of $0$) as a sum of the $\Gamma_{\sigma(l)}$ for $l\neq N$. This deformation being in the domain where the integrand of the oscillatory integrals is holomorphic (see Definition~\ref{def:A}), the deformation does not change the total value of the integral. The coefficients for passing from the two sets of contours $(\Gamma_{\sigma(l)})_{l\neq N}$ to $(W'_k)_{k\neq K}$ will be of the form
\[
P=\left(
\begin{array}{ccc}
L_{K-1} & * \\
0 & L_{N-K} 
\end{array}
\right),
\] 
with $L_{K-1}$ and $L_{N-K}$ being two lower-triangular matrices with indicated dimension with ones on the diagonal. This implies that the matrix $P$ is invertible. As a consequence,  the first system~\eqref{eq:syst1_bis} is equivalent to the fact that for every $k\neq K$,
\begin{equation*}%\label{eq:syst1_bis}
\int_{W_k'}e^{Q(\zeta)/\varepsilon}\sum_{j=1}^Nv_j\zeta^{-j-\lambda/\varepsilon}\frac{\dd\zeta}{\zeta}=0.
\end{equation*}

Finally, the contour $(W_K')^+\cup\Gamma_{\sigma(N)}^+$  when following $\Gamma_{\sigma(N)}^+$ clockwise  is a loop not containing $\ell_N'$. Removing the first equation~\eqref{eq:syst1_bis} with the encircled leaves to the left-hand side of the second equation~\eqref{eq:syst2} for each non-encircled leaves $\ell_k$ in $W_K'$ gives a path that one can deform into $(W_K')^+$ by staying in the holomorphic part of the integrand, whereas $\partial\D$ can be deformed into $(W_K')^+\cup (W_K')^-$. We transform this equation into
\begin{equation*}%\label{eq:syst2}
-(e^{-2i\pi\lambda/\varepsilon}-1)\int_{(W_K')^+}e^{Q(\zeta)/\varepsilon}\sum_{j=1}^Nv_j\zeta^{-j-\lambda/\varepsilon-1}\dd\zeta=\int_{(W_K')^+\cup(W_K')^-}z^{-\lambda}e^{Q(z)/\varepsilon}\sum_{j=1}^Nv_jz^{-j-1}\dd z.
\end{equation*}
This leads to the fact that $AV=0$ if and only if $BV=0$, hence to  Proposition~\ref{prop:contours}.

%\newpage
%%%%%%%%%%%%%%%%%%%%%%%%%%%%%%%%%%%%%%%%%%%%%%%%%%%%%%%%%%%%%%%%%%%%%%%%%%%%%%%%%%%%%%%%%%%%%%%%%%%%%%%%%%%%%%%%%%%%%%%%%%%%%%%%%%%%%%%%%%%%%%%%%
\subsection{Eigenvalues outside the bulk}\label{part:contours_outside}
%%%%%%%%%%%%%%%%%%%%%%%%%%%%%%%%%%%%%%%%%%%%%%%%%%%%%%%%%%%%%%%%%%%%%%%%%%%%%%%%%%%%%%%%%%%%%%%%%%%%%%%%%%%%%%%%%%%%%%%%%%%%%%%%%%%%%%%%%%%%%%%%%

The eigenvalues outside follow the description from Appendix A in~\cite{MillerWetzel2016}. The strategy is similar as part~\ref{part:contours_inside}, but the node $p_{\pm}$ becomes replaced by one of the critical points $p_K$ that plays a special role.

\begin{mydef}[Suitable contours]
The sequence of loops $W_1,\dots, W_N$ is suitable if:
\begin{enumerate}
\item for every $1\leq k\leq N$, the loop $W_k$ starts along the segment $[0,r_ke^{i\theta_k}]$ and finishes along the segment $[r_k e^{i\theta_{k+1}},0]$ for small enough $r_k>0$;
\item each contour $W_k$ for $1\leq k\leq N$ passes through exactly one critical point $p_k(\lambda)$ of $S_{\lambda}$, and $\Re(S_{\lambda})$ is maximized along $W_k$ at $p_k(\lambda)$.
\end{enumerate}
\end{mydef}
As a consequence of this definition, a family of suitable contours satisfies that for every critical point of $S_{\lambda}$ in $\overline{\D}$, there is exactly one contour which passes through this point.

\begin{prop}[Existence of  suitable contours]\label{prop:contours_outside}
There exists a family of suitable contours $(W_k)_{1\leq k\leq N}$ with $W_N=W_N^+\cup W_N^-$ such that the following holds. Let $B(\lambda;\varepsilon)$ be the matrix with coefficients given for $1\leq k\leq N-1$ and $1\leq \ell\leq N$ by
\[
B_{k,\ell}=\int_{W_k}e^{Q(\zeta)/\varepsilon}\zeta^{-\ell-1-\lambda/\varepsilon}\dd\zeta,
\]
\[
B_{k,N}^{\pm}=\int_{W_N^{\pm}}e^{Q(\zeta)/\varepsilon}\zeta^{-\ell-1-\lambda/\varepsilon}\dd\zeta,\]
\[
B_{k,N}=B_{k,N}^++e^{-2i\pi\lambda/\varepsilon}B_{k,N}^-.
\]
Then $\det(B(\lambda;\varepsilon))=0$ if and only if $\det(A(\lambda;\varepsilon))=0$.
\end{prop}

\paragraph{Description of the level sets} Again, we start with a description of the level sets. An illustration of the level sets appears in figure~\ref{fig:level0_outside}.

\begin{figure}
\begin{center}
\includegraphics[scale=0.5]{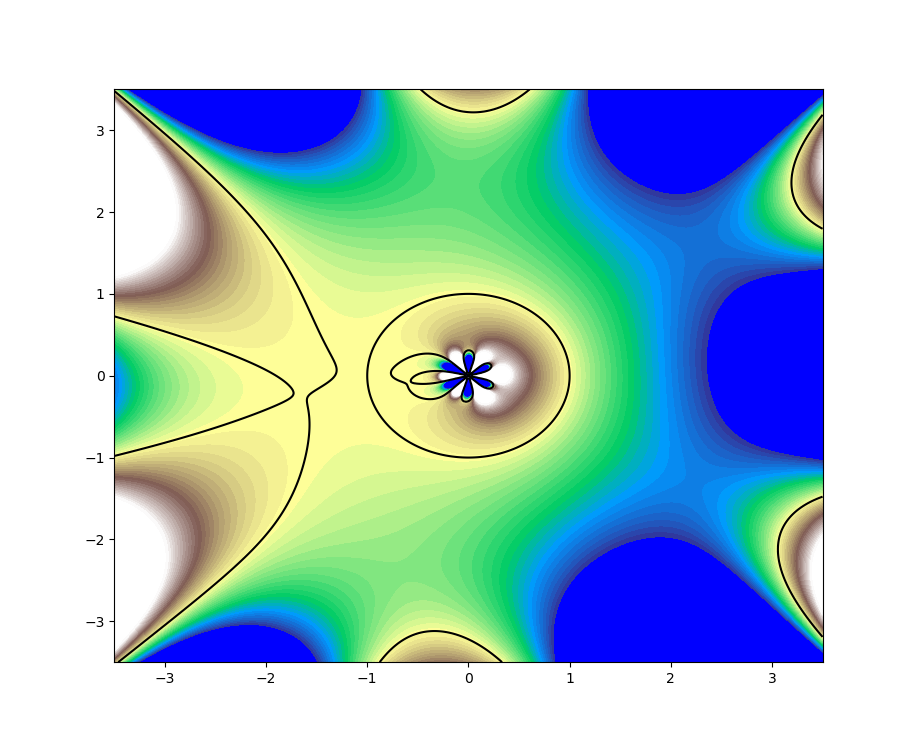}
\caption{Map of the landscape when $u(x)=8\cos(x)-\sin(2x)/2+ \sin(4x)/4+\cos(6x)/10$ and $\lambda=7.95$. The green to dark blue color represents very low values of $\Re(S)$ (lakes) whereas the brown to white color represents very high values of $\Re(S)$ (mountains with snow). The black curve is the zero level set.}\label{fig:level0_outside}
\end{center}
\end{figure}

First, at the sea level $L=0$, we know that $\Re(S_{\lambda})=0$ on the circle $\partial\D$, moreover, since $S_{\lambda}'(z)=-(u+\lambda)/z$, a Taylor expansion of $S_{\lambda}(r z)$ around $z\in\partial\D$ and $r\to 1$ implies that $\Re(S_{\lambda}(z))>0$ for $|z|\to 1^-$ and $\Re(S_{\lambda}(z))<0$ for $|z|\to 1^+$.

When $L\gg 0$, there are isolated infinite mountains localized around the lines $(0,r_0 e^{2ik\pi/N})$ for small enough $\delta$. These mountains are separated by a big ocean connecting 0 to $\infty$. As $L$ decreases, the isolated mountains become connected at critical sea levels $L=L_k$, forming lakes which are not connected to the big ocean. Moreover, there exists one critical sea level $L=L_K$ such that going from $L_K^+$ to $L_K^-$, the point $0$ becomes disconnected from $\infty$. Equivalently, the mountains that merge together at the critical point were already connected before, and at $L=L_K^-$ they encircle $0$, so that they become a {\it continent}. We note that at $L=0$, the circle $\partial\D$ is already part of the continent, so this particular sea level satisfies $L_K>0$. When $L\ll 0$, all the mountains are connected, and there are isolated infinite lakes localized around the lines $(0,r_0 e^{i(2k+1)\pi/N})$ for small enough $r_0$.

By a counting argument, at every level $L_k$, $1\leq k\leq N$, $k\neq K$, two groups of mountains that were not connected before merge together at the node $p_k$, and at some level $L_K$, one group of mountains becomes a continent at the node $p_K$.

We denote the leaves $\ell_1,\dots,\ell_N$ as a label for the mountains localized around the lines $(0,r_0 e^{2ik\pi/N})$ when $L\gg 0$.

\paragraph{Tree pruning algorithm}

We construct a tree as before, for which the node $p_K$ at the level set $L_K$ where one group of mountains becomes a continent replaces the role of the level $p_{\pm}$.

\paragraph{Choice of steepest descent contours}

The choice of contours is also similar except for the path $(W'_K)^{\pm}=W_N^{\pm}$, which both path through $p_K$ but with different orientations. In this case, the ocean, which contains some lake around $e^{i\theta_m}$ when $L=L_K^+$, becomes separated from $\infty$. We take:
\begin{itemize}
\item a short arc of the steepest descent path at the vicinity of $p_K$, descending to the level set $L_K-\delta_K$ for $\delta_K$ small enough so that $L_K-\delta_K>0$;
\item any path staying inside the set $\{z\in\C\mid\Re(S_{\lambda}(z))\leq L_K-\delta_K\}$ and joining the inside extremity of the steepest descent path to $r_K e^{i\theta_m}$ for some $\theta_m$;
\item the segment $[0,r_K e^{i\theta_m}]$;
\item any path staying inside the set $\{z\in\C\mid 0\leq \Re(S_{\lambda}(z))\leq L_K-\delta_K\}$ and joining the outside extremity of the steepest descent path to a point $e^{i\theta}\in\partial\D$;
\item the arc of circle on $\partial\D$ from   $e^{i\theta}$ to $1$ in the clockwise direction if we consider $(W_N')^+$, in the counterclockwise direction if we consider $(W_N')^-$.
\end{itemize}
We choose the determination of the logarithm to be defined in $\C\setminus\R_+$ in this case.

\paragraph{Vanishing of the Evans function}

We recall that since each contour $\Gamma_l$ encloses exactly the mountain $\ell_l$, the steepest descent contours $W'_k$ for $k\neq K$ defined as in part~\ref{part:contours_inside} can be written up to deformation (outside of $0$) as a sum of the $\Gamma_{\sigma(l)}$ for $l\neq N$. This deformation being in the domain where the integrand of the oscillatory integrals is holomorphic, the deformation does not change the total value of the integral, so that the coefficients for passing from the two sets of contours $(\Gamma_{\sigma(l)})_{l\neq N}$ to $(W'_k)_{k\neq K}$ are of the form
\[
P=\left(
\begin{array}{ccc}
L_{K-1} & * \\
0 & L_{N-K} 
\end{array}
\right),
\] 
with $L_{K-1}$ and $L_{N-K}$ being two lower-triangular matrices with ones on the diagonal. This implies that the matrix $P$ is invertible. As a consequence,  the first system~\eqref{eq:syst1_bis} is equivalent to the fact that for every $k\neq K$,
\begin{equation*}%\label{eq:syst1_bis}
\int_{W_k'}e^{Q(\zeta)/\varepsilon}\sum_{j=1}^Nv_j\zeta^{-j-\lambda/\varepsilon}\frac{\dd\zeta}{\zeta}=0.
\end{equation*}

The contour $(W_K')^+\cup\Gamma_{\sigma(N)}^+$  when following $\Gamma_{\sigma(N)}^+$ clockwise is a loop encircling some leaves $\ell_k'$, but not $\ell_N'$.
However adding the first equation~\eqref{eq:syst1_bis} with index $k$ for each leaf $\ell_k$ encircled to the left-hand side of the second equation~\eqref{eq:syst2}, we  get
\begin{equation*}%\label{eq:syst2}
-(e^{-2i\pi\lambda/\varepsilon}-1)\int_{\Gamma_{\sigma(N)}^+}e^{Q(\zeta)/\varepsilon}\sum_{j=1}^Nv_j\zeta^{-j-\lambda/\varepsilon-1}\dd\zeta=\int_{\partial \D}z^{-\lambda}e^{Q(z)/\varepsilon}\sum_{j=1}^Nv_jz^{-j-1}\dd z
\end{equation*}
 gives a path that one can deform into $(W_K')^+$, whereas $\partial\D$ can be deformed into $(W_K')^+\cup (W_K')^-$ without changing the value of the integral. We get
\begin{equation*}%\label{eq:syst2}
-(e^{-2i\pi\lambda/\varepsilon}-1)\int_{(W_K')^+}e^{Q(\zeta)/\varepsilon}\sum_{j=1}^Nv_j\zeta^{-j-\lambda/\varepsilon-1}\dd\zeta=\int_{ (W_K')^+\cup(W_K')^-}z^{-\lambda}e^{Q(z)/\varepsilon}\sum_{j=1}^Nv_jz^{-j-1}\dd z.
\end{equation*}
Hence one can replace the contours $\Gamma_k$ by $W_k'$, which leads to the fact that $AV=0$ if and only if $BV=0$, hence to  Proposition~\ref{prop:contours}.

%\newpage
%%%%%%%%%%%%%%%%%%%%%%%%%%%%%%%%%%%%%%%%%%%%%%%%%%%%%%%%%%%%%%%%%%%%%%%%%%%%%%%%%%%%%%%%%%%%%%%%%%%%%%%%%%%%%%%%%%%%%%%%%%%%%%%%%%%%%%%%%%%%%%%%%
%\section{Asymptotic expansion of the eigenvalues for trigonometric polynomials}
%%%%%%%%%%%%%%%%%%%%%%%%%%%%%%%%%%%%%%%%%%%%%%%%%%%%%%%%%%%%%%%%%%%%%%%%%%%%%%%%%%%%%%%%%%%%%%%%%%%%%%%%%%%%%%%%%%%%%%%%%%%%%%%%%%%%%%%%%%%%%%%%%

%\newpage
%%%%%%%%%%%%%%%%%%%%%%%%%%%%%%%%%%%%%%%%%%%%%%%%%%%%%%%%%%%%%%%%%%%%%%%%%%%%%%%%%%%%%%%%%%%%%%%%%%%%%%%%%%%%%%%%%%%%%%%%%%%%%%%%%%%%%%%%%%%%%%%%%
\subsection{Asymptotic expansion of small eigenvalues}\label{part:steepest_descent}
%%%%%%%%%%%%%%%%%%%%%%%%%%%%%%%%%%%%%%%%%%%%%%%%%%%%%%%%%%%%%%%%%%%%%%%%%%%%%%%%%%%%%%%%%%%%%%%%%%%%%%%%%%%%%%%%%%%%%%%%%%%%%%%%%%%%%%%%%%%%%%%%%
In this part, we  assume that $\lambda\in\widetilde{\Lambda_-}(\delta)$.
Thanks to Proposition~\ref{prop:contours} and Theorem~\ref{thm:evans}, the eigenvalue equation~\eqref{eq:evans1} becomes
\begin{equation*}%\label{eq:evans2}
\det(B(\lambda;\varepsilon))=0.
\end{equation*}
We deduce the asymptotics of the eigenvalues by using the method of steepest descent on every coefficient of the matrix $B(\lambda;\varepsilon)$.

For $1\leq k\leq N-1$ and $1\leq\ell\leq N$, we recall that the coefficients of $B(\lambda;\varepsilon)$ are given in Proposition~\ref{prop:contours} by
%\[
%B_{k,\ell}=\int_{W_k}e^{Q(\zeta)/\varepsilon}\zeta^{-\ell-1-\lambda/\varepsilon}\dd\zeta,
%\]
%\[
%B_{N,\ell}:=B_{N,\ell}^++e^{-2i\pi\lambda}B_{N,\ell}^-,
%\]
%\[
%B_{N,\ell}^{\pm}=\int_{W_N^{\pm}}e^{Q(\zeta)/\varepsilon}\zeta^{-\ell-1-\lambda/\varepsilon}\dd\zeta.
%\]
%We write
\[
B_{k,\ell}
	=\int_{W_k}e^{S_{\lambda}(\zeta)/\varepsilon}\zeta^{-\ell-1}\dd\zeta,
\]
\[
B_{N,\ell}:=B_{N,\ell}^++e^{-2i\pi\lambda}B_{N,\ell}^-,
\]
\[
B_{N,\ell}^{\pm}
	=\int_{W_N^{\pm}}e^{S_{\lambda}(\zeta)/\varepsilon}\zeta^{-\ell-1}\dd\zeta,
\]
with
\[
S_{\lambda}(z)=Q(z)-\lambda\log(z).
\]
Let us apply the method of steepest descent to these oscillatory integrals.

\paragraph{Method of steepest descent}

Let $1\leq\ell\leq N$. For $1\leq k\leq N-1$, we have
\begin{align*}
B_{k,\ell}
	&=\sqrt{\frac{2\pi\varepsilon}{|S_{\lambda}''(p_k)|}}e^{i\theta_k-S_{\lambda}(p_k)/\varepsilon} p_k^{-\ell-1}+R_{k,\ell}(\varepsilon),
\end{align*}
where $\theta_k$ is the steepest descent angle. Similarly, 
\[
B_{N,\ell}^{\pm}
	=\sqrt{\frac{2\pi\varepsilon}{|S_{\lambda}''(p_{\pm})|}}e^{i\theta_N^{\pm}-S_{\lambda}(p_{\pm})/\varepsilon} p_{\pm}^{-\ell-1}+R_{N,\ell}^{\pm}(\varepsilon),
\]
and the steepest descent angles are $\theta_N^+=-\frac{\pi}{4}$ and $\theta_N^-=\frac{\pi}{4}$ since this corresponds to the stationary phase method.
The remainder terms are bounded by
\[
|R_{k,l}(\varepsilon)|
	\leq C_u(\delta)\varepsilon.
\]
%The remainder terms are uniformly bounded (i.e. independently of $u$) if one can restrict the oscillatory integrals to on an compact neighborhood $K_k$, for which $S_{\lambda}$ stays in a bounded set in $\classeC^{4}(K_k)$  of $p_k$ and $|z-p_k|/|S_{\lambda}'(z)|$ has a uniform bound on $K_k$, see~\cite{Hormander2015}, Theorem 7.7.5.

%\begin{itemize}
%\item $\|S_{\lambda}\|_{\classeC^2}\leq C(\delta)$;
%\item $|S''_{\lambda}(x)|\geq \frac{1}{C(\delta)}$ if there exists $k$ such that $|x-p_k(\lambda)|\leq \frac{1}{C'(\delta)}$. \color{purple}il faut faire comonotone de $u''$\color{black}
%\item $|p_k(\lambda)|\geq\frac{1}{C(\delta)}$.
%\end{itemize}

%\paragraph{Remainder term in the determinant}

%We first check that the $\gdO(\varepsilon)$ behaves well with respect to the determinant.

\paragraph{Asymptotic expansion of the Evans function}

Expanding the determinant we deduce that $\det(B(\lambda;\varepsilon))$ has the form
\begin{multline*}
\det(B(\lambda;\varepsilon))
	=(2\pi\varepsilon)^{N/2}\prod_{k=1}^{N-1}\frac{e^{i\theta_k-S_{\lambda}(p_k)/\varepsilon}}{\sqrt{|S_{\lambda}''(p_k)|}}
	\frac{e^{-i\pi/4-S_{\lambda}(p_+)/\varepsilon}}{\sqrt{|S_{\lambda}''(p_+)|}}\det(V_+)\\
	 +(2\pi\varepsilon)^{N/2}\prod_{k=1}^{N-1}\frac{e^{i\theta_k-S_{\lambda}(p_k)/\varepsilon}}{\sqrt{|S_{\lambda}''(p_k)|}}\frac{e^{-2i\pi\lambda+ i\pi/4-S_{\lambda}(p_-)/\varepsilon}}{\sqrt{|S_{\lambda}''(p_-)|}}\det(V_-)
	  +(2\pi\varepsilon)^{N/2} R(\sqrt{\varepsilon}),
\end{multline*}
where $V_{\pm}$ is the Vandermonde matrix with coefficients
\[
(V_{\pm})_{k,\ell}=p_k^{-\ell-1}, \quad p_N=p_{\pm},
\]
and the remainder term satisfies that if $|S_{\lambda}''(p_k)|\geq\frac{1}{C_u(\delta)}$ and $|p_{k}|\leq C_u(\delta)$ for every $1\leq k\leq N-1$ or $k=\pm$, then
\[
|R(\sqrt{\varepsilon})|
	\leq C 2^NN!  (C_u(\delta))^{CN^2} \sqrt{\varepsilon}.
\]

We simplify the Vandermonde determinants
\begin{equation*}
\det(V_{\pm})
	=\prod_{1\leq j<k\leq N}(p_j^{-1}-p_k^{-1})\prod_{k=1}^Np_k^{-2}
	=\prod_{1\leq j<k\leq N}(p_k-p_j)\prod_{k=1}^Np_k^{-(N+1)}.
\end{equation*}
We deduce that
\begin{equation}\label{eq:B_dev}
\det(B(\lambda;\varepsilon))
	=(2\pi\varepsilon)^{N/2}\left(P(\lambda;\varepsilon)
	\Delta(\lambda;\varepsilon)+R(\sqrt{\varepsilon})\right),
\end{equation}
\begin{equation}\label{eq:P}
P(\lambda;\varepsilon)
	=\frac{e^{-i\pi/4-S_{\lambda}(p_+)/\varepsilon}}{\sqrt{|S_{\lambda}''(p_+)|}}\det(V_+)\prod_{k=1}^{N-1}\frac{e^{i\theta_k-S_{\lambda}(p_k)/\varepsilon}}{\sqrt{|S_{\lambda}''(p_k)|}},
\end{equation}
\[
\Delta(\lambda;\varepsilon) =
		1 + e^{-2i\pi\lambda+i\pi/2-(S_{\lambda}(p_-)-S_{\lambda}(p_+))/\varepsilon}\frac{\sqrt{|S_{\lambda}''(p_+)|}}{\sqrt{|S_{\lambda}''(p_-)|}}\frac{p_-^{-(N+1)}}{p_+^{-(N+1)}}\prod_{k=1}^{N-1}\frac{p_k-p_-}{p_k-p_+}.
\]
Using Lemma~\ref{lem:S''}, we see that the quotient $\frac{\sqrt{|S_{\lambda}''(p_+)|}}{\sqrt{|S_{\lambda}''(p_-)|}}\frac{p_-^{-(N+1)}}{p_+^{-(N+1)}}\prod_{k=1}^{N-1}\frac{p_k-p_-}{p_k-p_+}$ has modulus $1$, hence there exists $\psi(\lambda)\in\T$ such that
\begin{equation}\label{eq:Delta}
\Delta(\lambda;\varepsilon)=1-e^{2i\pi\psi(\lambda)-(S_{\lambda}(p_-)-S_{\lambda}(p_+))/\varepsilon}.
\end{equation}
Since $\arg(p_+)-\arg(p_-)=2\pi F(\lambda)$, the angle $\psi(\lambda)$ is equal to
\begin{equation}\label{eq:psi}
\psi(\lambda)=-\frac{1}{4}-\lambda+(N+1) F(\lambda)+\frac{1}{2\pi}\sum_{k=1}^{N-1}\arg(p_-(\lambda)-p_k(\lambda))-\arg(p_+(\lambda)-p_k(\lambda)).
\end{equation}
In this case, $\psi$ is a $C_u(\delta)$-Lipschitz function of $\lambda$ when the critical points are on a compact set of values for which the $p_k$ have single multiplicity, for instance when $\lambda\in\widetilde{\Lambda_-}(\delta)\cup\widetilde{\Lambda_+}(\delta)$. 
%In the case of a general trigonometric polynomial, this property will still be true if we remove a vicinity of size $\delta$ from the set of eigenvalues $\lambda$ for which at least one root $p_k$ is not a simple root. 
In this case, we get smoothness of the simple roots $p_k,p_+,p_-$ along the parameter $\lambda$, so that $\psi$ is a Lipschitz function of $\lambda$. 
Note that when $u$ is even, then $p_+=\overline{p_-}$, moreover $p_k$ and $\overline{p_k}$ both appear in the sum for every $k$, so that
\begin{equation}\label{eq:psi_even}
\psi(\lambda)=-\frac{1}{4}-\lambda+(N+1) F(\lambda).
\end{equation}

We now estimate $P(\lambda;\varepsilon)$. We know that by definition, $S_{\lambda}''(p_k)=u'(p_k)$, so
\[
|S_{\lambda}''(p_k)|
	=c_N|p_k|^{-N}|p_+-p_k||p_--p_k|\prod_{\substack{j=1\\j\neq k}}^{N-1}|p_k-p_j| \prod_{\substack{j=1}}^{N-1}|p_k-1/\overline{p_j}|.
\]
Using Lemma~\ref{lem:S''} for $p_+$, we deduce that
\begin{multline*}
|S_{\lambda}''(p_+)|\prod_{k=1}^{N-1}|S_{\lambda}''(p_k)|\\
	=c_N^N|p_+-p_-|\prod_{k=1}^{N-1}\frac{|p_+-p_k|^3|p_--p_k|}{|p_k|^{N+1}}\prod_{1\leq j\neq k\leq N-1}|p_k-p_j|\prod_{1\leq j,k\leq N-1}|p_k-1/\overline{p_j}|
\end{multline*}
belongs to $\left[\frac{1}{C_u(\delta)},C_u(\delta)\right]$. Similarly, $|\det(V_{\pm})|\in \left[\frac{1}{C_u(\delta)},C_u(\delta)\right].$
As a consequence,
\[
|P(\lambda;\varepsilon)|\in \left[\frac{1}{C_u(\delta)},C_u(\delta)\right].
\]
%We conclude that $|P(\lambda;\varepsilon)|\geq 1/C(\delta)$.

The eigenvalue equation $\det(B(\lambda;\varepsilon))=0$ combined with expression~\eqref{eq:B_dev} therefore becomes
\[
\frac{S_{\lambda}(p_+)-S_{\lambda}(p_-)}{i\varepsilon}=2\pi\left(\psi(\lambda)+n+R_n'(\sqrt{\varepsilon})\right),
\]
\[
|R_n' (\sqrt{\varepsilon})|\leq C_u(\delta)\sqrt{\varepsilon},\quad n\in\Z.\]
Finally, Lemma~\ref{lem:aire}
\(
S_{\lambda}(p_+)-S_{\lambda}(p_-)
	=2i\pi\int_{-\lambda}^{\max(u)}F(\nu)\dd\nu
\)
implies that
\begin{equation}\label{eq:small-eigenvalues}
\int_{-\lambda}^{\max(u)}F(\nu)\dd\nu=\varepsilon(n+\psi(\lambda))+R'(\varepsilon\sqrt{\varepsilon}),
\end{equation}
\[
|R'(\varepsilon\sqrt{\varepsilon})|\leq C_u(\delta)\varepsilon\sqrt{\varepsilon}.
\]
%This approach is still valid with a non-uniform remainder term in the more general case $\lambda\in\Lambda_-(\delta)$ such that the critical points $p_k(\lambda)$ have single multiplicity.

%\newpage
%%%%%%%%%%%%%%%%%%%%%%%%%%%%%%%%%%%%%%%%%%%%%%%%%%%%%%%%%%%%%%%%%%%%%%%%%%%%%%%%%%%%%%%%%%%%%%%%%%%%%%%%%%%%%%%%%%%%%%%%%%%%%%%%%%%%%%%%%%%%%%%%%
%\section{Asymptotic expansion of the eigenvalues outside the bulk for trigonometric polynomials}
%%%%%%%%%%%%%%%%%%%%%%%%%%%%%%%%%%%%%%%%%%%%%%%%%%%%%%%%%%%%%%%%%%%%%%%%%%%%%%%%%%%%%%%%%%%%%%%%%%%%%%%%%%%%%%%%%%%%%%%%%%%%%%%%%%%%%%%%%%%%%%%%%

%%%%%%%%%%%%%%%%%%%%%%%%%%%%%%%%%%%%%%%%%%%%%%%%%%%%%%%%%%%%%%%%%%%%%%%%%%%%%%%%%%%%%%%%%%%%%%%%%%%%%%%%%%%%%%%%%%%%%%%%%%%%%%%%%%%%%%%%%%%%%%%%%
%\subsection{Description of the level sets}
%%%%%%%%%%%%%%%%%%%%%%%%%%%%%%%%%%%%%%%%%%%%%%%%%%%%%%%%%%%%%%%%%%%%%%%%%%%%%%%%%%%%%%%%%%%%%%%%%%%%%%%%%%%%%%%%%%%%%%%%%%%%%%%%%%%%%%%%%%%%%%%%%

%\paragraph{Eigenvalues outside the bulk}
%
%
%In particular, $L_k<0$ for every $k$.

%\newpage
%%%%%%%%%%%%%%%%%%%%%%%%%%%%%%%%%%%%%%%%%%%%%%%%%%%%%%%%%%%%%%%%%%%%%%%%%%%%%%%%%%%%%%%%%%%%%%%%%%%%%%%%%%%%%%%%%%%%%%%%%%%%%%%%%%%%%%%%%%%%%%%%%
\subsection{Asymptotic expansion of large eigenvalues}\label{part:steepest_descent_outside}
%%%%%%%%%%%%%%%%%%%%%%%%%%%%%%%%%%%%%%%%%%%%%%%%%%%%%%%%%%%%%%%%%%%%%%%%%%%%%%%%%%%%%%%%%%%%%%%%%%%%%%%%%%%%%%%%%%%%%%%%%%%%%%%%%%%%%%%%%%%%%%%%%

We now focus on the case when $\lambda\in\widetilde{\Lambda_+}(\delta)\subseteq[-\min(u)+\delta,K(\delta)]$.
Thanks to Proposition~\ref{prop:contours_outside} and Theorem~\ref{thm:evans}, the eigenvalue equation~\eqref{eq:evans1} becomes
\begin{equation*}%\label{eq:evans2}
\det(B(\lambda;\varepsilon))=0.
\end{equation*}
We recall that using Proposition~\ref{prop:contours_outside}, the coefficients of $B$ are given for $1\leq k\leq N-1$ and $1\leq\ell\leq N$ by
\[
B_{k,\ell}
	%=\int_{W_k}e^{Q(\zeta)/\varepsilon}\zeta^{-\ell-1-\lambda/\varepsilon}\dd\zeta
	=\int_{W_k}e^{S_{\lambda}(\zeta)/\varepsilon}\zeta^{-\ell-1}\dd\zeta,
\]
\[
B_{N,\ell}:=B_{N,\ell}^++e^{-2i\pi\lambda}B_{N,\ell}^-,
\]
\[
B_{N,\ell}^{\pm}
	%=\int_{W_N^{\pm}}e^{Q(\zeta)/\varepsilon}\zeta^{-\ell-1-\lambda/\varepsilon}\dd\zeta
	=\int_{W_N^{\pm}}e^{S_{\lambda}(\zeta)/\varepsilon}\zeta^{-\ell-1}\dd\zeta.
\]
We deduce the asymptotics of the eigenvalues by using the method of steepest descent on every coefficient of the matrix $B$: for $1\leq k\leq N-1$ and $1\leq\ell\leq N$,
\begin{equation*}
B_{k,\ell}	
	=\sqrt{\frac{2\pi\varepsilon}{|S_{\lambda}''(p_k)|}}e^{i\theta_k-S_{\lambda}(p_k)/\varepsilon}p_k^{-\ell-1} +R_{k,\ell}(\varepsilon),
\end{equation*}
where $\theta_k$ is the steepest descent angle, and for $k=N$ and $1\leq \ell\leq N$, we have
\[
B_{N,\ell}
	=-\left(1-e^{-2i\pi\lambda}\right)
	\sqrt{\frac{2\pi\varepsilon}{|S_{\lambda}''(p_{N})|}}e^{i\theta_N-S_{\lambda}(p_N)/\varepsilon}p_N^{-\ell-1}+R_{N,\ell}(\varepsilon),
\]
moreover, the remainder terms are bounded by
\[
|R_{k,\ell}(\varepsilon)|
	\leq C_u(\delta)\varepsilon.
\]

The determinant of $B$ therefore admits the asymptotic expansion
\[
\det(B(\lambda;\varepsilon))=(2\pi\varepsilon)^{N/2}\left(1-e^{-2i\pi\lambda}\right)
\left(\prod_{k=1}^N\frac{e^{i\theta_k-iS_{\lambda}(p_k)/\varepsilon}}{\sqrt{|S_{\lambda}''(p_k)|}}\right)\det(V)+(2\pi\varepsilon)^{N/2} R(\sqrt{\varepsilon}),
\]
where $V$ is the Vandermonde matrix with coefficients $V_{k,\ell}=p_k^{-\ell-1}$. The determinant of $V$ is equal to
\[
\det(V)=\prod_{1\leq j<k\leq N}(p_j^{-1}-p_k^{-1})\prod_{k=1}^Np_k^{-2}=\prod_{1\leq j<k\leq N}(p_j-p_k)\prod_{k=1}^Np_k^{-(N+1)},
\]
therefore it does not vanish.  More precisely,  if $|S_{\lambda}''(p_k)|\geq\frac{1}{C_u(\delta)}$, $|p_{k}|\leq C_u(\delta)$ and $|p_{k}-p_{\ell}|\geq \frac{1}{C_u(\delta)}$ for every $k,\ell\in\{1,\dots,N-1,+,-\}$ and $k\neq \ell$, then
\[
\left(\prod_{k=1}^N\frac{1}{\sqrt{|S_{\lambda}''(p_k)|}}\right)|\det(V)|\geq \frac{1}{C_u(\delta)^{CN^2}}
\]
whereas
\[
|R(\sqrt{\varepsilon})|
	\leq C 2^NN!  (C_u(\delta))^{CN^2} \sqrt{\varepsilon}.
\]
We deduce that for $\varepsilon$ small enough,
\[
\left|\left(\prod_{k=1}^N\frac{e^{i\theta_k-S_{\lambda}(p_k)/\varepsilon}}{\sqrt{|S_{\lambda}''(p_k)|}}\right)\det(V)+R(\sqrt{\varepsilon})\right|\geq \frac{1}{C'_u(\delta)}.
\]
As a consequence, $\det(B(\lambda;\varepsilon))=0$ implies that $2\pi\lambda/\varepsilon=2\pi(R'(\sqrt{\varepsilon})+n)$, or 
\begin{equation}\label{eq:large-eigenvalues}
\lambda=n\varepsilon+\varepsilon R'(\sqrt{\varepsilon}),
\end{equation}
\[
|R'(\sqrt{\varepsilon})|\leq C_u(\delta)\sqrt{\varepsilon}.
\]

%\newpage
%%%%%%%%%%%%%%%%%%%%%%%%%%%%%%%%%%%%%%%%%%%%%%%%%%%%%%%%%%%%%%%%%%%%%%%%%%%%%%%%%%%%%%%%%%%%%%%%%%%%%%%%%%%%%%%%%%%%%%%%%%%%%%%%%%%%%%%%%%%%%%%%%
\section{Zero-dispersion limit for even trigonometric polynomials}\label{sec:zero}
%%%%%%%%%%%%%%%%%%%%%%%%%%%%%%%%%%%%%%%%%%%%%%%%%%%%%%%%%%%%%%%%%%%%%%%%%%%%%%%%%%%%%%%%%%%%%%%%%%%%%%%%%%%%%%%%%%%%%%%%%%%%%%%%%%%%%%%%%%%%%%%%%

We finally show that when $u$ is an even weakly bell shaped trigonometric polynomial (see Definition~\ref{def:bell shaped}), then the asymptotic expansion of the Lax eigenvalues is precise enough to determine the weak limit of solutions as $\varepsilon\to 0$. We rely on the fact that the difference of consecutive phase constants are multiple of $\pi$ when $u$ is even, see part~\ref{appendix:phase}.%Therefore we prove Theorem~\ref{thm:zero}.

First, in part~\ref{part:phase} we show that the phase constants satisfy a weak form of convergence 
\[
e^{i(\theta_{n+1}-\theta_n)(u;\varepsilon)}\approx e^{i\frac{x_+(-\lambda_n)+x_-(-\lambda_n)}{2}+i\pi}=-1,
\]
%``$e^{i(\theta_{n+1}(u;\varepsilon)-\theta_n(u;\varepsilon))}\to -1$''. 
As this form of convergence is weaker than the required assumptions in Definition~1.5 of~\cite{bo_zero}, in part~\ref{part:time} we finally explain how to adapt the general strategy in order to prove our main Theorem~\ref{thm:zero}.

%\newpage
%%%%%%%%%%%%%%%%%%%%%%%%%%%%%%%%%%%%%%%%%%%%%%%%%%%%%%%%%%%%%%%%%%%%%%%%%%%%%%%%%%%%%%%%%%%%%%%%%%%%%%%%%%%%%%%%%%%%%%%%%%%%%%%%%%%%%%%%%%%%%%%%%
\subsection{Phase constants}\label{part:phase}
%%%%%%%%%%%%%%%%%%%%%%%%%%%%%%%%%%%%%%%%%%%%%%%%%%%%%%%%%%%%%%%%%%%%%%%%%%%%%%%%%%%%%%%%%%%%%%%%%%%%%%%%%%%%%%%%%%%%%%%%%%%%%%%%%%%%%%%%%%%%%%%%%

%In order to satisfy Definition~1.5 in~\cite{bo_zero}, it remains to establish a uniform bound on the approximation of the phase term $e^{i(\theta_{n+1}(u)-\theta_n(u))}$ by the phase term of the truncated Fourier series. For now, we are only able to prove a weaker bound in the special case  when $u$ is even, using the fact that for every $n$, there holds $e^{i(\theta_{n+1}-\theta_n)(u;\varepsilon)}=\pm 1$ (see Theorem~\ref{thm:phase}).

%Our goal is to prove that when $u$ is an even bell shaped potential, then the phase constants satisfy a weak form of convergence moreover, we establish an upper bound on the error term.

Let us first estimate the phase constants in the zero-dispersion limit.
\begin{prop}[Phase constants in the zero-dispersion limit]\label{prop:phase-zero}
Let $0<c<r<\frac 13$, and define the ``upside down'' terms as
\[
J:=\left\{n\geq 1\mid \lambda_n+\varepsilon\in\Lambda_-(\delta),\quad   e^{i(\theta_{n+1}-\theta_{n})(u;\varepsilon)}=1\right\}.
\]
Then for $\varepsilon<\varepsilon_0(\delta)$, there holds
\[
\varepsilon\sum_{n\in J} \sinc(\pi F(-\lambda_n))
	\leq C(\delta)\varepsilon^{r-c}+C\delta.
\]
\end{prop}
Note that this series is only comprised of nonnegative terms since $F(-\lambda_n)\in [0,1]$, so that the general term of the series satisfies $\sinc(\pi F(-\lambda_n))\in [0,1]$.

\begin{proof}
In proof of Theorem~3.9 in~\cite{bo_zero}, one can see that inequality~(17) is still valid in the case $k=1$. This implies that for some fixed choice of constants $0<c<r<1/3$,
\begin{equation*}%\label{eq:(17)}
\left|\widehat{u_0}(1)
	-\varepsilon\sum_{\substack{n_2=n_1\\\lambda_{n_1}+\varepsilon\in\Lambda_-(\delta)}} \sinc(\pi F(-\lambda_{n_1})) \frac{F(-\lambda_{n_1})e^{i(\theta_{n_1+1}-\theta_{n_1})}}{n_{2}-n_1-F(-\lambda_{n_1})}\right|
	\leq C(\delta)\varepsilon^{r-c}+C\delta,
\end{equation*}
which simplifies as
\begin{equation}\label{eq:(17)}
\left|\widehat{u_0}(1)
	-\varepsilon\sum_{\substack{n\geq 1\\ \lambda_n+\varepsilon\in\Lambda_-(\delta)}} \sinc(\pi F(-\lambda_{n})) (-1)e^{i(\theta_{n+1}-\theta_{n})(u;\varepsilon)}\right|
	\leq C(\delta)\varepsilon^{r-c}+C\delta.
\end{equation}
We now divide the sum between the indices for which $e^{i(\theta_{n+1}-\theta_{n})}=-1$ and $e^{i(\theta_{n+1}-\theta_{n})}=1$.

According Theorem~3.9 in~\cite{bo_zero}, we already know that the {\it admissible} the initial data $u_0^{\varepsilon}$  chosen such that for every $n$, $\lambda_n(u_0^{\varepsilon};\varepsilon)=\lambda_n(u_0;\varepsilon)$ and $e^{i(\theta_{n+1}-\theta_{n})(u_0^{\varepsilon};\varepsilon)}=-1$ satisfies
\begin{equation}\label{eq:3.13}
\left|\widehat{u_0^{\varepsilon}}(1)-\varepsilon\sum_{\substack{n\geq 1\\ \lambda_n+\varepsilon\in\Lambda_-(\delta)}} \sinc(\pi F(-\lambda_n))\right|
	\leq C(\delta)\varepsilon^{r-c}+R(\varepsilon)+C\delta.
\end{equation}
Moreover, Theorem~3.14 from~\cite{bo_zero} at $t=0$  implies that there exists $\varepsilon_0(\delta)>0$ such that if $0<\varepsilon<\varepsilon_0(\delta)$, then
\[
\left|\widehat{u_0^{\varepsilon}}(1)+\frac{i}{2\pi}\int_{\min(u_0)}^{\max(u_0)}e^{-ix_+(\eta)}-e^{-ix_-(\eta)}\dd\eta\right|\leq C\delta.
\]
We recognize $\widehat{u_0}(1)$ in the second term of the left-hand side, so that \begin{equation}\label{eq:u0eps-u_0}
\left|\widehat{u_0^{\varepsilon}}(1)-\widehat{u_0}(1)\right|
	\leq C \delta.
\end{equation}
We conclude from~\eqref{eq:(17)},~\eqref{eq:3.13} and~\eqref{eq:u0eps-u_0} that when $0<\varepsilon<\varepsilon_0(\delta)$, then the indices $n\in J$ for which $(-1)e^{i(\theta_{n+1}-\theta_n)(u;\varepsilon)}\neq 1$ satisfy
\[
2\varepsilon\sum_{\substack{n\in J}} \sinc(\pi F(-\lambda_{n})) e^{i(\theta_{n+1}-\theta_{n})(u;\varepsilon)}\leq C(\delta)\varepsilon^{r-c}+R(\varepsilon)+C\delta.\qedhere
\]
\end{proof}

%%%%%%%%%%%%%%%%%%%%%%%%%%%%%%%%%%%%%%%%%%%%%%%%%%%%%%%%%%%%%%%%%%%%%%%%%%%%%%%%%%%%%%%%%%%%%%%%%%%%%%%%%%%%%%%%%%%%%%%%%%%%%%%%%%%%%%%%%%%%%%%%%
\subsection{Time evolution}\label{part:time}
%%%%%%%%%%%%%%%%%%%%%%%%%%%%%%%%%%%%%%%%%%%%%%%%%%%%%%%%%%%%%%%%%%%%%%%%%%%%%%%%%%%%%%%%%%%%%%%%%%%%%%%%%%%%%%%%%%%%%%%%%%%%%%%%%%%%%%%%%%%%%%%%%

%First, we note that part 3.2 in~\cite{bo_zero} does not rely on the refined asymptotic expansions of the Lax eigenvalues, but only on the estimates from Lemma~\ref{thm:lax_u}. The only difference is the upper bounds from Corollary 2.2 in~\cite{bo_zero} (originally of the form $C(\delta)\varepsilon\sqrt{\varepsilon}+C\delta$), which have to be relaxed into upper bounds of the form $C_N(\delta)\varepsilon\sqrt{\varepsilon}+C\delta+\delta_N$ from Lemma~\ref{thm:lax_u}. Moreover, 

We now proceed to the proof of Theorem~3.9 in~\cite{bo_zero}, where we replace the solution to~\eqref{eq:bo} with approximate initial data $u_0^{\varepsilon}$ by the solution to~\eqref{eq:bo} with the initial data $u_0$ itself. Since we did not prove that the approximation $e^{i(\theta_{n+1}-\theta_n)(u_0;\varepsilon)}\approx -1$ is valid in the sense of Definition 1.5 in~\cite{bo_zero},  one should instead make use of Proposition~\ref{prop:phase-zero} for the rest of this proof. Theorem 3.9 in~\cite{bo_zero}, taking into account the time evolution, becomes as follows.
\begin{prop}[Theorem 3.9 in~\cite{bo_zero}, with weaker phase constant assumption and time evolution]\label{prop:M_ACV} 
Let $0<r<1$ and $T>0$. There exists $C(\delta),\varepsilon_0(\delta,T)>0$ such that for every $0<\varepsilon<\varepsilon_0(\delta,T)$,
\begin{equation*}
\left|\widehat{u^{\varepsilon}(t)}(k)
	- \varepsilon\sum_{\substack{n\geq 1\\\lambda_n+\varepsilon\in\Lambda_-(\delta+\delta_N)}}\sinc(k\pi F(-\lambda_n))(-1)^ke^{2ikt\lambda_n}\right|
		\leq C(\delta)\varepsilon^{r-c}+C\delta.
\end{equation*}
\end{prop}

In order to take into account the time evolution, we recall that the phase constants satisfy $\theta_n(u^\varepsilon(t);\varepsilon)=\theta_n(u_0;\varepsilon)+\omega_n(u_0;\varepsilon)t$ with
\begin{equation}\label{eq:wn}
\omega_{n+1}(u_0;\varepsilon)-\omega_{n}(u_0;\varepsilon)
	=2\lambda_n(u_0;\varepsilon)+\varepsilon,
\end{equation}
see for instance the proof of Theorem~3.14 in~\cite{bo_zero}.

\begin{proof}[Proof of Proposition~\ref{prop:M_ACV}]
Thanks to Corollary~\ref{thm:lax_u}, the proof of Theorem 3.9 in~\cite{bo_zero} is still valid as long as we do not use the approximation on the phase constants but only the approximation of the Lax eigenvalues. Since the Lax eigenvalues are constant over time, the estimates are also valid for any $t\in\R$. Hence inequality~(17) in~\cite{bo_zero} still holds:
\[
\left|\widehat{u^{\varepsilon}(t)}(k)
	-\varepsilon\sum_{\substack{n_{k+1}=n_1\\|n_{i+1}-n_i|\leq\varepsilon^{-r}\\ \lambda_{n_1}+\varepsilon\in\Lambda_-(\delta)}} \sinc(\pi F(-\lambda_{n_1}))^k\prod_{i=1}^k\frac{F(-\lambda_{n_1})e^{i(\theta_{n_i+1}-\theta_{n_i}+(\omega_{n_i+1}-\omega_{n_i})t)}}{n_{i+1}-n_i-F(-\lambda_{n_1})}\right|
		\leq C(\delta)\varepsilon^{r-c}+C\delta.
\]
Thanks to~\eqref{eq:wn} and the fact that $|\lambda_p-\lambda_n|\leq C(\delta)\varepsilon^{1-r}$ when $|n-p|\leq\varepsilon^{-r}$, we can replace $\omega_{n_i+1}-\omega_{n_i}$ by $\lambda_{n_1}+\varepsilon$. We get that for every $\varepsilon<\varepsilon_0(\delta,T)$, for every $t\in[0,T]$,
\[
\left|\widehat{u^{\varepsilon}(t)}(k)
	-\varepsilon\sum_{\substack{n_{k+1}=n_1\\|n_{i+1}-n_i|\leq\varepsilon^{-r}\\ \lambda_{n_1}+\varepsilon\in\Lambda_-(\delta)}} \sinc(\pi F(-\lambda_{n_1}))^ke^{2ik\lambda_{n_1}t}\prod_{i=1}^k\frac{F(-\lambda_{n_1})e^{i(\theta_{n_i+1}-\theta_{n_i})}}{n_{i+1}-n_i-F(-\lambda_{n_1})}\right|
		\leq C(\delta)\varepsilon^{r-c}+C\delta.
\]
We now want to use the fact that the approximation $e^{i(\theta_{n_i+1}-\theta_{n_i})}\approx -1$ is valid on a large set of indices. Thanks to the Toeplitz absolute convergence from Lemma 3.12 in~\cite{bo_zero}, there holds the upper bound
\[
\varepsilon\sum_{\substack{n_{k+1}=n_1\\|n_{i+1}-n_i|\leq\varepsilon^{-r}\\  n_1\in J}} \sinc(\pi F(-\lambda_{n_1}))^k\prod_{i=1}^k\frac{F(-\lambda_{n_1})}{|n_{i+1}-n_i-F(-\lambda_{n_1})|}
	\leq C\varepsilon\sum_{ n_1\in J} \sinc(\pi F(-\lambda_{n_1}))^k.
\]
Using that $\sinc(\pi F(-\lambda_{n_1}))^k\leq \sinc(\pi F(-\lambda_{n_1}))$ since $\sinc(\pi F(-\lambda_{n_1}))\in[0,1]$, Proposition~\ref{prop:phase-zero} implies
\[
\varepsilon\sum_{\substack{n_{k+1}=n_1\\|n_{i+1}-n_i|\leq\varepsilon^{-r}\\ n_1\in J}} \sinc(\pi F(-\lambda_{n_1}))^k\prod_{i=1}^k\frac{F(-\lambda_{n_1})}{|n_{i+1}-n_i-F(-\lambda_{n_1})|}
	\leq C(\delta)\varepsilon^{r-c}+C\delta.
\]
As a consequence, the contribution of the ``upside down'' terms such that $n_i\in J$ for some index $1\leq i\leq k$ is bounded by $C(\delta)\varepsilon^{r-c}+C\delta$, so that
\[
\left|\widehat{u^{\varepsilon}(t)}(k)
	-\varepsilon\sum_{\substack{n_{k+1}=n_1\\|n_{i+1}-n_i|\leq\varepsilon^{-r}\\ \lambda_{n_1}+\varepsilon\in\Lambda_-(\delta)}} \sinc(\pi F(-\lambda_{n_1}))^ke^{2ik\lambda_{n_1}t}(-1)^k\prod_{i=1}^k\frac{F(-\lambda_{n_1})}{n_{i+1}-n_i-F(-\lambda_{n_1})}\right|
		\leq C(\delta)\varepsilon^{r-c}+C\delta.
\]
We then use the Toeplitz identity from Lemma 3.12 in~\cite{bo_zero} and get
\[
\left|\widehat{u^{\varepsilon}(t)}(k)
	-\varepsilon\sum_{\substack{n_1\geq 1\\ n_1\in\Lambda_-(\delta+\delta_N)}} \sinc(k\pi F(-\lambda_{n_1}))^ke^{2ik\lambda_{n_1}t}(-1)^k\right|
		\leq C(\delta)\varepsilon^{r-c}+R(\varepsilon)+C\delta.\qedhere
\]
\end{proof}

We now proceed to the proof of our main theorem.
\begin{proof}[Proof of Theorem~\ref{thm:zero}]
We follow the proof of Theorem 3.14 in~\cite{bo_zero}, where we transform the sum involved in the statement of Proposition~\ref{prop:M_ACV} into a Riemann sum. We deduce that for every $\varepsilon<\varepsilon_0(\delta,T)$,
\[
\left|\widehat{u^{\varepsilon}(t)}(k)
	+\frac{i}{2k\pi}\int_{\min(u_0)}^{\max(u_0)} e^{-ik(x_+(\eta)+2\eta t)}-e^{-ik(x_-(\eta)+2\eta t)}\dd \eta\right|
		\leq C\delta.
\]
To conclude, using Proposition~3.15 in~\cite{bo_zero}, we have proven that
\[
\limsup_{\varepsilon\to 0}\left|\widehat{u^{\varepsilon}(t)}(k)
	+\frac{i}{2k\pi}\int_{\min(u_0)}^{\max(u_0)} e^{-ik(x_+(\eta)+2\eta t)}-e^{-ik(x_-(\eta)+2\eta t)}\dd \eta\right|
		\leq C\delta.
\]
Since $\delta$ is arbitrary, we obtain 
\begin{equation}\label{eq:uk}
\lim_{\varepsilon\to 0}\widehat{u^{\varepsilon}(t)}(k)
	=-\frac{i}{2k\pi}\int_{\min(u_0)}^{\max(u_0)} e^{-ik(x_+(\eta)+2\eta t)}-e^{-ik(x_-(\eta)+2\eta t)}\dd \eta.
\end{equation}
When $u_0$ is bell shaped, we conclude from Proposition~3.15 in~\cite{bo_zero} that $\widehat{u^{\varepsilon}(t)}(k)\to \widehat{u_{alt}^B(t)}(k)$ as $\varepsilon\to 0$. 
\end{proof}

%\newpage
%%%%%%%%%%%%%%%%%%%%%%%%%%%%%%%%%%%%%%%%%%%%%%%%%%%%%%%%%%%%%%%%%%%%%%%%%%%%%%%%%%%%%%%%%%%%%%%%%%%%%%%%%%%%%%%%%%%%%%%%%%%%%%%%%%%%%%%%%%%%%%%%%
\section{Extension to bell shaped initial data}\label{sec:bell_shaped}
%%%%%%%%%%%%%%%%%%%%%%%%%%%%%%%%%%%%%%%%%%%%%%%%%%%%%%%%%%%%%%%%%%%%%%%%%%%%%%%%%%%%%%%%%%%%%%%%%%%%%%%%%%%%%%%%%%%%%%%%%%%%%%%%%%%%%%%%%%%%%%%%%

To extend the study of the zero-dispersion limit from bell shaped trigonometric polynomials to general bell shaped initial data, we rely on an inversion formula from~\cite{Gerard2022explicit} for which one can pass to the limit both in the trigonometric approximation and in the dispersion parameter $\varepsilon\to 0$.

Let $u_0\in L^{\infty}(\T)\cap L^2(\T)$ with zero mean. The solution $u^{\varepsilon}$ to~\eqref{eq:bo} with initial data $u_0$ is well-defined~\cite{Molinet2008,MolinetPilod2012CauchyBO}. We decompose $u^{\varepsilon}=\Pi(u^\varepsilon)+\overline{\Pi(u^{\varepsilon})}$ where $\Pi$ is the Szeg\H{o} projector. Then according to Theorem~3 in~\cite{Gerard2022explicit}, the holomorphic extension of $\Pi(u^{\varepsilon})$ to $\D$ satisfies: for every $z\in\D$,
\begin{equation}\label{eq:fla_bo}
\Pi(u^{\varepsilon}(t,z))=\left\langle \left(\operatorname{Id}-ze^{i\varepsilon t}e^{2itL_{u_0}(\varepsilon)}S^*\right)^{-1}\Pi(u_0)\mid 1\right\rangle.
\end{equation}

Using this formula, we first establish the existence of a weak limit for $u^{\varepsilon}$ as $\varepsilon\to 0$. Then we show that this weak limit is $u^{B}_{alt}$ when $u_0$ is bell shaped, using the fact that this property is true for trigonometric polynomials.

We recall that
\[
L_{u_0}(\varepsilon)=\varepsilon D-T_{u_0},
\]
where $D=-i\partial_x$ and $T_{u_0}:h\in L^2_+\to \Pi(u_0h)$ is the Toeplitz operator of symbol $u_0$.

Note that when $u_0\in L^{\infty}(\T)\cap L^2(\T)$, we have that $T_{u_0}$ is a bounded self-adjoint operator in~$\mathcal{L}(L^2_+)$. We deduce that the operator $e^{it}e^{2itL_{u_0}(\varepsilon)}S^*$ has norm at most $1$ for every $\varepsilon\geq 0$. Hence for every $z\in\R$, formula~\eqref{eq:fla_bo} is continuous from $L^{\infty}(\T)\cap L^2(\T)$ to $\C$.

\begin{prop}[Existence of a weak limit]\label{prop:weaklim}
Let $u_0\in L^{\infty}(\T)\cap L^2(\T)$. Define $\widetilde{u}:=\Pi(\widetilde{u})+\overline{\Pi(\widetilde{u})}$ by the formula
\[
\Pi(\widetilde{u}(t,z))=\left\langle \left(\operatorname{Id}-z e^{-2itT_{u_0}}S^*\right)^{-1}\Pi(u_0)\mid 1\right\rangle.
\] 
 The solution $u^{\varepsilon}$ to~\eqref{eq:bo} with initial data $u_0$ is weakly convergent to $\widetilde{u}$ as $\varepsilon\to 0$: uniformly on finite time intervals
\[
u^{\varepsilon}\weaklim{\varepsilon\to 0}{} \widetilde{u}.
\]
\end{prop}

\begin{proof}
But for every $\varepsilon>0$, we know from the conservation of mass that $\|\Pi u^{\varepsilon}(t)\|_{L^2_+}=\|\Pi u_0\|_{L^2_+}$ is bounded independently of $\varepsilon$. Using the sequential version of Banach-Alaoglu's theorem on the Hilbert space $L^2_+(\T)$, we deduce that there is a sequence $\varepsilon_n$ and $f\in L^2_+$ such that $\Pi u^{\varepsilon}(t)\weaklim{\varepsilon\to 0}{} f$ in $L^2_+(\T)$.

Using formula~\eqref{eq:fla_bo}, we know that when $|z|<1$ there holds as $\varepsilon\to 0$ the pointwise convergence $\Pi u^{\varepsilon}(t,z) \to \Pi\widetilde{u}(t,z)$.
As a consequence, the unique possible accumulation point is $f=\Pi \widetilde{u}(t)$. This implies that $\Pi u^{\varepsilon}(t) \weaklim{\varepsilon\to 0}{} \Pi\widetilde{u}(t)$ in $L^2_+$.
%Let $r<1$. We know that for every $\varepsilon\geq 0$, for every $|z|\leq r$,
%\[
%\vertiii{\left(\operatorname{Id}-ze^{it}e^{2itL_{u_0}(\varepsilon)}S^*\right)^{-1}}
%	\leq \frac{1}{1-r}
%\]
%hence $|\Pi u^{\varepsilon}(t,z)|\leq C(r)$ and  $|\Pi\widetilde{u}(t,z)|\leq C(r)$.
%Moreover, for every $|z|\leq r$, as $\varepsilon\to 0$, there holds $\Pi u^{\varepsilon}(t,z) \to \Pi\widetilde{u}(t,z)$. In particular, for $r<1$, we know that $\Pi u^{\varepsilon}(t) \weaklim{\varepsilon\to 0}{} \Pi\widetilde{u}(t)$ on $L^2(\partial B(0,r))$ hence
%\[
%\|\Pi\widetilde{u}(t)\|_{L^2(\partial B(0,r))}
%	\leq \liminf_{\varepsilon\to 0}\|\Pi u^{\varepsilon}(t)\|_{L^2(\partial B(0,r))}
%	\leq \liminf_{\varepsilon\to 0}\|\Pi u^{\varepsilon}(t)\|_{L^2_+}.
%\]
%But for every $\varepsilon>0$, we know from the conservation of mass that $\|\Pi u^{\varepsilon}(t)\|_{L^2_+}=\|\Pi u_0\|_{L^2_+}$ is bounded. Hence $\|\Pi\widetilde{u}(t)\|_{L^2_+}<+\infty$.

Finally, since $\|\Pi u^{\varepsilon}(t)\|_{L^2_+}$ is bounded for $t\in\R$ and $\varepsilon>0$ thanks to the conservation of mass, one can see that the weak convergence is actually uniform on finite time intervals. 
\end{proof}

\begin{rk}
One may be tempted to use this formula directly for the approximate initial data~$u_0^{\varepsilon}$ from~\cite{bo_zero} without relying on the present work. However, one cannot easily ensure that the chosen approximate initial data $(u_0^{\varepsilon})_{\varepsilon}$ are uniformly bounded in $L^{\infty}$, whereas the proof of Proposition~\ref{prop:weaklim} relies on the continuity of formula~\eqref{eq:fla_bo} on $L^{\infty}\cap L^2$.
\end{rk}

\begin{prop}[Formula for the weak limit]
Let $u_0$ be a bell shaped initial data. Then the function $\widetilde{u}$ from Proposition~\ref{prop:weaklim} satisfies
\[
\widetilde{u}=u^B_{alt}.
\]
\end{prop}

\begin{proof}
Let $u_0$ be an even bell shaped initial data. Let $u_{0,N}$ be the truncated Fourier series of $u_0$ up to order $N$. Then according to Proposition~\ref{prop:comonotone},  $u_{0,N}$ is increasing on $(x_{\min}(N),x_{\max}(N))$ and decreasing on $(x_{\max}(N),x_{\min}(N)+2\pi)$ where $x_{\min}(N)\to 0$ and $x_{\max}(N)\to \pi$. We denote by $u_N^{\varepsilon}$ the solution to~\eqref{eq:bo} with initial data $u_{0,N}$.
Using Proposition~\ref{prop:weaklim}, we denote by $\widetilde{u}_N$ the weak limit of $u_N^{\varepsilon}$ as $\varepsilon\to 0$, and by $\widetilde{u}$ the weak limit of $u^{\varepsilon}$.
 
We know from Section~\ref{sec:zero} that if $u_{0,N}$ is bell shaped, then $\widetilde{u}_N$ is the signed sum of Branches for the multivalued solution of Burgers' equation obtained using the method of characteristics: $\widetilde{u}_N=(u_N)^B_{alt}$. Since this assumption is not entirely true here, we instead use the formula for the limit of Fourier coefficients given by equality~\eqref{eq:uk}: denoting $x_{N,\pm}(\eta)$ the antecedents of $\eta$ by $u_{0,N}$, for every~$k\in\Z$ there holds
\begin{equation*}
\lim_{\varepsilon\to 0}\widehat{u_N^{\varepsilon}(t)}(k)
	=-\frac{i}{2k\pi}\int_{\min(u_{0,N})}^{\max(u_{0,N})} e^{-ik(x_{N,+}(\eta)+2\eta t)}-e^{-ik(x_{N,-}(\eta)+2\eta t)}\dd \eta.
\end{equation*}
This implies that
\begin{equation*}
\widehat{\widetilde{u}_N(t)}(k)
	=-\frac{i}{2k\pi}\int_{\min(u_{0,N})}^{\max(u_{0,N})} e^{-ik(x^N_{+}(\eta)+2\eta t)}-e^{-ik(x^N_{-}(\eta)+2\eta t)}\dd \eta.
\end{equation*}

Finally, according to Proposition~\ref{prop:inflection}, one can pass to the limit $N\to +\infty$ in the right hand side to get that
\begin{equation*}
\widehat{\widetilde{u}_N(t)}(k)
	\longrightarroww{N\to+\infty}{}
	-\frac{i}{2k\pi}\int_{\min(u_{0})}^{\max(u_{0})} e^{-ik(x_{+}(\eta)+2\eta t)}-e^{-ik(x_{-}(\eta)+2\eta t)}\dd \eta.
\end{equation*}
Proposition~3.15 in~\cite{bo_zero}, applied to the bell shaped initial data $u_0$, implies that the right-hand side is actually equal to $u^B_{alt}(t)$.

We now make use of  the continuity of formula~\eqref{eq:fla_bo} on $L^{\infty}(\T)\cap L^2(\T)$ for every $\varepsilon\geq 0$ and $t\in\R$. As a consequence, as $N\to +\infty$, we have that
\[
\widetilde{u}_N(t) \weaklim{N\to+\infty}{} \widetilde{u}(t).
\]
This implies that for every $k\in\Z$,
\begin{equation*}
\widehat{\widetilde{u}(t)}(k)=u^B_{alt}(t).
\end{equation*}
We conclude that $\widetilde{u}(t)=u^B_{alt}(t)$ for every $t\in\R$.
\end{proof}

\appendix
%\newpage
%%%%%%%%%%%%%%%%%%%%%%%%%%%%%%%%%%%%%%%%%%%%%%%%%%%%%%%%%%%%%%%%%%%%%%%%%%%%%%%%%%%%%%%%%%%%%%%%%%%%%%%%%%%%%%%%%%%%%%%%%%%%%%%%%%%%%%%%%%%%%%%%%
\section{Appendices}
%%%%%%%%%%%%%%%%%%%%%%%%%%%%%%%%%%%%%%%%%%%%%%%%%%%%%%%%%%%%%%%%%%%%%%%%%%%%%%%%%%%%%%%%%%%%%%%%%%%%%%%%%%%%%%%%%%%%%%%%%%%%%%%%%%%%%%%%%%%%%%%%%

%\newpage
%%%%%%%%%%%%%%%%%%%%%%%%%%%%%%%%%%%%%%%%%%%%%%%%%%%%%%%%%%%%%%%%%%%%%%%%%%%%%%%%%%%%%%%%%%%%%%%%%%%%%%%%%%%%%%%%%%%%%%%%%%%%%%%%%%%%%%%%%%%%%%%%%
\subsection{Phase constants of even initial data}\label{appendix:phase}
%%%%%%%%%%%%%%%%%%%%%%%%%%%%%%%%%%%%%%%%%%%%%%%%%%%%%%%%%%%%%%%%%%%%%%%%%%%%%%%%%%%%%%%%%%%%%%%%%%%%%%%%%%%%%%%%%%%%%%%%%%%%%%%%%%%%%%%%%%%%%%%%%

In this part, we prove that if $u$ is even, then for every $n\in\N$ such that $\zeta_n(u)\neq 0$, there holds that $\overline{\zeta_n(u;\varepsilon)}\zeta_{n+1}(u;\varepsilon)\in\R$. %By continuity, it is enough to establish this fact when $u$ is a trigonometric polynomial of order $N$.

\begin{thm}[Phase constants for an even function]\label{thm:phase}
Let $u$ be an even bell shaped potential. Then for every $n\in\N$, $\overline{\zeta_n(u;\varepsilon)}\zeta_{n+1}(u;\varepsilon)\in\R$.
\end{thm}

We show this result for even bell shaped trigonometric polynomials of degree $N$. Indeed, for general even bell shaped function $u$, one can use the truncated Fourier series $u_N$ of degree $N$, that satisfies Proposition~\ref{prop:comonotone}. The theorem applied to $u_N$ implies that $\overline{\zeta_n(u_N;\varepsilon)}\zeta_{n+1}(u_N;\varepsilon)\in\R$ for every $n$. Moreover, this term goes to $\overline{\zeta_n(u;\varepsilon)}\zeta_{n+1}(u;\varepsilon)\in\R$ as $N\to+\infty$ by continuity of the Birkhoff map.

We have seen that the trigonometric polynomial $u_N$ extends to a meromorphic function on the unit disk $\overline{\D}$ as
\[
u_N(z)=\sum_{k=-N}^Nc_kz^{k}.
\]
Since $u$ is even, then for every $-N\leq k\leq N$, $c_k\in\R$. In particular, the function $Q$ defined in~\eqref{def:Q} as $Q(z)=\sum_{k=1}^N\left(-c_k\frac{z^k}{k}+c_k\frac{z^{-k}}{k}\right)$ satisfies $Q(\overline{z})=\overline{Q(z)}$ for every $z\in\D$.

We first study the matrix $A$ involved  in the eigenvalue equation and introduced in Definition~\ref{def:A}: for $1\leq k\leq N-1$ and $1\leq \ell\leq N$, we have
\[
A_{k,\ell}:=\int_{\Gamma_k}e^{Q(\zeta)/\varepsilon}\zeta^{-\ell-\lambda/\varepsilon}\frac{\dd\zeta}{\zeta},
\]
\[
A_{N,\ell}^{\pm}:=\int_{\Gamma_N^{\pm}}e^{Q(\zeta)/\varepsilon}\zeta^{-\ell-\lambda/\varepsilon}\frac{\dd\zeta}{\zeta},
\]
\[
A_{N,\ell}:=A_{N,\ell}^++e^{-2i\pi\lambda/\varepsilon}A_{N,\ell}^-.
\]

\begin{lem}[Matrix coefficients]
For $1\leq k\leq N-1$ and $1\leq\ell\leq N$, there holds
\[
A_{k,\ell}=-\overline{A_{N-k,\ell}}.
\]
Moreover, for $1\leq\ell\leq N$, we have
\[
e^{i\pi\lambda/\varepsilon}A_{N,\ell}\in\R.
\]
\end{lem}

\begin{proof}
Let $1\leq k\leq N-1$ and $1\leq \ell\leq N$. We recall that $\Gamma_k$ is the union of the segment $[0,e^{i\theta_k}]$, the arc of circle of radius $1$ going from $e^{i\theta_k}$ to $e^{i\theta_{k+1}}$ and the segment $[e^{i\theta_{k+1}},0]$ for $\theta_k=\frac{2k-1}{N}\pi$ and $\theta_{k+1}=\frac{2k+1}{N}\pi$.

$\bullet$ First, we parameterize the segment $[0,e^{i\theta_k}]$ by $\zeta=te^{i\theta_k}$ for $t\in[0,1]$, so that
\[
\int_{[0,e^{i\theta_k}]}e^{Q(\zeta)/\varepsilon}\zeta^{-\ell-\lambda/\varepsilon}\frac{\dd\zeta}{\zeta}
	=\int_0^1e^{Q(te^{i\theta_k})/\varepsilon}(te^{i\theta_k})^{-\ell-\lambda/\varepsilon}\frac{\dd t}{t}.
\]
Taking the conjugate, we have
\[
\overline{\int_{[0,e^{i\theta_k}]}e^{Q(\zeta)/\varepsilon}\zeta^{-\ell-\lambda/\varepsilon}\frac{\dd\zeta}{\zeta}}
	=\int_0^1e^{Q(te^{-i\theta_k})/\varepsilon}(te^{-i\theta_k})^{-\ell-\lambda/\varepsilon}\frac{\dd t}{t}.
\]
But $-\theta_k=\frac{2N-2k+1}{N}\pi=\theta_{N-k+1}$ modulo $2\pi$, therefore
\[
\overline{\int_{[0,e^{i\theta_k}]}e^{Q(\zeta)/\varepsilon}\zeta^{-\ell-\lambda/\varepsilon}\frac{\dd\zeta}{\zeta}}
	=\int_{[0,e^{i\theta_{N-k+1}}]}e^{Q(\zeta)/\varepsilon}\zeta^{-\ell-\lambda/\varepsilon}\frac{\dd\zeta}{\zeta}.
\]
In particular
\[
\overline{\int_{[0,e^{i\theta_k}]}e^{Q(\zeta)/\varepsilon}\zeta^{-\ell-\lambda/\varepsilon}\frac{\dd\zeta}{\zeta}}
	=-\int_{[e^{i\theta_{N-k+1}},0]}e^{Q(\zeta)/\varepsilon}\zeta^{-\ell-\lambda/\varepsilon}\frac{\dd\zeta}{\zeta}.
\]
The same relationship holds by considering $[0,e^{i\theta_{k+1}}]$ and $[0,e^{i\theta_{N-k}}]$.

$\bullet$ Now, we consider the integral on the arc of circle $\gamma(\theta_k,\theta_{k+1})$ of radius $1$ going from $e^{i\theta_k}$ to~$e^{i\theta_{k+1}}$. We parameterize $\zeta=e^{i\theta}$ for $\theta\in[\theta_k,\theta_{k+1}]$, so that
\[
\int_{\gamma(\theta_k,\theta_{k+1})}e^{Q(\zeta)/\varepsilon}\zeta^{-\ell-\lambda/\varepsilon}\frac{\dd\zeta}{\zeta}
	=i\int_{\theta_k}^{\theta_{k+1}}e^{Q(e^{i\theta})/\varepsilon}(e^{i\theta})^{-\ell-\lambda/\varepsilon}\dd \theta.
\]
Taking the conjugate we get
\[
\overline{\int_{\gamma(\theta_k,\theta_{k+1})}e^{Q(\zeta)/\varepsilon}\zeta^{-\ell-\lambda/\varepsilon}\frac{\dd\zeta}{\zeta}}
	=-i\int_{\theta_k}^{\theta_{k+1}}e^{Q(e^{-i\theta})/\varepsilon}(e^{-i\theta})^{-\ell-\lambda/\varepsilon}\dd \theta.
\]
We make the change of variable $\alpha=-\theta$ and get
\[
\overline{\int_{\gamma(\theta_k,\theta_{k+1})}e^{Q(\zeta)/\varepsilon}\zeta^{-\ell-\lambda/\varepsilon}\frac{\dd\zeta}{\zeta}}
	=-i\int^{\theta_{N-k+1}}_{\theta_{N-k}}e^{Q(e^{i\alpha})/\varepsilon}(e^{-i\alpha})^{-\ell-\lambda/\varepsilon}\dd \alpha.
\]
As a consequence,
\[
\overline{\int_{\gamma(\theta_k,\theta_{k+1})}e^{Q(\zeta)/\varepsilon}\zeta^{-\ell-\lambda/\varepsilon}\frac{\dd\zeta}{\zeta}}
	=-\int_{\gamma(\theta_{N-k},\theta_{N-k+1})}e^{Q(\zeta)/\varepsilon}\zeta^{-\ell-\lambda/\varepsilon}\frac{\dd\zeta}{\zeta}.
\]
Similarly in the case $k=N$, we obtain that
\[
\overline{\int_{\gamma(\theta_N,0)}e^{Q(\zeta)/\varepsilon}\zeta^{-\ell-\lambda/\varepsilon}\frac{\dd\zeta}{\zeta}}
	=-\int_{\gamma(0,\theta_{1})}e^{Q(\zeta)/\varepsilon}\zeta^{-\ell-\lambda/\varepsilon}\frac{\dd\zeta}{\zeta}.\qedhere
\]
\end{proof}

\begin{lem}
Let $f\in L^2_+$ be an eigenvector of $L_u(\varepsilon)$ associated to the eigenvalue $\lambda$. Then $\langle f|Sf\rangle\in\R$.
\end{lem}

\begin{proof}
We recall that the vector $V=(v_1,\dots,v_N)$ is solution to $AV=0$.

Let $C_1^T,\dots, C_N^T$ be the lines of $A$. Then $e^{i\pi\lambda/\varepsilon}C_N^T$ is real valued whereas $C_{N-1-k}=-\overline{C_k}$ for every $1\leq k\leq N-1$. The fact that $AV=0$ implies that for every $1\leq k\leq N$,
\[
C_k^T V=0.
\]
In particular, one can see that if $1\leq k\leq N-1$,
\[
C_k^T\overline{V}=-\overline{C_{N-k}^T V}=0
\]
and also
\[
e^{i\pi\lambda/\varepsilon} C_N^T\overline{V}=\overline{e^{i\pi\lambda/\varepsilon} C_N^TV}=0.
\]
This implies that $\overline{V}$ also satisfies $A\overline{V}=0$. Since $\lambda$ is a simple eigenvalue (see~\cite{GerardKappeler2019}) and the coefficients of the eigenvectors solve a triangular system depending on the Taylor expansion of the eigenfunction at $0$, we deduce that there exists $\alpha\in\C$ such that $\overline{V}=\alpha V$. However this implies that $\overline{V}=|\alpha|^2\overline{V}$ therefore $|\alpha|=1$. Writing $\alpha=e^{2i\theta}$ for some $\theta\in\T$, we get that $\overline{e^{i\theta}V}=e^{i\theta} V$, so that $e^{i\theta}V$ is a real-valued vector.

We recall the equation~\eqref{eq:f} satisfied by $f$:
\begin{equation*}
zf'(z)-\left(\sum_{k=1}^Nc_kz^k+\overline{c_k}z^{-k}\right)f(z)-\lambda f(z)=\sum_{j=1}^N v_jz^{-j}.
\end{equation*}
We note that, up to replacing $f$ by $e^{-i\theta}f$, which transforms $V$ into $e^{-i\theta}V$, one can assume that $V$ is real-valued.

Now, we check that $z\mapsto \overline{f(\overline{z})}$ is a holomorphic function on $\D$ also satisfying~\eqref{eq:f}. Since $\lambda$ is a simple eigenvalue, this function is colinear to $f$, so that there exists $\alpha\in\C$ such that for every $z\in\D$, we have $\overline{f(\overline{z})}=\alpha f(z)$.
We write the Taylor expansion of $f$ around zero as $f(z)=\sum_{\ell\geq 0}b_{\ell}z^{\ell}$ for some $b_{\ell}\in\C$. Then $\overline{f(\overline{z})}=\sum_{\ell\geq 0}\overline{b_{\ell}}z^{\ell}$. Therefore the number $\alpha\in\C$ is so that for every $\ell\geq 0$, $\overline{b_{\ell}}=\alpha b_{\ell}$. Since $f$ is nonzero, there exists $\ell$ such that $b_{\ell}\neq 0$. We deduce that $|\alpha|=1$. Writing $\alpha=e^{2i\theta}$ for some $\theta\in\T$, we deduce that $\overline{e^{i\theta} b_{\ell}}=e^{i\theta} b_{\ell}$ and therefore $e^{i\theta}b_{\ell}\in\R$ for every $\ell\geq 0$. As a consequence, the parametrization $z=e^{ix}$ leads to
\[
\langle f|Sf\rangle=\int_{0}^{2\pi}|f(e^{ix})|^2e^{-ix}\dd x=-i\int_{\partial\D}\sum_{k,\ell\geq 0}b_k\overline{b_{\ell}}z^{k-\ell-2}\dd z.
\]
The residue formula implies $k-\ell-2=-1$, hence
\[
\langle f|Sf\rangle={2\pi}\sum_{\ell\geq 0}b_{\ell+1}\overline{b_{\ell}}.
\]
We conclude that $\langle f|Sf\rangle\in\R$ by noting that $b_{\ell+1}\overline{b_{\ell}}=e^{i\theta}b_{\ell+1}\overline{e^{i\theta} b_{\ell}}\in\R$ for every $\ell\geq 0$.
\end{proof}

\begin{proof}[Proof of Theorem~\ref{thm:phase}]
We use the fact that 
\[
\langle f_n(u;\varepsilon)|Sf_n(u;\varepsilon)\rangle
	=-\frac{1}{\varepsilon}\sqrt{\mu_{n+1}(u;\varepsilon)}\frac{\sqrt{\kappa_n(u;\varepsilon)}}{\sqrt{\kappa_{n+1}(u;\varepsilon)}}\overline{\zeta_n(u;\varepsilon)}\zeta_{n+1}(u;\varepsilon),
\]
see for instance equality~(7.8) in~\cite{GerardKappeler2019}. As a consequence, since the term $\langle f_n(u;\varepsilon)|Sf_n(u;\varepsilon)\rangle$ is real, so is $\overline{\zeta_n(u;\varepsilon)}\zeta_{n+1}(u;\varepsilon)$.
\end{proof}

%\newpage
%%%%%%%%%%%%%%%%%%%%%%%%%%%%%%%%%%%%%%%%%%%%%%%%%%%%%%%%%%%%%%%%%%%%%%%%%%%%%%%%%%%%%%%%%%%%%%%%%%%%%%%%%%%%%%%%%%%%%%%%%%%%%%%%%%%%%%%%%%%%%%%%%
\subsection{Approximation by  truncated Fourier series}\label{sec:trigonometric_approx}
%%%%%%%%%%%%%%%%%%%%%%%%%%%%%%%%%%%%%%%%%%%%%%%%%%%%%%%%%%%%%%%%%%%%%%%%%%%%%%%%%%%%%%%%%%%%%%%%%%%%%%%%%%%%%%%%%%%%%%%%%%%%%%%%%%%%%%%%%%%%%%%%%

In this part, we fix a bell shaped function $u$. Our goal is to show that the approximation of $u$ by its truncated Fourier series $u_N$ is comonotone so that it has weakened properties of a bell shaped function. We deduce in particular Corollary~\ref{thm:lax_u}.

%In this part, we consider a general bell shaped initial data and establish the rate of convergence for the comonotone approximation by truncated Fourier series.%trigonometric polynomials in order to tackle the zero dispersion limit problem for~\eqref{eq:bo}.

More precisely, we use the following result.
\begin{prop}[Comonotone approximation]\label{prop:comonotone}
Let  $u$ be a bell shaped function in $\classeC^3(\T)$. We decompose $u$ into a Fourier series
\[
u(x)=\sum_{k=-\infty}^{\infty}c_ke^{ikx}.
\]
Then there exists $N_0\geq 1$ such that for every $N\geq N_0$, the truncated Fourier series of order $N$
\[
u_N(x):=\sum_{k=-N}^{N}c_ke^{ikx}
\]
satisfies the following properties. There exists a sequence $x_{\min}(N)\to 0$ and $x_{\max}(N)\to x_{\max}$ as $N\to+\infty$ such that the $2\pi$-periodic function $u_N$ is increasing on $(x_{\min}(N),x_{\max}(N))$ and decreasing on $(x_{\max}(N),x_{\min}(N)+2\pi)$, moreover,
\begin{equation}\label{eq:distance_approxN}
\|u-u_N\|_{L^{\infty}(\T)}\leq \frac{C}{N\sqrt{N}}.
\end{equation}
\end{prop}

Consequently, the trigonometric approximation is also bell shaped up to spatial translation and up to removing the mean value.

\begin{proof}
Let $u$ be a bell shaped initial data. Then we estimate
\begin{align*}
\|u''-u_N''\|_{L^{\infty}(\T)}
	&\leq \sum_{|k|\geq N+1}k(k-1)|c_k|\\
	&\leq \left(\sum_{|k|\geq N+1}k^6|c_k|^2\right)^{1/2}\left(\sum_{|k|\geq N+1}\frac{1}{k^2}\right)^{1/2}\\
	&\leq C\frac{\|u\|_{H^3(\T)}}{\sqrt{N}}.
\end{align*}
By assumption, there are only two inflection points $\xi_{\pm}$ such that $u''(\xi_{\pm})=0$. 
%Since the inflection points are simple $u'''(\xi_{\pm})\neq 0$, there exist $\alpha_1,\alpha_2>0$ such that if $|x-\xi_{\pm}|\geq \alpha_1$, then $|u''(x)|\geq \alpha_2|x-\xi_{\pm}|$.
Let $\alpha_1>0$ sufficiently small. We remove a box of size $\alpha_2$ around $\xi_{\pm}$ so that if $x\in [0,2\pi]\setminus([\xi_--\alpha_2,\xi_-+\alpha_2]\cup[\xi_+-\alpha_2,\xi_++\alpha_2])$, then $|u''(x)|\geq \alpha_1$. When $ C\frac{\|u\|_{\classeC^3(\T)}}{\sqrt{N}}\leq \alpha_1$, we deduce that $u_N''$ has the same sign as $u''$ on this set.

As a consequence, $u_N'$ is strictly increasing on $[\xi_++\alpha_2-2\pi,\xi_--\alpha_2]$ and strictly decreasing on $[\xi_-+\alpha_2,\xi_+-\alpha_2]$.
By choosing $\alpha_2$ small enough, we also get that $u_N'>0$ on $[\xi_--\alpha_2,\xi_-+\alpha_2]$ and $u_N'<0$ on $[\xi_+-\alpha_2,\xi_++\alpha_2]$.

We deduce that there exist unique $x_{\min}(N),x_{\max}(N)\in[0,2\pi]$ such that $u_N$ is increasing on $(x_{\min}(N),x_{\max}(N))$ and decreasing on $(x_{\max}(N),x_{\min}(N)+2\pi)$. Finally, an application of the Cauchy-Schwarz' inequality as in the beginning of this proof leads to~\eqref{eq:distance_approxN}. This implies that $x_{\min}(N)\to 0$ and $x_{\max}(N)\to x_{\max}$ as $N\to+\infty$.
\end{proof}

%\newpage

\paragraph{Proof of Corollary~\ref{thm:lax_u}}
From the min-max formula
\[
\lambda_n(u;\varepsilon)=\max_{\dim F=n}\min\{\langle L_u(\varepsilon)h|h\rangle \mid h\in H^1_+\cap F^{\perp},\;\|h\|_{L^2}=1\},
\]
one can see that for every $u,v\in \classeC^0(\T)$, there holds
\[
|\lambda_n(u;\varepsilon)-\lambda_n(v;\varepsilon)|\leq \|u-v\|_{L^{\infty}}.
\]
Using truncated Fourier series, see Proposition~\ref{prop:comonotone} below, we deduce that if $u\in\classeC^3(\T)$, then for $N\geq N_0$, we have 
\begin{equation*}%\label{eq:lambda_lipschitz}
|\lambda_n(u;\varepsilon)-\lambda_n(u_N;\varepsilon)|\leq C\frac{1}{N\sqrt{N}}
\end{equation*}
% We allow the constants $C$ to change from line to line, they may depend on the norm of $u_0$ in $\classeC^4(\T)$ but not on the parameters $\varepsilon$, $\delta$.
This approximation leads to the estimate
\begin{equation*}
\left|\int_{-\lambda_n(u;\varepsilon)}^{-\lambda_p(u;\varepsilon)}F(\eta)\dd\eta-(n-p)\varepsilon\right|
	\leq \frac{C}{N\sqrt{N}}+C_N(\delta)\varepsilon\sqrt{\varepsilon}.
\end{equation*}
This also leads to point 1.\@ in Corollary~\ref{thm:lax_u}.
However, this estimate is not sufficient compared to Definition 1.5 in~\cite{bo_zero} for an admissible approximate initial data, as the necessary condition is
\begin{equation}\label{eq:def1.5}
\left|\int_{-\lambda_n(u;\varepsilon)}^{-\lambda_p(u;\varepsilon)}F(\eta)\dd\eta-(n-p)\varepsilon\right|
	\leq C_u(\delta)\varepsilon \sqrt{\varepsilon}.
\end{equation}
This condition is especially important in the proof of Lemma 3.10 in~\cite{bo_zero}.
%It turns out in the proof that the factor $\sqrt{\varepsilon}$ is stronger than we need, hence we could replace it by anything of the form $R(\varepsilon)\to 0$ as $\varepsilon\to 0$.
 One could study the growth of $C_N(\delta)$ with $N$ to show that if $u$ has very fast decaying Fourier coefficients as $|\widehat{u}(n)|\leq Cn^{-Cn^2}$, then the growth of $C_N(\delta)$ is sufficiently slow so that for $N=N(\varepsilon)$ well chosen, for every $n,p\in\Lambda_-(\delta)$,
 \[
  \frac{C}{N(\varepsilon)\sqrt{N(\varepsilon)}}+C_{N(\varepsilon)}(\delta)\varepsilon\sqrt{\varepsilon}\leq C_u(\delta)\varepsilon\sqrt{\varepsilon}.
 \]
As a consequence~\eqref{eq:def1.5} still holds. However, we have seen that the approach in Section~\ref{sec:bell_shaped} does not require such a decay of the Fourier coefficients.
%\begin{equation*}
%\left|\int_{-\lambda_n(u)}^{-\lambda_p(u)}F(\eta)\dd\eta-(n-p)\varepsilon\right|
%	\leq C_u(\delta)\varepsilon R(\varepsilon).
%\end{equation*}

%In this purpose, we show that it is possible to move the critical points $p_k(\lambda)$ a little (hence to replace $u_N$ by an other trigonometric polynomial $v_N$ of degree $N$) so that these critical points are far enough from the origin, far enough from $\partial\D$, and far enough from each other. Then we estimate more precisely the remainder terms involved in the asymptotic expansions of the Lax eigenvalues in~\eqref{eq:small-eigenvalues} and~\eqref{eq:large-eigenvalues}.

We  can now turn to very large the Lax eigenvalues for bell shaped initial data to get point~2.\@ of Corollary~\ref{thm:lax_u}. Let $u_0$ be a bell shaped initial data, fix $N\geq N_0$ and let $u_N$ be the truncated Fourier series of $u$ up to order $N$. 
%We use the expansion of Lax eigenvalues of $u_N$ to deduce the weak limit of the solution to~\eqref{eq:bo} with initial data $u_N$ as $\varepsilon\to 0$, then we pass to the limit $N\to\infty$ to deduce information on the solution to~\eqref{eq:bo} with initial data $u$.

Fix $\delta>0$. Using the Parseval formula, for $K\in\N$, we have
\[
K\varepsilon\sum_{k=K}^{\infty}\gamma_k(u;\varepsilon)
	\leq \varepsilon\sum_{k=1}^{\infty}k\gamma_k(u;\varepsilon)=\|u\|_{L^2}^2/2.
\]
As a consequence, if $K(\delta)= \frac{C}{\delta}$, one has $\sum_{k=K(\delta)/\varepsilon}^{\infty}\gamma_k(u;\varepsilon)\leq \delta$, therefore for every $n\geq K(\delta)/\varepsilon$,
\[
|\lambda_n-n\varepsilon|\leq \delta.
\]
In particular one can show that for $n\geq \frac{K(\delta)}{\varepsilon}$ then $\lambda_n\geq K(\delta)-\delta$, whereas for $n=\frac{K(\delta)}{\varepsilon}$, then $\lambda_n\leq K(\delta)+\delta$ (and this property is also true for the smaller indices).

\paragraph{Estimate of $x_{\pm}$ in the trigonometric approximation}
Consider  $u\in\classeC^0(\T)$ weakly bell shaped, strictly increasing from $u_{\min}$ to $u_{\max}$ on $(x_{\min},x_{\max})$ and strictly decreasing on $(x_{\max},x_{\min}+2\pi)$. For $\eta\in(u_{\min},u_{\max})$, we define $x_+(\eta)$ as the antecedent of $\eta$ by $u$ on $(x_{\min},x_{\max})$ and by $x_-(\eta)$ as the antecedent of $\eta$ by $u$ on $(x_{\max},x_{\min}+2\pi)$. When considering  a sequence $(u_n)_n$ we denote $x^n_{\pm}$ the points $x_{\pm}$ associated to $u_n$.
\begin{prop}\label{prop:inflection}
Let $u\in\classeC^1(\T)$ be a weakly bell shaped function and $(u_n)_n\in\classeC^0(\T)$ a sequence of weakly bell shaped functions uniformly convergent to $u$. Then for every $\delta>0$,
\[
\|x^n_+(\eta)-x_+(\eta)\|_{L^{\infty}([u_{\min}+\delta,u_{\max}-\delta])} +\|x^n_-(\eta)-x_-(\eta)\|_{L^{\infty}([u_{\min}+\delta,u_{\max}-\delta])} \longrightarroww{n\to+\infty}{} 0.
\] 
\end{prop}

\begin{proof}
Let $\delta>0$. Since $\|u_n-u\|_{L^{\infty}}\to 0$, we know that there exists $n_0$ such that for every $n\geq n_0$, $\|u_n-u\|_{L^{\infty}}\leq \delta/2$. In particular $\eta$ has an antecedent by $u_n$ for every $\eta\in [u_{\min}+\delta,u_{\max}-\delta]$. Moreover, by strict monotonicity, $|u'|$ is bounded from below by $\frac{1}{C(\delta)}$ when $x$ is far from $x_{\min}$ and $x_{\max}$ in the sense that $u(x)\in  [u_{\min}+\delta,u_{\max}-\delta]$. We consider for instance the values of $x$ of the form $x=x_+(\eta)$ for some $\eta\in  [u_{\min}+\delta,u_{\max}-\delta]$

 Let $x$ such that $u(x)\in  [u_{\min}+\delta,u_{\max}-\delta]$, and $n\geq n_0$. We know that there exists $y_n:=x_+^n(u_n)$ satisfying $u(x)=u_n(y_n)$. Since $|u'|$ is bounded from below, we have that
\[
|u(x)-u(y_n)|\geq \frac{1}{C'(\delta)}|x-y_n|.
\]
Moreover, $u(x)=u_n(y_n)$ so the left hand side is
\[
|u(x)-u(y_n)|=|u_n(y_n)-u(y_n)|\leq\|u_n-u\|_{L^{\infty}}.
\]
We deduce that 
\[
|x-y_n|\leq C(\delta)\|u_n-u\|_{L^{\infty}}
\]
Let $\nu>0$. Then for $n\geq n_1$ large enough, we have that for every $x$ such that $u(x)\in  [u_{\min}+\delta,u_{\max}-\delta]$,  denoting $y_n\in\T$ the point at which $u(x)=u_n(y_n)$, there holds
\(
|x-y_n|\leq \nu.
\)
We deduce that for every $\eta\in [u_{\min}+\delta,u_{\max}-\delta]$,
\[
|x_+(\eta)-x_+^n(\eta)|\leq \nu.\qedhere
\]
\end{proof}

\Addresses
\end{document}